\documentclass[a4paper,12pt,leqno]{amsart} 
\usepackage{amsmath,amsthm,amssymb}  
\usepackage{amscd}
\usepackage[all]{xy} 
\usepackage{mathdots}
\usepackage{graphicx}

\numberwithin{equation}{section}

\theoremstyle{plain}
\newtheorem{theorem}{Theorem}[section]
\newtheorem{proposition}[theorem]{Proposition}         
\newtheorem{corollary}[theorem]{Corollary} 
\newtheorem{lemma}[theorem]{Lemma} 
\newtheorem{definition}[theorem]{Definition}  

\theoremstyle{definition}  
 
\newtheorem{remark}[theorem]{Remark} 

\newcommand{\C}{\mathbb C}   
\newcommand{\R}{\mathbb R}
\newcommand{\Z}{\mathbb Z}

\newcommand{\al}{\alpha}
 
\newcommand{\ga}{\gamma}
\newcommand{\de}{\delta}
\newcommand{\la}{\lambda}
\newcommand{\si}{\sigma} 
\newcommand{\La}{\Lambda}
\newcommand{\eps}{\epsilon}
\newcommand{\Om}{\Omega}
\newcommand{\De}{\Delta}

\newcommand{\om}{\omega}

\newcommand{\Ga}{\Gamma}
\newcommand{\ze}{\zeta}
\newcommand{\ka}{\kappa}

\DeclareMathOperator{\diag}{diag}

\DeclareMathOperator{\argu}{arg}
\DeclareMathOperator{\wt}{wt}

\newcommand{\SL}{\textrm{SL}}
\renewcommand{\sl}{\frak s\frak l}

\newcommand{\no}{\noindent}
      
\newcommand{\pr}{\prime} 
\newcommand{\prr}{{\prime\prime}} 
\newcommand{\st}{\ \vert\ }   
\renewcommand{\ll}{\lq\lq}
\newcommand{\rr}{\rq\rq\ }
\newcommand{\rrr}{\rq\rq}  
\renewcommand{\b}{\partial}
\newcommand{\bla}{\b_\la}
\newcommand{\bs}{\b_s}
\newcommand{\bz}{\b_z}
\newcommand{\bt}{\b_t}
\newcommand{\btbar}{\b_{\bar t}}
\newcommand{\na}{\nabla}
\newcommand{\bp}{\begin{pmatrix}} 
\newcommand{\ep}{\end{pmatrix}} 
\newcommand{\bsp}{\left(\begin{smallmatrix}} 
\newcommand{\esp}{\end{smallmatrix}\right)}

\newcommand{\zbar}{  {\bar z}  }
\newcommand{\zzb}{ {z\bar z}  }

\newcommand{\tbar}{  {\bar t}  }
\newcommand{\ttb}{ {t\bar t}  }

\renewcommand{\i}{ {\scriptscriptstyle\sqrt{-1}}\, }
\newcommand{\ii}{ {\scriptstyle\sqrt{-1}}\, }

\newcommand{\Psiz}{  \Psi^{(0)}  }
\newcommand{\Psii}{  \Psi^{(\infty)}  }

\newcommand{\Sz}{  S^{(0)}  }
\newcommand{\Si}{  S^{(\infty)}  }
\newcommand{\Qz}{  Q^{(0)}  }
\newcommand{\Qi}{  Q^{(\infty)}  }

\newcommand{\barQi}{  {\bar Q}^{(\infty)}  }
\newcommand{\Omz}{  \Om^{(0)}  }
\newcommand{\Omi}{  \Om^{(\infty)}  }

\newcommand{\Phiz}{\Phi^{(0)}}
\newcommand{\Phii}{\Phi^{(\infty)}}
\newcommand{\phiz}{\phi^{(0)}}
\newcommand{\phii}{\phi^{(\infty)}}
\newcommand{\nn}{m}
\newcommand{\NN}{M}
\newcommand{\MM}{\Lambda}

\newcommand{\eo}{\ell_\om}

\newcommand{\emodz}{\ell_{\vert z\vert}}

\newcommand{\Asharp}{\Ga}
\newcommand{\gazi}{\gamma}
\newcommand{\euler}{\gamma_{\scriptscriptstyle\text{eu}}}

\newcommand{\pga}{\vphantom{{A_A}_A}\ga}  
\newcommand{\pal}{\vphantom{{A_A}_A}\al}  

\begin{document}     

\title[Isomonodromy aspects III: asymptotics]{Isomonodromy aspects of the tt*
equations of Cecotti and Vafa 
\\
III.  Iwasawa factorization and asymptotics
}  

\author{Martin A. Guest, Alexander R. Its, and Chang-Shou Lin}      

\date{}   

\begin{abstract} 
This paper, the third in a series, completes our description of all (radial) solutions on $\C^\ast$ of the tt*-Toda equations
$2(w_i)_{\ttb}=-e^{2(w_{i+1}-w_{i})} + e^{2(w_{i}-w_{i-1})}$, using a combination of methods from p.d.e., isomonodromic deformations (Riemann-Hilbert method), and loop groups.  

We place these global solutions into the broader context of solutions which are smooth near $0$. For such solutions, we compute explicitly the Stokes data and connection matrix of the associated meromorphic system, in the resonant cases as well as the non-resonant case. 

This allows us to give a complete picture of the monodromy data, holomorphic data, and asymptotic data of the global solutions.
\end{abstract}


\maketitle 

\section{Introduction}\label{intro}

In \cite{GuLi14} we observed that the Iwasawa factorization of a certain loop group provides the link between two objects studied by Cecotti and Vafa (\cite{CeVa91},\cite{CeVa92a}) in their classification of supersymmetric field theories.  On the one hand there is the chiral ring, a holomorphic object, related to Frobenius manifolds and quantum cohomology. On the other hand there is the renormalization group flow, represented by a smooth function which combines holomorphic and antiholomorphic data, generalizing a variation of Hodge structures. This function is a solution of the (nonlinear) topological-antitopological fusion equations --- the tt* equations.

The link is a generalization of the well known correspondence between holomorphic functions and solutions of the Liouville equation, although much less explicit because an infinite-dimensional factorization is involved.
We have studied this correspondence for an equation (introduced by Cecotti and Vafa) related to the Toda equations, which we call the tt*-Toda equations.  

It is known that the tt*-Toda equations are integrable.  Nevertheless, considerable effort is required to integrate them, and what  Cecotti and Vafa call the \ll magical\rr expected properties of these solutions were left as conjectures in \cite{CeVa91},\cite{CeVa92a}.  While the mirror symmetry aspects have been shown to fit into self-consistent frameworks in both physics and algebra, the geometric framework needs differential equations and properties of their solutions. The purpose of our project is to establish these properties rigorously and explicitly, without using unduly abstract machinery. 

For any fixed $l\in\{0,1,\dots,n\}$
the  tt*-Toda equations are
\begin{equation}\label{ost}
 2(w_i)_{\ttb}=-e^{2(w_{i+1}-w_{i})} + e^{2(w_{i}-w_{i-1})}, \ 
 w_i:\C^\ast\to\R, \ 
 i\in\Z
\end{equation}
where, for all $i$,  
$w_i=w_{i+n+1}$ (periodicity),  
$w_i=w_i(\vert t\vert)$
(radial condition), and $w_i+w_{l-i-1}=0$
(\ll anti-symmetry\rrr).  
In previous articles \cite{GuItLiXX}, \cite{GuItLi15},
we have proved by elementary p.d.e.\ methods that the global solutions (on $\C^\ast$) are characterized by their asymptotics as $\vert t\vert\to 0$, or by their asymptotics as $\vert t\vert\to \infty$, as predicted by 
Cecotti and Vafa.  In this article we shall establish more of their properties.

An alternative characterization involves \ll holomorphic data\rrr, which is perhaps closer to the underlying geometry and physics.
It is well known to differential geometers that the loop group  Iwasawa factorization method produces a local solution of the periodic Toda equations near
$z_0\in\C$ from any matrix of the form
\[
\bp
 & & & p_0\\
 p_1 & & & \\
  & \ddots & & \\
   & & p_n &
\ep
\]
where each function $p_i=p_i(z)$ is holomorphic in a neighbourhood of $z_0$ (see the Appendix of \cite{GuLi14} for a summary, with further references). 
Thus, for any suitable $p_0,\dots,p_n$ (incorporating the radial and anti-symmetry conditions),  we obtain a local solution  $w_0,\dots,w_n$ of the tt*-Toda equations.  The  difficulty lies in identifying the holomorphic data which corresponds to solutions defined globally on $\C^\ast$.  In \cite{GuLi14} we used p.d.e.\ methods to identify these solutions in terms of asymptotic data $\ga_i$, where $2w_i(\vert t\vert)\sim\ga_i\log\vert t\vert$ as $t\to 0$. Then
we observed that this data corresponds to 
\[
p_i(z)=c_iz^{k_i}
\]
for certain (real) $k_i$ and $c_i$.  We gave a simple formula for $k_i$ in terms of asymptotic data, but not yet for $c_i$ (cf.\ the remarks at the end of \cite{GuLi14}).  

The theory of isomonodromic deformations provides deeper information, and a third kind of data, which will allow us to compute $c_i$.  This theory applies because of the radial condition (as was pointed out by Dubrovin in  \cite{Du93}).  It was shown in \cite{GuItLiXX} that, in this context, the counterparts of the $\ga_i$ or the $k_i$ are certain entries $s^\R_i$ of Stokes matrices of an associated meromorphic connection, and that there are explicit bijective correspondences
\[
\boxed{
\begin{matrix}
\text{asymptotic data}    
\\
\ga_i
\end{matrix}
}
\longleftrightarrow
\boxed{
\begin{matrix}
\text{Stokes data}
\\
s^\R_i
\end{matrix}
}
\longleftrightarrow
\boxed{
\begin{matrix}
\text{holomorphic data}
\\
k_i
\end{matrix}
}
\]
between the three kinds of data.  

In this paper we shall extend these correspondences further, to the case of radial solutions of (\ref{ost}) which are defined near $t=0$.  That is, we do not insist that $w_i$ be defined on all of $\C^\ast$, just in a punctured neighbourhood of $t=0$. 
Even though our main concern is the global solutions, the \ll complete picture\rr is visible only in this broader context.    
Here (in the generic case --- the meaning of which will be explained shortly)
we have 
\[
2w_i(\vert t\vert)\sim\ga_i\log\vert t\vert + \rho_i + o(1)
\]
as $t\to 0$ where $\rho_i$ is an independent real parameter (related to $c_i$). The counterparts of the $\rho_i$ or the $c_i$ are certain entries $e^\R_i$  of \ll connection matrices\rrr.  We shall show how to compute all of this data explicitly --- for local solutions near $t=0$ and, in particular, for global solutions on $\C^\ast$.  

Thus, for local solutions near $t=0$, we shall obtain bijective correspondences
\[
\boxed{
\begin{matrix}
\text{asymptotic data}    
\\
\ga_i,\rho_i
\end{matrix}
}
\longleftrightarrow
\boxed{
\begin{matrix}
\text{Stokes data}
\\
s^\R_i,e^\R_i
\end{matrix}
}
\longleftrightarrow
\boxed{
\begin{matrix}
\text{holomorphic data}
\\
k_i,c_i
\end{matrix}
}
\]
For the global solutions, the supplementary data $\rho_i,e_i^\R,c_i$ must be determined by $\ga_i,s_i^\R,k_i$.  We have shown in \cite{GuItLi15} that the global solutions are characterized by (a condition equivalent to) the very simple condition $e_i^\R=1$.  Our computation of the $e_i^\R$ will allow us to
give the explicit conditions on $\rho_i,c_i$, which are rather more complicated.

Let us be more specific. 
In this article we focus on \ll case 4a\rr  (the first of the 10 cases with $n=3,4,5$ which are listed in Table 1 of \cite{GuItLiXX}).  Here we have 
$n=3$, $l=0$ and $w_0+w_3=0=w_1+w_2$, hence
\[
\ga_0+\ga_3=0=\ga_1+\ga_2, \quad \rho_0+\rho_3=0=\rho_1+\rho_2
\]
and also
\[
k_1=k_3, \quad c_1=c_3.
\]
We restrict to this case primarily for readability.  All methods in this article work equally well for the other 9 cases, and indeed for general $n$.

We recall (Table 2 of \cite{GuItLiXX} and Corollary 4.2 of \cite{GuItLi15})
that the data $\ga_i, s_i^\R, k_i$ are related in this case by  
\begin{align*}
\ga_0  &=
\tfrac{3\al_0 - 2\al_1 - \al_2}{N}  &s_1^\R &=
-2\cos \tfrac\pi N { \al_0} +  2\cos \tfrac\pi N { \al_2}
\\
\ga_1  &=
\tfrac{\al_0 + 2\al_1 - 3\al_2}{N} &s_2^\R &=
-2+4\cos \tfrac\pi N { \al_0} \, \cos \tfrac\pi N {\al_2}
\end{align*}
where
$
\al_i=k_i\!+\!1,\quad N=\al_0\!+\!\al_1\!+\!\al_2\!+\!\al_3=\al_0\!+\!2\al_1\!+\!\al_2 .
$
 
These formulae are unaffected by rescaling $\al_i\mapsto k\al_i$, $k\in\R_{>0}$, and the correspondence
\[
(s_1^\R,s_2^\R)  \leftrightarrow
(\ga_0,\ga_1) \leftrightarrow 
(\al_0,\al_1,\al_2,\al_3) \text{ mod } \R_{>0} 
\]
is bijective.  The condition $\al_0,\al_1,\al_2,\al_3\ge 0$ corresponds to the condition $0\le \ga_0+1\le  \ga_1+3\le 4$, and this is the parameter space of global solutions (see \cite{GuItLiXX}).  It is a closed triangular region.  {\em By the generic case we mean the case $\al_0,\al_1,\al_2,\al_3>0$;} this means the interior of the region. 

\no{\em Holomorphic data and asymptotic data of local solutions.}

For local (near $t=0$) radial solutions, we have the additional parameters $\rho_i\in\R$ (asymptotic data) or $c_i>0$ (holomorphic data).   Again, $\rho_0,\rho_1$ are unaffected by rescaling $c_i\mapsto lc_i$, $l\in\R_{>0}$.  
Let us normalize the $\al_i$ and $c_i$ so that 
$c_0c_1c_2c_3=1$ and 
$\al_0\!+\!\al_1\!+\!\al_2\!+\!\al_3=1$.  Then it is easy to show (Proposition \ref{asymptoticdata}) that
the $\rho_i$ are related to the $c_i$ by
\begin{align*}
e^{\rho_0}&=c_0 \ 2^{-6\al_0+4\al_1+2\al_2}
\\
e^{\rho_1}&=c_2^{-1} \ 2^{-2\al_0-4\al_1+6\al_2}
\end{align*}
and this (together with the relation between the $\ga_i$ and the $\al_i$) gives a bijection between the $\ga_i,\rho_i$ and the normalized
$\al_i,c_i$.

\no{\em Holomorphic data and monodromy data of local solutions.}

Next we describe the monodromy data, which is associated to a complex o.d.e.\ of the form 
\[
\frac{d\Psi}{d\la}=A\Psi.
\]
The coefficient matrix $A$ depends explicitly on a particular local solution $w_0,\dots,w_n$ of (\ref{ost}), and the o.d.e.\  is meromorphic in $\la$ with poles of order $2$ at $\la=0,\infty$.  The monodromy data of this system consists of the Stokes matrices at the poles and the connection matrix, and it is known that these are independent of $t$.  It follows from general principles (the Riemann-Hilbert correspondence; see \cite{Bo01} for a version which covers our situation) that such data
parametrizes the local radial solutions of (\ref{ost}).  For local radial solutions near $t=0$ we establish this directly, but our main results are that we can give precise formulae for the correspondence, in the following way.

First, the Stokes matrices are equivalent to the $s_1^\R,s_2^\R$ above (they are the same for all local solutions near $0$, in particular for the global solutions).  
Next, the connection matrix is equivalent to certain parameters $e^\R_1, e^\R_2$
and our main result
(in the generic case, Theorem \ref{explicite}) is an explicit formula for them.  In terms of the normalized holomorphic data we have (Corollary \ref{explicitehol}):
\begin{align*}
e^\R_1&=
c_0 \, 
\frac
{
\Ga(  \frac{\pal_0}{N}  )
\Ga(  \frac{\pal_0+\pal_1}{N})
\Ga(  \frac{\pal_0+\pal_1+\pal_2}{N})
}
{
\Ga(  \frac{\pal_1}{N})
\Ga(  \frac{\pal_1+\pal_2}{N})
\Ga(  \frac{\pal_1+\pal_2+\pal_3}{N})
}
\\
e^\R_2&= c_2^{-1} \, 
\frac
{
\Ga(  \frac{\pal_3}{N}  )
\Ga(  \frac{\pal_3+\pal_0}{N})
\Ga(  \frac{\pal_3+\pal_0+\pal_1}{N})
}
{
\Ga(  \frac{\pal_2}{N})
\Ga(  \frac{\pal_2+\pal_3}{N})
\Ga(  \frac{\pal_2+\pal_3+\pal_0}{N})
}
\end{align*}
Combining these with the above formulae for $\rho_0,\rho_1$, we can express $e^\R_1, e^\R_2$ in terms of the asymptotic data (Corollary \ref{expliciteasymp}):
\begin{align*}
e^\R_1&=
e^{\rho_0}\,
2^{2\ga_0}
\frac
{
\Ga(  \frac{\pga_0}{4}+\frac14  )
\Ga(  \frac{\pga_0+\pga_1}{8}+\frac12 )
\Ga(  \frac{\pga_0-\pga_1}{8}+\frac34  )
}
{
\Ga(  \frac{\pga_1-\pga_0}{8}+\frac14  )
\Ga(  -\frac{\pga_0+\pga_1}{8}+\frac12  )
\Ga(  -\frac{\pga_0}{4}+\frac34  )
}
\\
e^\R_2&= e^{\rho_1}\,
2^{\ga_1}
\frac
{
\Ga(  \frac{\pga_1-\pga_0}{8}+\frac14  )
\Ga(  \frac{\pga_0+\pga_1}{8}+\frac12 )
\Ga(  \frac{\pga_1}{4}+\frac34  )
}
{
\Ga(  -\frac{\pga_1}{4}+\frac14  )
\Ga(  -\frac{\pga_0+\pga_1}{8}+\frac12  )
\Ga(  \frac{\pga_0-\pga_1}{8}+\frac34  )
}
\end{align*}

\no{\em Global solutions.}

As the global solutions are given by $e_1^\R=e_2^\R=1$,  for these solutions we have 
(Corollary \ref{holomorphicdata}):
\begin{align*}
c_0&=
\frac
{
\Ga( \al_1)
\Ga( \al_1\!+\!\al_2)
\Ga( \al_1\!+\!\al_2\!+\!\al_3)
}
{
\Ga( \al_0  )
\Ga( \al_0\!+\!\al_1)
\Ga( \al_0\!+\!\al_1\!+\!\al_2)
}
\\
c_2&=
\frac
{
\Ga(  \al_3  )
\Ga(  \al_3\!+\!\al_0)
\Ga(  \al_3\!+\!\al_0\!+\!\al_1)
}
{
\Ga(  \al_2)
\Ga(  \al_2\!+\!\al_3)
\Ga(  \al_2\!+\!\al_3\!+\!\al_0)
}
\end{align*}
Here $\al_1=\al_3$ but we write them separately to indicate the pattern.  

For asymptotic data this becomes
(Corollary \ref{tracywidom}):
\begin{align*}
\rho_0  &=  - \log\ 2^{2\ga_0}
\frac
{
\Ga(  \frac{\pga_0}{4}+\frac14  )
\Ga(  \frac{\pga_0+\pga_1}{8}+\frac12 )
\Ga(  \frac{\pga_0-\pga_1}{8}+\frac34  )
}
{
\Ga(  \frac{\pga_1-\pga_0}{8}+\frac14  )
\Ga(  -\frac{\pga_0+\pga_1}{8}+\frac12  )
\Ga(  -\frac{\pga_0}{4}+\frac34  )
}
\\
\rho_1 &=   - \log\ 2^{2\ga_1}
\frac
{
\Ga(  \frac{\pga_1-\pga_0}{8}+\frac14  )
\Ga(  \frac{\pga_0+\pga_1}{8}+\frac12 )
\Ga(  \frac{\pga_1}{4}+\frac34  )
}
{
\Ga(  -\frac{\pga_1}{4}+\frac14  )
\Ga(  -\frac{\pga_0+\pga_1}{8}+\frac12  )
\Ga(  \frac{\pga_0-\pga_1}{8}+\frac34  )
}
\end{align*}
This formula was obtained by Tracy and Widom in \cite{TrWi98} by the method of Fredholm determinants, although they were not able to identify the relevant solutions with the class of {\em all} smooth solutions on  $\C^\ast$.  

To perform these calculations it is very convenient to make use of the Iwasawa factorization. This arises in the following way. The fundamental solution of the meromorphic o.d.e.\ is a function which takes values in the loop group of $\SL_{n+1}\C$. The Iwasawa factorization relates this to the fundamental solution of a {\em simpler} meromorphic o.d.e.\ 
$\frac{d\Phi}{d\la}=B\Phi$ whose coefficient matrix $B$ is given directly by the holomorphic data. This has poles of order $2,1$ at $\la=0,\infty$ (instead of $2,2$),
and it can be solved \ll explicitly\rrr.  The  Iwasawa factorization 
allows us to compute the connection matrix of $\frac{d\Psi}{d\la}=A\Psi$ from that of $\frac{d\Phi}{d\la}=B\Phi$.

The Iwasawa factorization is an effective tool for studying flat connections with compact structure group, because in that case there is only one \ll Iwasawa cell\rrr.  Equation (\ref{ost}) with the opposite sign (the usual Toda equation) is a typical example.  But for equation (\ref{ost}) the structure group is noncompact, and 
in this situation there is no guarantee that the factorization of the fundamental solution matrix can be carried out, because there are many Iwasawa cells and there is no guarantee that a solution remains within a single cell.  
It is a consequence of our results that {\em the Iwasawa factorization is possible for all $z\in\C^\ast$ only when $c_i$ takes the specific value above.}

\no{\em Resonance.}

In the sense of classical o.d.e.\ theory, our equation $\frac{d\Psi}{d\la}=A\Psi$ is non-resonant.  However, when the coefficient matrix $A$ corresponds to an interior point of the space of global solutions, the monodromy is better-behaved --- in particular, it is semisimple.  At boundary points the monodromy is not semisimple. This phenomenon is directly related to resonance (in the classical sense) of the simpler o.d.e.\ $\frac{d\Phi}{d\la}=B\Phi$. This makes the computation of the connection matrix more complicated, and 
we carry it out separately in sections \ref{resD1}-\ref{resGLOBAL}.  

To explain the nature of the results, we note first that, in the generic case, the essential ingredient of the connection matrix (for $\frac{d\Psi}{d\la}=A\Psi$) is a diagonal matrix
\[
\mathcal E = 
\bp
e^\R_1\   & & & \\
  &e^\R_2 & & \\
  & & 1/e^\R_2& \\
  & & & 1/e^\R_1
  \ep
\]
(see Theorem \ref{explicite}).  In the boundary case this becomes
\[
\mathcal E \mathcal F
\]
where $\mathcal E$ is diagonal but $\mathcal F$ is unipotent
(Theorem \ref{Rexplicite} and Table \ref{t4}).  It turns out that $\mathcal F$ reflects exactly the unipotent part of the monodromy of the simpler o.d.e.\ $\frac{d\Phi}{d\la}=B\Phi$.

In the boundary case, the parameters $e_1^\R,e_2^\R$ must be replaced by new parameters $e_1^\R,f_1^\R$ or $f_1^\R,f_2^\R$.  The global solutions are given by $e_1^\R=1,f_2^\R=0$ or $f_1^\R=f_2^\R=0$.  The asymptotic data may be computed as before, but now it involves derivatives of gamma functions.   

As a concrete example we mention here the \ll most resonant\rr situation, which is given by the vertex $(\gamma_0,\gamma_1)=(3,1)$ of the triangular region of solutions.  The corresponding global solution has an important physical or geometrical origin, namely the quantum cohomology of $\C P^3$ (see \cite{GuLi12,GuLi14}). 

Here the asymptotic formula $2w_i(\vert t\vert)\sim\ga_i\log\vert t\vert + \rho_i+o(1)$ must be replaced by

\no
$
2w_0(t)=  3\log\vert t\vert + \log  
\left(
-\tfrac1{24}\zeta(3)-\tfrac43\euler^3
-4\euler^2\log\tfrac{\vert t\vert}{4}
\right.
$
\newline
$
\quad\quad\quad\quad\quad\quad\quad\quad\quad
\quad\quad\quad
\left.
-4\euler\log^2\tfrac{\vert t\vert}{4}
-\tfrac43
\log^3\tfrac{\vert t\vert}{4}
\right)
+O(\vert t\vert^4 \log^6\vert t\vert)
$

\no
$
2w_0(t)+2w_1(t)= 4\log\vert t\vert + 
\log 
\left(
\vphantom{\log\tfrac{\vert t\vert}{4}}
-\tfrac1{12}\euler\zeta(3) +\tfrac43\euler^4
\right.
+(-\tfrac1{12}\zeta(3) 
$
\newline
$
\left.
+\tfrac{16}3\euler^3)\log\tfrac{\vert t\vert}{4}
+8\euler^2\log^2\tfrac{\vert t\vert}{4}
+\tfrac{16}3\euler\log^3\tfrac{\vert t\vert}{4}
+\tfrac43\log^4\tfrac{\vert t\vert}{4}
\right)
+O(\vert t\vert^4 \log^6\vert t\vert)
$

\no
(see Corollary \ref{Rasymptoticdata}).
We refer to sections \ref{resE1}, \ref{resGLOBAL} for other explicit formulae of this kind.

 The paper is organised as follows. In section \ref{four} we introduce four connection forms 
$\al$, $\om$, $\hat\om$, $\hat\al$.  Except for $\hat\om$, these were used in our previous articles, but now we need more details of the gauge transformations relating them.   In section \ref{omegahat}
we give the Stokes and connection matrices of $\hat\om$.   In section \ref{dkandek} we use the Iwasawa factorization to compute the connection matrix of $\hat\al$ corresponding to local solutions of (\ref{ost}) near $t=0$.  In section \ref{conclusions} we characterize the global solutions in terms of monodromy data, holomorphic data, and asymptotic data.   Sections
\ref{omegahat}-\ref{conclusions} concentrate on the generic case; the modifications needed for the non-generic case are made in sections  
\ref{resD1}-\ref{resGLOBAL}. We have made some effort to simplify the presentation there, and omit repetitive computations. But the non-generic case is rather involved, and the general reader is advised to glance at the notation at the start of section \ref{resD1}, then jump to section \ref{resGLOBAL} to see a summary of the results.

Acknowledgements:   
The first author was partially supported by JSPS grant (A) 25247005, and the second author was partially supported by NSF grant  DMS-1361856.  Both are grateful to Taida Institute for Mathematical Sciences for financial support and hospitality.  We would like to express our sincere appreciation to Yuqi Li, who verified numerically all our asymptotic formula, thereby revealing several errors in our original calculations.
Numerical aspects of the tt*-Toda equations will be discussed by him in a future article.

\section{Preliminaries:  four connection forms}\label{four}

In this section we introduce the four connection forms
$\om,\hat\om,\al,\hat\al$ which will be used to study equation (\ref{ost}).
Although $\al,\hat\al$ appeared already in \cite{GuItLiXX} and \cite{GuItLi15}, some
new details are given here, in order to establish the precise relation with $\om,\hat\om$.

\subsection{The connection form $\al$}\label{four1}  \ 

\begin{definition}\label{alpha}
Let
$
 \al = (w_t+\tfrac1\la W^T)dt + (-w_{\tbar}+\la W)d\tbar,
$
where
 \[
 w=\diag(w_0,\dots,w_n),\ \ 
 W=
 \left(
\begin{array}{c|c|c|c}
\vphantom{(w_0)_{(w_0)}^{(w_0)}}
 & \!e^{w_1\!-\!w_0}\! & &  
 \\
\hline
  &  & \  \ddots\   & \\
\hline
\vphantom{\ddots}
  & &  &  e^{w_n\!-\!w_{n\!-\!1}}\!\!\!
\\
\hline
\vphantom{(w_0)_{(w_0)}^{(w_0)}}
\!\! e^{w_0\!-\!w_n} \!\!  & &  &  \!
\end{array}
\right),
\]
$W^T$ denotes the transpose of $W$, and $\la$ is a complex parameter.
\end{definition}

Equation (\ref{ost}) is equivalent to $d\al+\al\wedge\al=0$, i.e.\ the condition that the connection $\na=d+\al$ has curvature zero.  This is the compatibility condition for the linear system
\begin{align}
\begin{cases}
\Psi_t&=(w_t+\tfrac1\la W)\Psi\\
\Psi_\tbar&=(-w_{\tbar}+\la W^T)\Psi
\end{cases}
\end{align}
where $\Psi$ takes values in the Lie group $\SL_{n+1}\C$ (the invertible complex 
$n\!+\!1\times n\!+\!1$ matrices with determinant $1$).  In other words, it is the condition for the existence of a (local) basis of flat sections of $d+\al$ or of the dual connection $d-\al^T$, i.e.\ the existence of $\Psi$ such that $\al^T=d\Psi\Psi^{-1}$. 

Consider the following automorphisms $\tau, \si, c$ of the Lie algebra $\sl_{n+1}\C$ (the complex $n\!+\!1\times n\!+\!1$ matrices with trace $0$):
\begin{align*}
\tau(X)&=d_{n+1}^{-1} X d_{n+1}
\\
\si(X)&=-\De\, X^T\De
\\
c(X)&=\De \bar X  \De
\end{align*}
where
$d_{n+1}=\diag(1,\om,\dots,\om^n)$,  $\om=e^{{2\pi \i}/{(n+1)}}$,
and
\[
\De=\De_{l,n+1-l}=
\bp
J_l & \\
 & J_{n+1-l}
 \ep,
 \quad
 J_l=
 \bp
  & & 1 \\
  & \iddots \, & \\
1 & &
  \ep  \ \text{($l\times l$ matrix).}
 \]
We use the same notation for the corresponding automorphisms of the Lie group
$\SL_{n+1}\C$: for $A\in\SL_{n+1}\C$ this means
\begin{align*}
\tau(A)&=d_{n+1}^{-1} A d_{n+1}
\\
\si(A)&=\De\, A^{-T}\,\De
\\
c(A)&=\De \bar A  \De.
\end{align*}

It is easy to verify that $\al$ has the following properties.

\no{\em Cyclic symmetry: }  $\tau(\al(\la))=\al(e^{{2\pi \i}/{(n+1)}} \la)$

\no{\em Anti-symmetry: }  $\si(\al(\la))=\al(-\la)$

\no{\em Reality: }  $c(\al^\pr(\la))=\al^\prr(1/\bar\la)$
where $\al = \al^\pr dt + \al^\prr d\tbar$.

\subsection{The connection form $\om$}\label{four2}  \ 

Closely related to the $C^\infty$ connection form $\al$
is the holomorphic connection form $\om$, defined as follows:

\begin{definition}\label{omega}
Let $\om=\tfrac1\la\eta dz$,
where
\begin{equation*}
\eta=
\bp
 & & & p_0\\
 p_1 & & & \\
  & \ddots & & \\
   & & p_n &
   \ep
\end{equation*}
and
each $p_i=p_i(z)$ is a holomorphic function. 
\end{definition}

In our earlier articles \cite{GuLi14},\cite{GuLi12} the form $\om$ was in the background, providing the link with quantum cohomology.  In parts I and II  of this series (\cite{GuItLiXX},\cite{GuItLi15}),  it was not used at all.  However, in this part III, $\om$ will be essential.  

The cyclic symmetry condition  $\tau(\om(\la))=\om(e^{{2\pi \i}/{(n+1)}} \la)$ is satisfied for any $p_i$. We shall impose the anti-symmetry condition $\si(\om(\la))=\om(-\la)$; this means that
\begin{equation*}
\begin{cases}
\ \  p_0=p_l, \ p_1=p_{l-1},\ \ \dots\\
\ \  p_{l+1}=p_n, \ p_{l+2}=p_{n-1},\ \ \dots
 \end{cases}
\end{equation*}
The third symmetry, the reality condition, is not relevant to $\om$.

Let us review the relation between $\al$ and $\om$ as we shall need this later on (cf.\ sections 4 and 5 of \cite{GuLi14}).  Given $\om$ as above, and a basepoint $z_0$, let $L=L(z,\la)$ be the (unique) local holomorphic solution of the o.d.e.\ 
\[
L^{-1}L_z=\tfrac1\la \eta,\quad L\vert_{z=z_0}=I.
\]
Let
\[
L=L_\R L_+
\]
be the  Iwasawa factorization\footnote{This means the Iwasawa factorization for the complex loop group
$\La \SL_{n+1}\C$
with respect to the real form given by the involution 
$\ga(\la)\mapsto c(\ga(1/{\bar\la}))$, $\ga\in \La \SL_{n+1}\C$. See \cite{PrSe86}, \cite{Gu97}, \cite{BaDo01}, \cite{Ke99} for more information on Iwasawa factorizations.}
of $L$, where 
$c(L_\R(z,\zbar,1/{\bar\la}))=L_\R(z,\zbar,\la)$, and
$L_+(z,\zbar,\la)=\sum_{i=0}^\infty L_i(z,\zbar) \la^i$, $L_0=\diag(b_0,\dots,b_n)$, $b_i>0$.
This factorization is valid (i.e.\ the $C^\infty$ maps $L_\R,L_+$ exist and are unique) for all $z$ in some neighbourhood of $z_0$.  
We have $L_\R\vert_{z=z_0}=I=L_+\vert_{z=z_0}$.  

Both $L_\R$ and $L_+$ satisfy the cyclic symmetry and anti-symmetry conditions. 
From this one can deduce that $L_\R^{-1} dL_\R$ must be of the form
$
\al = (a+\tfrac1\la A^T)dz + (-\bar a+\la \bar A)d\zbar,
$
where
 \[
 a=\diag(a_0,\dots,a_n),\ \ 
 A=
 \left(
\begin{array}{c|c|c|c}
\vphantom{(w_0)_{(w_0)}^{(w_0)}}
 & \! A_1 \! & &  
 \\
\hline
  &  & \  \ddots\   & \\
\hline
\vphantom{\ddots}
  & &  & A_{n}\!\!\!
\\
\hline
\vphantom{(w_0)_{(w_0)}^{(w_0)}}
\!\! A_0 \!\!  & &  &  \!
\end{array}
\right)
\]
and
$A_i=p_i b_i/b_{i-1}$,  $a_i=(\log b_i)_z$.  As $\al=L_\R^{-1}dL_\R$, the zero curvature equation $d\al+\al\wedge\al=0$ holds, which means
\begin{equation}\label{preost}
(a_i)_{\zbar} + (\bar a_i)_z
 = -\vert A_{i+1}\vert^2 + \vert A_{i}\vert^2.
\end{equation}
Thus, the holomorphic form $\om$ gives rise to $C^\infty$ functions $a_i,A_i$ near $z_0$ which satisfy the p.d.e.\ (\ref{preost}).

To establish the relation with our p.d.e.\  (\ref{ost}), let us introduce
\begin{equation}\label{definitionofwi}
w_i=\log b_i/\vert h_i\vert,\quad
G=\diag( \vert h_0\vert/h_0, \dots, \vert h_n\vert/h_n )=\vert h\vert/h
\end{equation}
where the $h_i$ are holomorphic functions.
If we choose the $h_i$ so that
\begin{equation}\label{nu}
p_0h_0/h_n = p_1 h_1/h_0 = \cdots = p_n h_n/h_{n-1} 
\ \  \text{ ($=\nu$, say)}
\end{equation}
hence $\nu^{n+1}=p_0\dots p_n$,
then direct calculation gives 
\[
(L_\R G)^{-1} d(L_\R G) =
w_z dz +\tfrac1\la W^T \nu dz - w_{\zbar} d\zbar + \la W \bar\nu d\zbar.
\]
The zero curvature equation (\ref{preost}) becomes
\begin{equation}\label{nuost}
2(w_i)_{\zzb}=
-\vert \nu\vert^2 e^{2(w_{i+1}-w_{i})} 
+ \vert \nu \vert^2 e^{2(w_{i}-w_{i-1})}.
\end{equation}
Then the change of variable\footnote{In \cite{GuItLiXX},\cite{GuItLi15}, for economy of notation, we wrote both (\ref{ost}) and (\ref{nuost}) with $z$.  Here, to avoid confusion,  we use $t$ for $\al$ and (\ref{ost}), and
$z$ for $\om$ and (\ref{nuost}).
}
$dt/dz=\nu$ gives 
$(L_\R G)^{-1} d(L_\R G) =
(w_t+\tfrac1\la W^T)dt + (-w_{\tbar}+\la W)d\tbar = \al$, and (\ref{nuost}) gives the tt*-Toda equations (\ref{ost}). 

Any choice of $h_0$ and $n+1$-th root
$\nu=(p_0\dots p_n)^{\frac1{n+1}}$ determines all the $h_i$, as
$h_i=h_0 \nu^i (p_1\dots p_i)^{-1}$.   If we impose the natural anti-symmetry  condition
\begin{equation*}
\begin{cases}
\ \  h_0h_{l-1}=1, \ h_1h_{l-2}=1,\ \ \dots\\
\ \  h_lh_{n}=1, \ h_{l+1}h_{n-1}=1,\ \ \dots
 \end{cases}
\end{equation*}
then $h_0$ (and hence all $h_i$) are uniquely determined in terms of the $p_i$. 

For example, in the case $n=3$, $l=0$, we have
$p_1=p_3$ and $h_0h_3=1$, $h_1h_2=1$, and we obtain
\begin{equation}\label{hzerohone}
h_0=p_0^{-3/8}p_1^{2/8}p_2^{1/8}=h_3^{-1},
\quad
h_1=p_0^{-1/8}p_1^{-2/8}p_2^{3/8}=h_2^{-1}.
\end{equation}

This completes our summary of the relation between $\al$ and $\om$.  We regard $\om$ as \ll holomorphic data\rr for the construction of (local, near any point $z_0$) solutions of (\ref{ost}). To study such solutions by the isomonodromy method, we introduce two further connections $\hat\om,\hat\al$ next.  

\subsection{The connection form $\hat\om$}\label{four3}  \ 

From now on, we consider only $p_i$ of the special form
\[
p_i=c_i z^{k_i}
\]
with $c_i>0,k_i\in\R$.    As in section \ref{intro}, let us write
\[
\al_i=k_i+1
\]
and
\[
c=
\overset{\scriptstyle n}{\underset{\scriptstyle i=0}\Pi}
 \ c_i,\quad
N=
 \overset{\scriptstyle n}{\underset{\scriptstyle i=0}\Sigma}
\ \al_i.
\]
Then the coordinate $t$ (given by
$dt/dz=\nu=(p_0\dots p_n)^{\frac1{n+1}}$)
can be taken as
\[
t=\tfrac{n+1}N \ c^{\frac1{n+1}} \ z^{\frac N{n+1}}.
\]

The special form of the $p_i$ can be interpreted as a homogeneity (or scaling-invariance) condition
on the flat connection $d+\om$.  We shall show that it leads to an extended flat connection
$d+\om+\hat\om$, where $\hat\om$ represents a covariant derivative in the $\la$-direction. As we shall see,  $d+\hat\om$ is a meromorphic connection with poles at $\la=0$ and $\la=\infty$.   It is an elementary but fundamental fact that the flatness of $d+\om+\hat\om$ implies that the monodromy data of $\hat\om$ at these poles is independent of $t$.  

The homogeneity condition also gives the radial property of the corresponding local solutions $w_i$ of equation (\ref{ost}), i.e.\  that they depend only on $\vert t\vert$.  This follows from the discussion of $\hat\al$ in the next subsection.

The relation between $\hat\om$ and $\om$ is well known and often stated, though not often explained, in the literature.  Therefore we shall give two explanations, one computational and one theoretical.  Both will be used later on.

The first approach is based on the weighted homogeneity of matrix entries, where the weights of the variables are declared to be
\[
\wt(\la) = 2,\quad \wt(z) = \tfrac2N(n+1),\quad \wt(t)=2.
\]
A function $f=f(z,\la)$ is said to have weight $k$ if $(2\bla+\frac2N(n+1)\bz)f=kf$,
where  $\bla=\la \frac{\b}{\b\la}$, $\bz=z \frac{\b}{\b z}$ (and similarly for $f=f(t,\la)$).  The following notation will also be convenient:

\begin{definition}\label{weightofhandchat}  $ $

\no(i) Let $\nn=\diag(\nn_0,\dots,\nn_n)$, where
$\wt(h_i)=2\nn_i$.

\no(ii) Let $\hat c=\diag(\hat c_0,\dots,\hat c_n)$, where
$h_i=\hat c_i t^{\nn_i}$.
\end{definition}

We can observe now that the $(i,j)$ entry of the matrix-valued $1$-form $\om=\frac1\la\eta dz$ has weight $2(\nn_j-\nn_i)$ (if it is nonzero).  If $\om=L^{-1}L_z dz$ for some $L$ with $L(0)=I$, then the $(i,j)$ entry of $L$ would also have weight $2(\nn_j-\nn_i)$.  In this case, let us consider $gL$ where
\[
g(\la)=
\diag(\la^{\nn_0},\dots,\la^{\nn_n}) = \la^\nn.
\]
Then all entries of the $i$-th column of $gL$ have weight $2\nn_i$.  Thus
\[
(2\bla + \tfrac2N(n+1) \bz) (gL) = (gL) 2\nn.
\]
We obtain
\begin{align*}
(gL)^{-1}(gL)_\la &=  - \tfrac{n+1}N \tfrac z{\la} (gL)^{-1}(gL)_z + \tfrac1\la \nn\\
&= - \tfrac{n+1}N \tfrac z{\la} L^{-1} L_z + \tfrac1\la \nn\\
&=- \tfrac{n+1}N \tfrac z{\la^2} \eta +  \tfrac1\la \nn.
\end{align*}
This calculation motivates\footnote{At this point we have not proved the existence of $L$ with $L(0)=I$; we are just using $L$ as motivation. Neither Definition \ref{hatomega} nor the flatness of
$d+\om+\hat\om$ depends on the existence of $L$.}
 the introduction of the connection form 
$\hat\om=(gL)^{-1} (gL)_\la \,d\la$:

\begin{definition}\label{hatomega}  
$
\hat\om=
\left[
-\tfrac{n+1}N \tfrac{z}{\la^2}
\ \eta
+\tfrac1\la
\ \nn
\right]
d\la
$.
\end{definition}
It can be verified that the extended connection $d+\om+\hat\om$ is flat.  

The more abstract approach involves the rank $n+1$ D-module 
$
D_z/(T_0),
$
where $D_z$ is the ring of differential operators in $z$ and $(T_0)$ is the left ideal of $D_z$ generated by the operator
\[
T_0=\b_z
\overset{\scriptscriptstyle n}{\underset{\scriptstyle k=1}\Pi}
(\b_z-
\overset{\scriptscriptstyle k}{\underset{\scriptstyle j=1}\Sigma}
\al_j)
-\tfrac{z^{n+1}}{\la^{n+1}} 
\overset{\scriptscriptstyle n}{\underset{\scriptstyle j=0}\Pi}
\, p_j.
\]
In the rest of this article we give detailed calculations only in the case $n=3$, $l=0$ (case 4a in Table 1 of \cite{GuItLiXX}), so, for simplicity,  let us focus on this.  
Here we have $c_1=c_3$, $\al_1=\al_3$, and 
$h_i$ is given by (\ref{hzerohone}).  From Definition \ref{weightofhandchat} we have
\begin{equation}\label{weights}
\begin{cases}
\nn_0 = \tfrac1{2N}(-3\al_0+2\al_1+\al_2)=-\nn_3
\\
\nn_1= \tfrac1{2N}(-\al_0-2\al_1+3\al_2)=-\nn_2
\end{cases}
\end{equation}
and
\begin{equation}\label{chat}
\begin{cases}
\hat c_0&= h_0/t^{\nn_0} 
=
\left( \tfrac N4 \right)^{\nn_0} 
c_0^{\frac{-3-2\nn_0}8}c_1^{\frac{2-4\nn_0}8}c_2^{\frac{1-2\nn_0}8} 
=
\ {\hat c_3}^{-1}
\\
\hat c_1&=h_1/t^{\nn_1}
=
\left( \tfrac N4 \right)^{\nn_1}
c_0^{\frac{-1-2\nn_1}{8}}c_1^{-\frac{-2-4\nn_1}{8}}c_2^{\frac{3-2\nn_1}{8}}
=
\ {\hat c_2}^{-1}
\end{cases}
\end{equation}
The operator
\[
T_0=\b_z(\b_z-\al_1)(\b_z-(\al_1+\al_2))(\b_z-(\al_1+\al_2+\al_3))
-\tfrac{z^4}{\la^4} p_0p_1p_2p_3
\]
is related to the flat connection $d+\om$ in the following way (see Chapter 4 of \cite{Gu08} for more details).  

The equation for flat sections of the (dual) connection $d-\om^T$ is
\begin{equation}\label{omsec}
\bz Y =\tfrac z\la \eta^T Y,\quad
Y=
\bp
y_0 \\ y_1 \\ y_2 \\ y_3
\ep.
\end{equation}
This system corresponds to the scalar equation 
\[
T_0 \, y_0=0
\]
for $y_0$, together with
the formulae
\begin{align*}
y_1&= (zp_1)^{-1}\la\bz \,y_0\overset{\scriptstyle \text{def} }=P_1\,y_0 
\\ 
y_2& = (zp_2)^{-1}\la\bz (zp_1)^{-1}\la\bz \,y_0\overset{\scriptstyle \text{def} }= P_2\,y_0
\\
y_3&= (zp_3)^{-1}\la\bz (zp_2)^{-1}\la\bz (zp_1)^{-1}\la\bz \,y_0 \overset{\scriptstyle \text{def} }= P_3\,y_0 
\end{align*}

Conversely, starting with the scalar operator $T_0$, the connection form $\om$ may be recovered  by expressing the operator $\bz:[P]\mapsto[\bz P]$ on $D_z/(T_0)$ with respect to the basis
$P_0=1, P_1, P_2, P_3$.
The operators $P_0,P_1,P_2,P_3$ have weights 
\[
0, \quad 2-\tfrac 8N \al_1, \quad 4-\tfrac 8N (\al_1+\al_2), \quad 6-\tfrac 8N (\al_1+\al_2+\al_3)
\]
respectively. From the formula (\ref{nu}) defining the $h_i$, and the fact that $\wt(h_i) = 2\nn_i$, we have
\begin{equation}\label{alphaiandni}
\tfrac 4N\al_i=\nn_{i-1}-\nn_i+1.
\end{equation}
Thus $\wt(P_i)=-2\nn_0+2\nn_i$.

The extended connection $\hat\om$ arises because $T_0$ is a homogeneous differential operator (of weight $0$): the D-module can be extended by adding $\la$ as a new variable and the \ll Euler vector field\rr as a new relation.  This extended D-module
 \[
 D_{z,\la}/(T_0,\bla+\tfrac4N\bz)
 \]
 also has rank $4$ (by direct calculation, or by the criterion of Corollary 4.19 of \cite{Gu08}) and
 $P_0,P_1,P_2,P_3$ is still a basis.  One obtains the connection form $\hat\om$ of Definition \ref{hatomega}
 by expressing the operator $\bla:[P]\mapsto[\bla P]$ with respect to the (modified) basis 
 $\la^{\nn_0}P_0, \la^{\nn_0}P_1, \la^{\nn_0}P_2, \la^{\nn_0}P_3$. (This modification by 
 $\la^{\nn_0}$ makes $\hat\om$ have trace zero, which is convenient for later use.) The advantage of this second approach is that $\hat\om$ arises naturally from the D-module, without consideration of (hypothetical) solutions $L$.

The scalar operator corresponding to $\hat\om$ is 
\[
\hat T_0=
(\bla-\nn_0)(\bla-\nn_1+1)(\bla-\nn_2+2)(\bla-\nn_3+3)
- \tfrac{4^4}{N^4} \tfrac{z^4}{\la^4} p_0p_1p_2p_3.
\]
This is obtained by substituting $\bz=-\tfrac N4 \bla$ into $\la^{\nn_0} T_0 \la^{-\nn_0}$, and using the fact that $\bla \la^{-\nn_0}= \la^{-\nn_0}(\bla-\nn_0)$. 

The equation for flat sections of $d-\hat\om^T$ is
\begin{equation}\label{hatomsec}
\bla \hat Y =\left[
-\tfrac4N \tfrac{z}{\la}
\ \eta^T
+
\ \nn
\right] \hat Y,\quad
\hat Y= 
\bp
\hat y_0 \\ \hat y_1 \\ \hat y_2 \\ \hat y_3
\ep
=
\la^{\nn_0} Y.
\end{equation}
This system corresponds to the scalar equation
\[
\hat T_0 \,\hat y_0=0
\]
for $\hat y_0=\la^{\nn_0} y_0$, together with the expressions for the $\hat y_i$ in terms of $\hat y_0$:
\begin{align*}
\hat y_1& =-\tfrac N4 (zp_1)^{-1} \la  (\bla-\nn_0) \,\hat y_0
\overset{\scriptstyle \text{def} }=\hat P_1\,\hat y_0
\\ 
\hat y_2&=\left(-\tfrac N4\right)^2 (z^2p_1p_2)^{-1} \la  (\bla-\nn_1) \la (\bla-\nn_0) \,\hat y_0
\overset{\scriptstyle \text{def} }=\hat P_2\,\hat y_0
\\
\hat y_3&=\left(-\tfrac N4\right)^3 (z^3p_1p_2p_3)^{-1} \la  (\bla-\nn_2) \la  (\bla-\nn_1) \la (\bla-\nn_0) \,\hat y_0
\overset{\scriptstyle \text{def} }=\hat P_3\,\hat y_0
\end{align*}
The scalar equation will be useful for calculations in section \ref{omegahat}.

\subsection{The connection form $\hat\al$}\label{four4}   \ 

Just as $\hat\om=(gL)^{-1}(gL)_\la d\la$ can be obtained from $\om= 
L^{-1}L_z dz$, we obtain an extended connection form 
$\hat\al=(gL_\R G)^{-1}(gL_\R G)_\la d\la$ from $\al=(L_\R G)^{-1}d(L_\R G)
= (L_\R G)^{-1}(L_\R G)_t dt + (L_\R G)^{-1}(L_\R G)_\tbar d\tbar$.
As we shall have no use for the scalar operator (or D-module) in this case, and the weighted homogeneity calculation is 
very easy,  we shall just give the latter.

Let us note first that $L_\R$ (and $L_+$) has the same homogeneity as $L$, i.e.\ the 
$(i,j)$ entry has weight $2(\nn_j-\nn_i)$.  Here we assume that $L$ satisfies $L\vert_{z=0}=I$
and we assign $\wt(\zbar)=-\wt(z)$.  
Now, the weights of the diagonal entries of both $g$ and $G^{-1}$ are $\nn_0,\nn_1,\nn_2,\nn_3$.  Hence all entries of $g L_\R G$ have weight zero. This means
$(\bla+\bt-\btbar)\, g L_\R G=0$ (the variables $t$, $\tbar$ have weights $2,-2$).  Hence
\begin{align*}
(gL_\R G)^{-1}\bla (gL_\R G)  &= 
-(gL_\R G)^{-1}\bt (gL_\R G) + (gL_\R G)^{-1}\btbar(gL_\R G)
\\
&= 
-(L_\R G)^{-1}\bt (L_\R G) + (L_\R G)^{-1}\btbar(L_\R G)
\\
&=-t(w_t+\tfrac1\la W^T) + \tbar(-w_{\tbar} + \la W)
\\
&= -\tfrac t\la W^T - tw_t - \tbar w_{\tbar} + \tbar\la W
\\
&=-\tfrac t\la W^T - xw_x + \tbar\la W \ \ \text{where $x=\vert t\vert$}.
\end{align*} 
This motivates the introduction of  
$\hat\al=(gL_\R G)^{-1}(gL_\R G)_\la d\la$:

\begin{definition}\label{hatalpha}  
$
\hat\al=
\left[
- \tfrac{t}{\la^2}
\ W^T
- \tfrac1\la xw_x + \tbar \,W
\right]
d\la.
$
\end{definition}
The extended connection $d+\al+\hat\al$ is flat.  We note that the weight of $W$ (and hence of each $w_i$) is zero; this means that $w_i=w_i(t,\bar t)$ is radial, i.e.\ 
$w_i=w_i(\vert t\vert)$.

\section{Monodromy data for $\hat\om$}\label{omegahat}

The meromorphic connection $d+\hat\om$ has poles of order $2$ (at $\la=0$) and 1 (at $\la=\infty$).  We shall see that the associated meromorphic differential equation
\begin{equation}\label{phiode}
\tfrac{d\Phi}{d\la}=
\left[
-\tfrac4N \tfrac{z}{\la^2}
\ \eta^T
+\tfrac1\la
\ \nn
\right]
\Phi
\end{equation}
admits a solution with an integral representation, whose asymptotic behaviour at $\la=0,\infty$ can be calculated, and that all monodromy data can be obtained from this.  
This approach is well known (see for example \cite{BaJuLu79} and more recently, in our context, 
\cite{Du96} and \cite{Gu99}).  Nevertheless, we shall carry it out explicitly in this section, as the results available in the literature do not cover what we need.  The conclusions are stated in Theorem \ref{stokesomegahat} (Stokes matrices) and Theorem  \ref{connectionomegahat} (connection matrix). 

To streamline the presentation we shall follow closely our treatment of $\hat\al$ in section 4 of \cite{GuItLiXX} and
section 2 of \cite{GuItLi15}, as $\hat\om$ has similar behaviour at $\la=0$ .  For $\hat\al$ the meromorphic differential equation
is
\begin{equation}\label{psiode}
\tfrac{d\Psi}{d\la}=
\left[
-\tfrac{t}{\la^2} W - \tfrac{1}{\la} xw_x + \tbar W^T
\right]
\Psi.
\end{equation}
Putting $\ze=\la/t$ converts this to
\begin{equation}\label{psiodezeta}
\tfrac{d\Psi}{d\ze}=
\left[
-\tfrac{1}{\ze^2} W - \tfrac{1}{\ze} xw_x + x^2 W^T
\right]
\Psi,
\end{equation}
which is the equation used for all Stokes calculations in \cite{GuItLiXX}, \cite{GuItLi15}. We shall not use $\ze$ in this section, but in our Stokes calculations we shall assume that $t>0$ for compatibility with \cite{GuItLiXX}, \cite{GuItLi15}. This has no effect on the results as all monodromy data is independent of $t$.

\no{\em Analysis at $\la=0$.}

The Stokes analysis of (\ref{phiode}) at $\la=0$ is very similar to that of (\ref{psiode}). The only difference is the diagonalizing matrix for the leading term $-\tfrac 4N \tfrac z{\la^2} \eta^T$ of (\ref{phiode}).  By formula (\ref{nu}), we have
\[
\eta^T=\nu\, h \, \Pi \, h^{-1}, \ \  
\Pi=
\bp
  & \!1\! & & \\
 & & \!1\! & \\
  & & & \!1\\
1\!   & & &
\ep,\ \ 
h=
\bp
\! h_0\!\! & & & \\
 & \!\!  h_1\!\!  & & \\
  & & \!\!  h_2\!\!  & \\
   & & & \!\!  h_3\!
\ep.
\]
Then\footnote{We have used the notation $\om=e^{\frac\pi2 \i}$ as well as $\om=\tfrac1\la\eta dz$,
but the intended meaning will be clear from the context.}, from
$\Pi=\Om d_4 \Om^{-1}$, 
\[
\Om=
\bp
1 & 1& 1 & 1\\
1 & \om & \om^2 & \om^3 \\
1 & \om^2 & \om^4 & \om^6 \\
1 & \om^3 & \om^6 & \om^9 
\ep, \ 
d_4=
\bp
\!1\! & & & \\
 & \!\om\! & & \\
  & & \!\!\om^2\!\! & \\
  & & & \!\!\om^3\!
\ep,\ 
\om=e^{\frac\pi2 \i},
\]
we obtain
$
\eta^T=\nu (h\Om) d_4 (h\Om)^{-1}.
$
This gives 
\[
-\tfrac 4N \tfrac z{\la^2} \eta^T = -\tfrac t{\la^2} O_0 d_4 O_0^{-1}
\]
where $O_0=h\Om$. Here we have used
$t=\tfrac{4}N \, c^{\frac1{4}} \, z^{\frac N{4}}$ and $\nu=
c^{\frac1{4}} \, z^{\frac {N-4}{4}}$ (see section \ref{four3}),  to write $\tfrac 4N z\nu =t$.

It follows (e.g.\ from \cite{FIKN06}, Proposition 1.1) that there is a unique formal solution of the form
\[
\Phiz_f=O_0 (I+O(\la) ) e^{\frac t\la d_4}.
\]
Let us choose initial Stokes sector
\[
\Omz_1=\{ \la\in\C^\ast \st -\tfrac{3\pi}4<\argu\la <\tfrac\pi2\},
\]
and then define successive Stokes sectors in the universal covering 
$\tilde{\C}^\ast$ by
\[
\Omz_{k+\frac14}= e^{-\frac\pi4\i} \Omz_k \quad (k\in \tfrac14\Z),
\]
as in section 4 of \cite{GuItLiXX} and
section 2 of \cite{GuItLi15}.
Then (e.g.\ from \cite{FIKN06}, Theorem 1.4) there exist unique holomorphic solutions 
$\Phiz_k$ on $\Omz_k$ 
such that 
$\Phiz_k\sim\Phiz_f$ as $\la\to0$ in $\Omz_k$. 

The Stokes matrices $R_k$ for (\ref{phiode}) are defined by 
\[
\Phiz_{k+1}=\Phiz_k R_k.
\]
We define matrices $P_k$ by $\Phiz_{k+\frac14}=\Phiz_k P_k$.  Thus 
$R_k=P_k P_{k+\frac14}P_{k+\frac24}P_{k+\frac34}$.

Like (\ref{psiode}),
equation (\ref{phiode})
has certain symmetries.  We state the following results without proof, as they may be obtained easily by the method of Lemmas 4.1, 4.3, 4.5 in \cite{GuItLiXX}.

\begin{lemma}\label{solutionsymmetries}   {\em (Symmetries of $\Phiz_k$)}

\no{Cyclic symmetry: }  
\newline
$d_4^{-1} \,\Phiz_f(\om\la)\,\Pi^{-1}=\Phiz_f(\la)$
\newline 
$d_4^{-1}\Phiz_{k-\frac12}(\om\la)\Pi^{-1}=\Phiz_k(\la)$

\no{Anti-symmetry: }  
\newline 
$\De\, \Phiz_f(\om^2\la)^{-T}\,4d_4^{-1}=\Phiz_f(\la)$
\newline
$\De \Phiz_{k-1}(\om^2\la)^{-T}\,4d_4^{-1}=\Phiz_k(\la)$
\end{lemma}

\begin{lemma}\label{stokessymmetries}   {\em(Symmetries of $P_k$)}

\no{Cyclic symmetry: }  
$P_{k+\scriptstyle\frac12} = \Pi\,  P_k \, \Pi^{-1}$

\no{Anti-symmetry: }  
$P_{k+1} =  d_4 \, P_k{}^{-T}\, d_4^{-1}$
\end{lemma}

Thus, all matrices $R_k,P_k$ can be determined from $P_1,P_{1\frac14}$. 
Using the same method as in section 4 of \cite{GuItLiXX}, it can be shown that these have the form
\begin{equation}\label{PP}
P_1=
\bp
1 & & & \\
\om^{\frac52}r^\R_1 & 1 & & \\
 & & 1 & \om^{-\frac12}r^\R_1 \\
  & & & 1
\ep,
\ \ 
P_{1\scriptstyle\frac14}=
\bp
1 & & & \\
 & 1 & & \om r^\R_2\\
 & & 1 &  \\
  & & & 1
\ep
\end{equation}
where 
$r_1^\R,r_2^\R\in\R$.   The fact that $r_1^\R,r_2^\R$ are real follows from the \ll elementary reality condition\rr\ 
$\overline{\hat\om(\bar \la)}=\hat\om(\la)$ (which holds because we are assuming $z>0$ in this section).  
The argument is exactly the same as in Proposition 4.6 of \cite{GuItLiXX}.


As in \cite{GuItLiXX}, \cite{GuItLi15}, the triviality of the formal monodromy leads to
the formula
$\Phiz_{k-2}(e^{2\pi \i}\la)=\Phiz_k(\la)$. From this we obtain
the monodromy of $\Phiz_k$:
\[
\Phiz_k(e^{2\pi \i}\la)=\Phiz_k(\la)R_kR_{k+1}.
\]
The cyclic symmetry gives  $R_kR_{k+1}=(P_kP_{k+\frac14}\Pi)^4$.

\no{\em Analysis at $\la=\infty$.}

As $\la=\infty$ is a regular singular point, the situation is simpler.  Let us consider first the non-resonant case, which we take to mean $\nn_i-\nn_j\notin\Z$ for all $i,j$. (The resonant case will be discussed at length in section \ref{resD1}.)
By Theorem 1.2 of \cite{FIKN06} (for example), there is a unique solution of (\ref{phiode}) of the form
\[
\Phii(\la)=\left(I+O(\tfrac1\la)\right)\la^\nn = \left(I+O(\tfrac1\la)\right)e^{(\log\la)\nn}.
\]
For compatibility with the notation at $\la=0$, let us choose the usual branch of $\log\la$ on the sector $\Omz_1$ (which takes positive values when $\la>1$), then extend by analytic continuation.  

\begin{lemma}\label{regsolutionsymmetries}   {\em (Symmetries of $\Phii$ in the non-resonant case)}

\no{Cyclic symmetry: }  
$d_4^{-1} \,\Phii(\om\la)d_4 \om^{-\nn}=\Phii(\la)$

\no{Anti-symmetry: }  
$\De\, \Phii(\om^2\la)^{-T} \om^{2\nn} \De=\Phii(\la)$
\end{lemma}

\begin{proof}  It is easy to check that $d_4^{-1}\Phii(\om\la)$ is also a solution of (\ref{phiode}), hence
$d_4^{-1}\Phii(\om\la)=\Phii(\la)C$ for some constant $C$. But
\begin{align*}
d_4^{-1}\Phii(\om\la)
&=
d_4^{-1} \left(I+O(\tfrac1\la)\right) \om^\nn \la^\nn\\
&=
\left(I+O(\tfrac1\la)\right)  \la^\nn \om^\nn d_4^{-1},
\end{align*}
so $C=\om^\nn d_4^{-1}$. This gives the cyclic symmetry.

Similarly, $\De\Phii(\om^2 \la)^{-T}$ is a solution, and we have
\begin{align*}
\De\Phii(\om^2\la)^{-T}
&=\De  \left(I+O(\tfrac1\la)\right) \om^{-2\nn} \la^{-\nn}\\
&= \left(I+O(\tfrac1\la)\right) \De\,  \om^{-2\nn} \la^{-\nn}\\
&=\left(I+O(\tfrac1\la)\right) \la^\nn \De\, \om^{-2\nn},
\end{align*}
and this must be $\Phii(\la)\De\, \om^{-2\nn}$.
\end{proof}

\no{\em Connection matrices.}

The connection matrices $D_k$ are defined by 
\[
\Phii=\Phiz_k D_k.
\]
From this, and the definitions of $R_k,P_k$ above, we have immediately:

\begin{lemma}\label{DandRandP}
$D_k=R_k D_{k+1}$ and
$D_k=P_k D_{k+\frac14}$ for all $k\in \tfrac14\Z$.
\end{lemma}

Thus, all connection matrices $D_k$ can be determined from $D_1$. 

\begin{lemma}\label{regsymmetries}  {\em (Symmetries of $D_k$ in the non-resonant case)}

\no{Cyclic symmetry: }  
$D_{k+\scriptstyle\frac12} = \Pi\,  D_k \, d_4\, \om^{-\nn}$

\no{Anti-symmetry: }  
$D_{k+1} =  \tfrac14 d_4 \, D_k{}^{-T}\, \om^{2\nn}\,  \De$
\end{lemma}

\begin{proof} The left hand side of the cyclic symmetry of Lemma \ref{regsolutionsymmetries} is
$d_4^{-1} \,\Phii(\om\la)d_4 \om^{-\nn}=
d_4^{-1} \,\Phiz_{k-\frac12}(\om\la)D_{k-\frac12}d_4 \om^{-\nn}$.
By the cyclic symmetry of Lemma \ref{solutionsymmetries}, this is 
$\Phiz_k(\la)\Pi D_{k-\frac12} d_4 \om^{-\nn}$.
The right hand side is
$\Phii(\la)=\Phiz_k(\la) D_k$.  We obtain
$D_k=\Pi D_{k-\frac12} d_4 \om^{-\nn}$.

Similarly,  the left hand side of the anti-symmetry condition 
of Lemma \ref{regsolutionsymmetries} is
\begin{align*}
\De\, \Phii(\om^2\la)^{-T}\om^{2\nn} \De 
&=
\De\, \Phiz_{k-1}(\om^2\la)^{-T} D_{k-1}^{-T} \om^{2\nn} \De
\\
&= \Phiz_k(\la) \tfrac14 d_4 D_{k-1}^{-T} \om^{2\nn} \De
\end{align*}
(using Lemma \ref{solutionsymmetries} again). The right hand side is
$\Phii(\la)=\Phiz_k(\la) D_k$, so 
$D_k=\tfrac14 d_4 D_{k-1}^{-T} \om^{2\nn} \De$.
\end{proof}

\begin{corollary}\label{corregsymmetries}   In the non-resonant case we have:

\no{Cyclic symmetry: }  
$D_k^{-1}(P_kP_{k+\frac14}\Pi)D_k = d_4^{-1} \om^\nn$

\no{Anti-symmetry: }  
$D_k=R_k \tfrac14 d_4 D_k^{-T} \om^{2\nn} \De$
\end{corollary}

\begin{proof} Lemma \ref{DandRandP} gives $D_k=P_k D_{k+\frac14} = 
P_k P_{k+\frac14} D_{k+\frac12}$ and $D_k=R_k D_{k+1}$, so
the assertions follow from the previous lemma.
\end{proof} 

The monodromy of $\Phii$ is visible directly from its definition, as we have
\[
\Phii(e^{2\pi\i}\la)=\Phii(\la) \om^{4\nn}.
\]
From $\Phii=\Phiz_k D_k$ and 
$\Phiz_k(e^{2\pi \i}\la)=\Phiz_k(\la)R_kR_{k+1}$ (above) we obtain
the \ll cyclic relation\rr
\[
R_kR_{k+1} = D_k \, \om^{4\nn} \, D_k^{-1}.
\]
Note that this can also be obtained by iterating the cyclic symmetry ($4$ times) or the anti-symmetry condition ($2$ times).  

\begin{theorem}\label{stokesomegahat} {\em (Stokes matrices of $\hat\om$)}
The Stokes matrices
$P_1,P_{1\scriptstyle\frac14}$
are given by (\ref{PP}), with:

$r_1^\R=2\cos\tfrac\pi4(2\nn_0 + 3)+
2\cos\tfrac\pi4(2\nn_1 + 1)$

$r_2^\R=-2 + 2\cos\tfrac\pi2(\nn_0+\nn_1) + 2\sin\tfrac\pi2(\nn_0-\nn_1)$
\end{theorem}

\begin{proof}  By Corollary \ref{corregsymmetries}, 
$D_k^{-1}\, P_kP_{k+\frac14}\Pi \, D_k =  \om^\nn d_4^{-1}$,
so the eigenvalues of $P_kP_{k+\frac14}\Pi$ are 
$e^{\frac\pi2 \i \nn_0},
\om^{-1} e^{\frac\pi2 \i \nn_1},
\om^{-2} e^{-\frac\pi2 \i \nn_1},
\om^{-3} e^{-\frac\pi2 \i \nn_0}
$. On the other hand, the characteristic polynomial of $P_1P_{1\frac14}\Pi$
is $x^4 - \om^{\frac52} r_1^\R x^3 - \om r_2^\R x^2 - \om^{-\frac12} r_1^\R x - 1$.
Hence
$r_1^\R= 
\om^{-\frac52} e^{\frac\pi2 \i \nn_0}+\om^{-\frac32} e^{-\frac\pi2 \i \nn_0}+
\om^{\frac12} e^{\frac\pi2 \i \nn_1}+\om^{-\frac12} e^{-\frac\pi2 \i \nn_1}
=
2\cos\left(  \tfrac\pi2 \nn_0 + \tfrac{3\pi}4 \right)+
2\cos\left(  \tfrac\pi2 \nn_1 + \tfrac{\pi}4 \right)$, and similarly
$r_2^\R=
-2 + 2\cos\tfrac\pi2(\nn_0+\nn_1) + 2\sin\tfrac\pi2(\nn_0-\nn_1)$.
\end{proof}

\begin{remark} 
In sections \ref{dkandek} and \ref{resE1},
by using the Iwasawa factorization when $k_i\ge-1$, we shall show that $r_i^\R=s_i^\R$ where $s_i^\R$ is the corresponding Stokes data  for the connection $\hat\al$.  We shall see in section \ref{conclusions} that $\nn_i=-\ga_i/2$, hence $2\nn_0+3=-\ga_0+3$, $2\nn_1+1=-\ga_1+1$. 
The above formulae for $r_i^\R$ then agree with 
the formulae for $s_i^\R$ in Corollary 4.2 of \cite{GuItLi15}.  As a side-benefit, this resolves the sign ambiguity of $s_1^\R$ in Theorem B of \cite{GuItLiXX}
and Proposition 2.1 of \cite{GuLi12}, when $k_i\ge-1$.  However, it is important to note that Theorem \ref{stokesomegahat} itself holds without any restriction on $k_i$.
\end{remark}

\no{\em An explicit solution.}

Following the strategy described at the beginning of this section, to compute the connection matrix $D_1$ (and hence all $D_k$) we shall make use of a specific solution given by an integral formula.  This is similar to the analysis of the standard Barnes integral for hypergeometric functions.

It will be convenient to use a normalized version of the operators $T_0,\hat T_0$, namely 
\[
T=(\bs-a_0)(\bs-a_1)(\bs-a_2)(\bs-a_3)-s^4,
\]
where $\bs=s \frac{\b}{\b s}$.
This gives $T_0$ if we put $s=t/\la$ (with $\la$ held constant) and
\begin{equation}\label{theai}
a_i=\tfrac4N(\al_1+\cdots+\al_i)=\nn_0-\nn_i+i
\
(i=1,2,3),
\ a_0=0.
\end{equation}
Note that $\bs=\bt=\tfrac4N\bz$ and $s^4=\tfrac{4^4}{N^4} \tfrac{z^4}{\la^4} p_0p_1p_2p_3$.  On the other hand we obtain 
$\la^{-\nn_0} \hat T_0 \la^{\nn_0}$ if we put $s=t/\la$ with $t$ held constant.
In this case $\bs=-\bla$.

\begin{proposition}\label{gzero}
For any $c<0$, the formula
\[
g_0(s)=
\int_{c-\i\infty}^{c+\i\infty}
\Ga(\tfrac {a_0}4 - t)\Ga(\tfrac {a_1}4 - t)\Ga(\tfrac {a_2}4 - t)\Ga(\tfrac {a_3}4 - t)
2^{-8t} s^{4t} dt
\]
defines a function which is holomorphic on
$\{ s\in\C^\ast \st -\tfrac{\pi}2<\argu s <\tfrac\pi2\}$. This function (and its analytic continuation to the universal covering $\tilde\C^\ast$)
satisfies $Tg_0(s)=0$.
\end{proposition}

\begin{proof} Assuming the convergence of the integral, and using the property $\Ga(k+1)=\Ga(k)k$ of the gamma function, it is easy to verify that
\[
g(s)=
\int_{c-\i\infty}^{c+\i\infty}
\Ga(\tfrac {a_0}4 - t)\Ga(\tfrac {a_1}4 - t)\Ga(\tfrac {a_2}4 - t)\Ga(\tfrac {a_3}4 - t)
e^{\i \pi \de t} s^{4t} dt
\]
satisfies $Tg(s)=0$ if $e^{\i \pi \de}=4^{-4}$, i.e.\ if $e^{\i\pi\de + 8\log 2}=1$.  
We obtain $g_0$ if we take $\ii\pi\de=-8\log 2$.  

Stirling's formula 
$\Ga(z)\sim e^{-z} e^{(z-\tfrac12)\log z} (2\pi)^{\frac12}$ (for $-\tfrac\pi2 < \argu z< \tfrac\pi2$) shows that, in fact, the integral converges for
$-\tfrac{\pi}2<\argu s <\tfrac\pi2$.   
\end{proof}

More generally,  we note that 
\[
g_n(s)=
\int_{c-\i\infty}^{c+\i\infty}
\Ga(\tfrac {a_0}4 - t)\Ga(\tfrac {a_1}4 - t)\Ga(\tfrac {a_2}4 - t)\Ga(\tfrac {a_3}4 - t)
2^{-8t} s^{4t} e^{2\pi\i n t} dt
\]
defines a solution of $Tg(s)=0$, the integral formula being valid on the sector 
$-\tfrac{\pi}2-\tfrac{\pi}2 n<\argu s <\tfrac\pi2 -\tfrac{\pi}2 n$.  This follows also from the proof above.

\begin{corollary}\label{taylor} {\em (Behaviour at $s=0$ in the non-resonant case)}  
\[
g_0(s)=C_0 u_0(s) + C_1 s^{a_1}u_1(s)+ C_2 s^{a_2}u_2(s)+ C_3 s^{a_3}u_3(s)
\]
where
\[
C_i=2\pi\ii \ 2^{-2a_i}\  
\Ga( \tfrac {a_{i+1}-a_i}4 ) \Ga( \tfrac {a_{i+2}-a_i}4 ) \Ga( \tfrac {a_{i+3}-a_i}4 )
\]
and each $u_i(s)=1+\sum_{j=1}^\infty u_{i,j}s^{4j}$ is a convergent power series in a neighbourhood of $s=0$.  Here
$a_i$ is given by (\ref{theai}) with the convention that $a_{i+4}=a_i$.
\end{corollary}

\begin{proof}
We shall compute $g_0$ by integrating along the line from $c-\ii R$ to $c+\ii R$ and closing the contour with a semicircle of radius $R$ to the right of this line.  As $R\to\infty$ the integral over the semicircle approaches zero.  Thus, by the Residue Theorem, $(-2\pi\ii)^{-1}g_0(s)$ is equal to the sum of the residues of 
$\Ga( - t)\Ga(\tfrac {a_1}4 - t)\Ga(\tfrac {a_2}4 - t)\Ga(\tfrac {a_3}4 - t)
2^{-8t} s^{4t}$.  

All poles of this function lie in the right half plane.  Indeed,
the function $\Ga(t)$ has poles at $0,-1,-2,\dots$ and the residue at $-k$ is $(-1)^k/k!$.  It follows that $\Ga(\frac{a_i}{4}-t)$ has poles at $\frac{a_i}{4},\frac{a_i}{4}+1,\frac{a_i}{4}+2,\dots$ and the residue at $\frac{a_i}{4}+k$ is  $-(-1)^k/k!$.

In the non-resonant case,  $\Ga(\tfrac {a_0}4 - t)\Ga(\tfrac {a_1}4 - t)\Ga(\tfrac {a_2}4 - t)\Ga(\tfrac {a_3}4 - t) 2^{-8t} s^{4t}$ has {\em simple} poles at $\frac{a_i}{4}+k$ where $i=0,1,2,3$ and $k=0,1,2,\dots$.  Thus $(-2\pi\ii)^{-1}g_0(s)=$
\begin{align*}
&
\sum_{k=0}^\infty - \tfrac{(-1)^k}{k!} 2^{-8k} s^{4k} 
\ \Ga(\tfrac{a_1}{4}\!-\!k)\Ga(\tfrac{a_2}{4}\!-\!k)\Ga(\tfrac{a_3}{4}\!-\!k)
\\
&\  +
\sum_{k=0}^\infty - \tfrac{(-1)^k}{k!} 2^{-8k-2a_1} s^{4k+a_1} 
\ \Ga(\!-\tfrac{a_1}{4}\!-\!k)\Ga(\!-\tfrac{a_1}{4}\!+\!\tfrac{a_2}{4}\!-\!k)\Ga(\!-\tfrac{a_1}{4}\!+\!\tfrac{a_3}{4}\!-\!k)
\\
&\  +
\sum_{k=0}^\infty - \tfrac{(-1)^k}{k!} 2^{-8k-2a_2} s^{4k+a_2} 
\ \Ga(\!-\tfrac{a_2}{4}\!-\!k)\Ga(\!-\tfrac{a_2}{4}\!+\!\tfrac{a_1}{4}\!-\!k)\Ga(\!-\tfrac{a_2}{4}\!+\!\tfrac{a_3}{4}\!-\!k)
\\
&\ +
\sum_{k=0}^\infty - \tfrac{(-1)^k}{k!} 2^{-8k-2a_3} s^{4k+a_3} 
\ \Ga(\!-\tfrac{a_3}{4}\!-\!k)\Ga(\!-\tfrac{a_3}{4}\!+\!\tfrac{a_1}{4}\!-\!k)\Ga(\!-\tfrac{a_3}{4}\!+\!\tfrac{a_2}{4}\!-\!k)
\end{align*}
Using $\Ga(v)=\Ga(v-k)\frac{(v-1)!}{(v-k-1)!}$, we obtain 
\begin{align*}
(2\pi\ii)^{-1}g_0(s)&=\Ga(\tfrac{a_1}{4})\Ga(\tfrac{a_2}{4})\Ga(\tfrac{a_3}{4})
u_0(s)
\\
&\ \ \ +
\Ga(\!-\tfrac{a_1}{4})\Ga(\tfrac{a_2}{4}\!-\!\tfrac{a_1}{4})\Ga(\tfrac{a_3}{4}\!-\!\tfrac{a_1}{4})
2^{-2a_1}s^{a_1}u_1(s)
\\
&\ \ \ +
\Ga(\!-\tfrac{a_2}{4})\Ga(\tfrac{a_1}{4}\!-\!\tfrac{a_2}{4})\Ga(\tfrac{a_3}{4}\!-\!\tfrac{a_2}{4})
2^{-2a_2}s^{a_2}u_2(s)
\\
&\ \ \ +
\Ga(\!-\tfrac{a_3}{4})\Ga(\tfrac{a_1}{4}\!-\!\tfrac{a_3}{4})\Ga(\tfrac{a_2}{4}\!-\!\tfrac{a_3}{4})
2^{-2a_3}s^{a_3}u_3(s)
\end{align*}
as required.
\end{proof}

\begin{proposition}\label{transform}
For $-\tfrac{\pi}2<\argu s <\tfrac\pi2$, 
we have $g_0(s)=$
\[
\pi\ii
\,
(\tfrac s4)^{\frac{a_1+a_2+a_3}{4}}
\!
\int_{0}^{\infty} \!\!\!\! \int_{0}^{\infty}  \!\!\!\! \int_{0}^{\infty} 
\!\! e^{-\frac s4 \left(   x_1+x_2+x_3 + \frac 1{x_1x_2x_3}\right)   }
\ 
\overset{\scriptscriptstyle 3}{\underset{\scriptstyle j=1}\Pi}
x_j^{ \frac{a_j}4 - 1  } 
\,
dx_1dx_2dx_3.
\]
\end{proposition}

\begin{proof}
Substituting $\Ga(k)=\int_0^\infty e^{-\tau} \tau^{k-1} d\tau$ into the definition of $g_0$, and putting $x=-\log\frac{s^4 2^{-8}}{\tau_0\tau_1\tau_2\tau_3}$, we obtain
\[
g_0(s)=
\int_{0}^{\infty} \!\!\!\! \int_{0}^{\infty}  \!\!\!\! \int_{0}^{\infty} \!\!\!\! \int_{0}^{\infty}
e^{-\sum_{i=0}^3 \tau_i} 
\overset{\scriptstyle 3}{\underset{\scriptstyle i=1}\Pi}
\tau_i^{\frac{a_i}4-1} 
\left(
\int_{c-\i\infty}^{c+\i\infty} e^{-xt} dt
\right)
\ 
\frac{d\tau_0}{\tau_0}d\tau_1d\tau_2d\tau_3.
\]
For the inner integral, the substitution $t=c+\ii y$ would give
\[
\int_{c-\i\infty}^{c+\i\infty} e^{-xt} dt =
\ii e^{-xc} \int_{-\infty}^{\infty} e^{-\i x y} dy=
2\ii e^{-xc} \lim_{R\to\infty} \tfrac{\sin xR}{x}.
\]
Hence $g_0(s)=$
\[
\lim_{R\to\infty}
2\ii
\int_{0}^{\infty} \!\!\!\! \int_{0}^{\infty}  \!\!\!\! \int_{0}^{\infty} \!\!\!\! \int_{0}^{\infty}
e^{-\sum_{i=1}^3 \tau_i} 
\overset{\scriptstyle 3}{\underset{\scriptstyle i=1}\Pi}
\tau_i^{\frac{a_i}4-1} 
\left(
e^{-\tau_0}
e^{-xc}  \,  \tfrac{\sin xR}{x}
\tfrac{d\tau_0}{\tau_0}
\right)
\ 
d\tau_1d\tau_2d\tau_3.
\]
Performing the $\tau_0$-integral of the quantity in parentheses, after the change of variable from $\tau_0=\frac{s^42^{-8}}{\tau_1\tau_2\tau_3} e^x$ to $x$, we obtain
\begin{align*}
\lim_{R\to\infty}
\int_{0}^{\infty}
e^{- \frac{s^42^{-8}}{\tau_1\tau_2\tau_3} e^x}
e^{-xc}
\ 
\tfrac{\sin xR}{x}
\ 
dx
&=
\lim_{R\to\infty}
\int_{0}^{\infty}
e^{- \frac{s^42^{-8}}{\tau_1\tau_2\tau_3} e^{y/R}}
e^{-yc/R}
\ 
\tfrac{\sin y}{y}
\ 
dy
\\
&=
e^{- \frac{s^42^{-8}}{\tau_1\tau_2\tau_3}} 
\int_{0}^{\infty}
\tfrac{\sin y}{y}
\ 
dy
=
\tfrac12\pi e^{- \frac{s^42^{-8}}{\tau_1\tau_2\tau_3}}.
\end{align*}
This gives
\[
g_0(s)=\pi \ii
\int_{0}^{\infty} \!\!\!\! \int_{0}^{\infty}  \!\!\!\! \int_{0}^{\infty}
e^{-\sum_{i=1}^3 \tau_i} 
\ 
\overset{\scriptstyle 3}{\underset{\scriptstyle i=1}\Pi}
\tau_i^{\frac{a_i}4-1} 
\ 
e^{- \frac{s^42^{-8}}{\tau_1\tau_2\tau_3}}
\ 
d\tau_1d\tau_2d\tau_3.
\]
Rescaling by $\tau_i=lx_i$, with $l=s/4$, we obtain the stated formula.
\end{proof}

\begin{corollary}\label{laplace} {\em (Behaviour at $s=\infty$)} For $-\tfrac{\pi}2<\argu s <\tfrac\pi2$, 
\[
g_0(s)\sim 
\ii \pi^{\frac52} 2^{-2\nn_0+\frac12} \, s^{\nn_0} e^{-s}
\]
as $s\to\infty$.
\end{corollary}

\begin{proof} We shall apply Laplace's method to the formula for $g_0$ in Proposition \ref{transform}.  For a (real) function $\phi$ on $(0,\infty)$ whose only critical point is a nondegenerate minimum at $c$, this says that
$\int_0^\infty f(x) e^{-s\phi(x)} dx \sim f(c) e^{-s\phi(c)} (2\pi)^{\frac12} s^{-\frac12}\phi^{\pr\pr}(c)^{-\frac12}$ as $s\to\infty$ with
$-\tfrac\pi2 <\argu s< \tfrac\pi2$.

The only critical point of 
$\tfrac14\left(x_1+x_2+x_3 + (x_1x_2x_3)^{-1}\right)$ in the region
$0<x_1,x_2,x_3<\infty$ occurs at $(x_1,x_2,x_3)=(1,1,1)$. The Hessian matrix is diagonalizable, with eigenvalues $\tfrac14,\tfrac14,1$.  We obtain
\begin{align*}
g_0(s)&\sim \pi\ii 
\,
(\tfrac s4)^{\frac{a_1+a_2+a_3}{4}}
\ 
e^{-s} \,  (2\pi)^{3/2} \,  s^{-3/2} \, (4^{-2})^{-1/2}
\\
&=\ii \pi^{5/2} \, s^{\frac{a_1+a_2+a_3}{4}-\frac32} \, e^{-s} \, 2^{-\frac{a_1+a_2+a_3}{2}} 2^{7/2}
\\
&=
\ii \pi^{5/2} 2^{-2\nn_0+1/2} s^{\nn_0} e^{-s},
\end{align*}
where we have used $\frac{a_1+a_2+a_3}{4} -\frac32=\nn_0$ at the last step.
\end{proof}

\no{\em Computation of connection matrices.}

To compute the connection matrix $D_1$, we can express $\Phiz_1$ in terms of $g_0$ by using Corollary \ref{laplace}, and $\Phii$ in terms of $g_0$ by using Corollary \ref{taylor}.  To do this, we need a basis of solutions of the scalar equation, and the natural candidate would be four consecutive functions such as $g_0,g_1,g_2,g_3$. However, these fail to produce a solution with the same asymptotics as $\Phiz_1$ on $\Omz_1$.  Therefore, the proof of the following theorem begins by expressing $\Phiz_1$ in terms of suitable combinations of functions $g_n$.

\begin{theorem}\label{connectionomegahat} {\em (Connection matrix for $\hat\om$ in the non-resonant case)}
\newline
If 
$\nn_i-\nn_j\notin\Z$ for all $i,j$,   
then the connection matrix $D_1$ is given by the formula
\[
D_1=
\ii \, \pi^{\frac52} \, 2^{-2\nn_0+\frac12}\,  (PK)^{-T}\, 
{\Asharp}^{-1} \hat c^{-1}
\]
where 
\[
P=
\bp
 1\!  & \!\! \om^{\frac12} r_1^\R \!\! & \hphantom{xx} & \hphantom{xx} \\
 & 1 & & \\
  & & 1 & \\
   & & & 1
\ep,
\ \ 
K=
\bp
\om^{2\nn_0} & \om^{2\nn_1-2} & \om^{2\nn_2-4} & \om^{2\nn_3-6} 
\\
\om^{\nn_0} & \om^{\nn_1-1} & \om^{\nn_2-2} & \om^{\nn_3-3}  
\\
1 & 1 & 1 & 1
\\
\om^{-\nn_0} & \om^{-(\nn_1-1)} & \om^{-(\nn_2-2)} & \om^{-(\nn_3-3)}  
\ep,
\]
and
\begin{align*}
\Asharp&=\diag(C_0,C_1 a_1,C_2 a_2 (a_2\!-\!a_1),C_3 a_3 (a_3\!-\!a_1) (a_3\!-\!a_2))
\\
\hat c&=\diag(\hat c_0,\hat c_1,\hat c_2,\hat c_3).
\end{align*}  
The $C_i$ are as in Corollary \ref{taylor}, $a_i=\nn_0-\nn_i+i$,
and the $\hat c_i$ are as in (\ref{chat}).
\end{theorem}

\begin{proof} 
\no{\em (i) Expression for $\Phiz_1$.}
The cyclic symmetry for $\Phiz_k$ (Lemma \ref{solutionsymmetries}) gives
$d_4\Phiz_1(\om^{-1}\la)\Pi = \Phiz_{\frac12}(\la) = 
\Phiz_{\frac34}(\la)P_{\frac12}^{-1} = \Phiz_1(\la) P_{\frac34}^{-1} P_{\frac12}^{-1}$.  This is a formula relating the columns of $\Phiz_1(\la)$.  
Let us denote the third column by $J(\la)$.  

Using the anti-symmetry condition, and the 
\ll elementary reality condition\rr\  
$\overline{\hat\om(\bar \la)}=\hat\om(\la)$, we find 
that
\[
P_{\scriptstyle\frac12}=
\bp
1 & & & \\
 & 1 & & \\
 &\om^{\scriptstyle\frac52} r_1^{\R} & \ \ 1\ \  &  \\
\! \om^{\scriptstyle-\frac12} r_1^{\R}\!\!  & & & \ \ 1\ \  
\ep
\ \ \ \ 
P_{\scriptstyle\frac34}=
\bp
1 & & & \\
 & 1 & & \\
\!\om r_2^{\R}\! & & \ 1\  &  \\
  & & & \ 1\ 
\ep.
\]
This gives
\[
\Phiz_1(\la)=
\bp
\vert & \vert & \vert & \vert 
\\
d_4^{-2}J(\om^2\la) \!-\! d_4^{-1} \om^{\frac52} r_1^\R J(\om\la)
&
d_4^{-1} J(\om\la)
&
J(\la)
&
d_4 J(\om^{-1}\la)
\\
\vert & \vert & \vert & \vert 
\ep.
\]
The rows of $\Phiz_1$ are obtained from the first row, which we denote by  $\phiz_1$, by differentiation (see section \ref{four}).  For $\phiz_1$ the above formula gives
\begin{equation}\label{phizeroprelim}
\phiz_1(\la)^T=P
\bp
j(\om^2\la)
\\
j(\om\la)
\\
j(\la)
\\
j(\om^{-1}\la))
\ep
,\ \ 
P=
\bp
1 & \om^{\frac12} r_1^\R & \hphantom{xxx} & \hphantom{xxx} \\
\hphantom{xx} & 1 & & \\
  & & 1 & \\
   & & & 1
\ep,
\end{equation}
where $j$ denotes the first entry of $J$.

We claim that 
\begin{equation}\label{j}
j(\la)=\la^{\nn_0} \ka_0^{-1} \hat c_0\, g_0(s)
\end{equation}
where
\[
\ka_0=\ii \pi^{\frac52} 2^{-2\nn_0+\frac12}.
\]
As $\Phiz_1$ is a fundamental solution of (\ref{phiode}), the functions 
$j(\om^2\la)$,  $j(\om\la)$,  $j(\la)$,  $j(\om^{-1}\la)$ are a basis of solutions of 
$\hat T_0 \hat y_0=0$, hence there exist $a,b,c,d$ (independent of $\la$) such that
\[
\la^{\nn_0}g_0(s)=
a  j(\om^2\la) + b j(\om\la) + c j(\la) + d j(\om^{-1}\la).
\]
By definition of $\Phiz_1$ we have 
$\Phiz_1\sim
h\Om (I+O(\la) ) e^{\frac t\la d_4}$ as $\la\to0$ in $\Omz_1$,
hence
$
(j(\om^2\la), j(\om\la), j(\la), j(\om^{-1}\la))P^T
\sim
h_0(e^{t/\la},e^{\om t/\la},e^{\om^2 t/\la},e^{\om^3 t/\la}) 
=
h_0(e^{s},e^{\i s},e^{-s},e^{-\i s}).
$
This holds in particular on the positive real axis $\la>0$, as $\la\to 0$ (i.e.\ for $s>0$, as $s\to \infty$). Here we have $\la^{\nn_0}g_0(s)\to 0$, while 
$j(\om^2\la), j(\om\la), j(\om^{-1}\la)$ are unbounded or oscillate with fixed amplitude. 
We deduce that  $a=b=d=0$, thus $\la^{\nn_0}g_0(s)=cj(\la)$. 

As $j(\la)\sim h_0 e^{-t/\la}=h_0 e^{-s}$ and
$g_0(s)\sim \ii \pi^{\frac52} 2^{-2\nn_0+\frac12} s^{\nn_0} e^{-s}$ 
(Corollary \ref{laplace}),  we obtain 
$c=h_0^{-1} \la^{\nn_0} s^{\nn_0}  \ii \pi^{\frac52} 2^{-2\nn_0+\frac12}
={\hat c_0}^{-1} \ka_0$.  This gives $j(\la)=c^{-1}\la^{\nn_0}g_0(s)
= \ka_0^{-1}\hat c_0 \la^{\nn_0} g_0(s)$, as claimed.    

From (\ref{phizeroprelim}) and (\ref{j}) we can express $\phiz_1$ in terms of $g_0$.  To facilitate comparison with $\phii$ later in the proof, however, let us convert from $s \ (=t/\la)$ to $\la$.  Since
$T=\la^{-\nn_0} \hat T_0 \la^{\nn_0}$, $g(\la)=\la^{\nn_0} g_0(s)$ is a solution of $\hat T_0 \hat y=0$, so we can write
\[
\la^{\nn_0}g_0(s)=
A_0 \la^{\nn_0}f_0(\la) + 
A_1 \la^{\nn_1-1}f_1(\la) + 
A_2 \la^{\nn_2-2}f_2(\la) + 
A_3 \la^{\nn_3-3}f_3(\la)
\]
where the $A_i$ are independent of $\la$.  From Corollary \ref{taylor} we deduce that
\[
A_0=C_0,\quad A_1=t^{a_1}C_1,\quad A_2=t^{a_2}C_2,\quad A_3=t^{a_3}C_3
\]
(and $f_i(\la)=u_i(s)$). 

Now, by (\ref{j}) we have 
\[
j(\la)=
\ka_0^{-1}\hat c_0
\left(
A_0 \la^{\nn_0}f_0+ 
A_1 \la^{\nn_1-1}f_1 + 
A_2 \la^{\nn_2-2}f_2 + 
A_3 \la^{\nn_3-3}f_3
\right).
\]
Substituting this into (\ref{phizeroprelim}) we find that
\begin{equation}\label{phizero}
\phiz_1(\la)^T= 
\ka_0^{-1} \hat c_0 \, P K A
\bp
\la^{\nn_0}f_0 \\ \la^{\nn_1-1}f_1 \\ \la^{\nn_2-2}f_2 \\\la^{\nn_3-3}f_3
\ep
\end{equation}
where 
\[
K=
\bp
\om^{2\nn_0} & \om^{2\nn_1-2} & \om^{2\nn_2-4} & \om^{2\nn_3-6} 
\\
\om^{\nn_0} & \om^{\nn_1-1} & \om^{\nn_2-2} & \om^{\nn_3-3}  
\\
1 & 1 & 1 & 1
\\
\om^{-\nn_0} & \om^{-(\nn_1-1)} & \om^{-(\nn_2-2)} & \om^{-(\nn_3-3)}  
\ep
\!\!,
\
\bp
\!A_0\!\! & & & \\
 & \!\!A_1\!\! & & \\
  & & \!\!A_2\!\! & \\
  & & & \!\!A_3\!
\ep
\!\!.
\]

\no{\em (ii) Expression for $\Phii$.}
In order to relate $\Phii$ to $g_0$, we note that the functions 
$\la^{\nn_0}f_0, \la^{\nn_1-1}f_1, \la^{\nn_2-2}f_2, \la^{\nn_3-3}f_3$
are just the solutions of the scalar o.d.e.\ $\hat T_0 \, \hat y_0=0$
provided by the Frobenius method (as the indicial roots $\nn_0, \nn_1-1, \nn_2-2, \nn_3-3$ are distinct modulo $\Z$). 
From these we obtain another fundamental solution 
\[
\Phi^{\text {Frob}}
=
\bp
- & \hat y & -
\\
- & \hat P_1 \hat y & -
\\
- & \hat P_2 \hat y & -
\\
- & \hat P_3 \hat y & -
\ep,
\quad
\hat y=(\la^{\nn_0}f_0, \la^{\nn_1-1}f_1, \la^{\nn_2-2}f_2, \la^{\nn_3-3}f_3)
\]
of (\ref{phiode}), where the $\hat P_i$ are the differential operators 
introduced at the end of section \ref{four3}.

We must have $\Phi^{\text {Frob}} = \Phii F$ 
for some constant matrix $F$.   From the above formula for $\Phi^{\text {Frob}}$ we find
$F=\diag(F_0,F_1,F_2,F_3)$ with
\begin{align*}
F_0&=1
\\
F_1&=\!-\tfrac{N}{4}(zp_1)^{-1}(\nn_1\!\!-\!\!1\!\!-\!\!\nn_0)
\\
F_2&=\left(\!-\tfrac{N}{4}\right)^2(z^2p_1p_2)^{-1}(\nn_2\!\!-\!\!2\!\!-\!\!\nn_0)(\nn_2\!\!-\!\!1\!\!-\!\!\nn_1)
\\
F_3&=
\left(\!-\tfrac{N}{4}\right)^3(z^3p_1p_2p_3)^{-1}(\nn_3\!\!-\!\!3\!\!-\!\!\nn_0)(\nn_3\!\!-\!\!2\!\!-\!\!\nn_1)(\nn_3\!\!-\!\!1\!\!-\!\!\nn_2)
\end{align*}
These expressions may be simplified further.  The identity 
$p_1\tfrac{h_1}{h_0}\dots p_i\tfrac{h_i}{h_{i-1}}=\nu^i=\left( \tfrac N4 t z^{-1}\right)^i$
gives $z^ip_1\dots p_i=\left(\tfrac{N}{4}\right)^i t^i h_0h_i^{-1}$.  
Using $h_i=\hat c_i t^{\nn_i}$ (Definition \ref{weightofhandchat}),
we obtain
$t^{-i} h_0^{-1}h_i = t^{-(\nn_0-\nn_i+i)}\hat c_0^{-1} \hat c_i$.  
Using the notation $a_i=\nn_0-\nn_i+i$ of (\ref{theai}),
we obtain
\begin{align*}
F_0&=1
\\
F_1&=t^{-a_1}\hat c_0^{-1}\hat c_1 a_1
\\
F_2&=t^{-a_2}\hat c_0^{-1}\hat c_2 a_2(a_2-a_1)
\\
F_3&=t^{-a_3}\hat c_0^{-1}\hat c_3 a_3(a_3-a_1)(a_3- a_2)
\end{align*}
which will be convenient for future use.

As the first rows are related by $\phi^{\text {Frob}}=\phii F$, we have
\begin{equation}\label{phiinfinity}
\phii(\la)^T =
F^{-1}\bp
\la^{\nn_0}f_0 \\ \la^{\nn_1-1}f_1 \\ \la^{\nn_2-2}f_2 \\\la^{\nn_3-3}f_3
\ep.
\end{equation}

\no{\em (iii) The computation of $D_1$.}  
We may now compute the connection matrix $D_1$ satisfying $\Phii(\la)=\Phiz_1(\la) D_1$.    It suffices to consider the first rows, which are related by $\phii(\la)=\phiz_1(\la) D_1$, or $D_1^{-T} {\phii}(\la)^T = {\phiz_1}(\la)^T$. 
Comparing (\ref{phizero}) and (\ref{phiinfinity}), we obtain
\[
D_1^{-T} 
=
\ka_0^{-1}\hat c_0  \, P K A F 
= 
\ka_0^{-1}  \, P K \Asharp \hat c
\]
where 
\[
\Asharp=\hat c_0 A F \hat c^{-1}
=
\diag(C_0,C_1 a_1,C_2 a_2 (a_2\!-\!a_1),C_3 a_3 (a_3\!-\!a_1) (a_3\!-\!a_2)),
\]
as required.
\end{proof}

\section{Monodromy data for $\hat\al$}\label{dkandek}

Let us recall some essential notation from \cite{GuItLi15} concerning the meromorphic system (\ref{psiodezeta}), which has poles of order $2$ at
$\ze=0$ and $\infty$.  

At $\ze=0$ we have the formal solution $\Psiz_f=P_0(I+O(\ze))e^{\frac1\ze d_4}$,  where $P_0=e^{-w}\Om$.  We take Stokes sectors $\Omz_k$ as in section \ref{omegahat}.  On $\Omz_k$ we have the fundamental solution 
$\Psiz_k$, and we define Stokes matrices by
$\Psiz_{k+1}=\Psiz_k\Sz_k$, $\Psiz_{k+\frac14}=\Psiz_k\Qz_k$.

At $\ze=\infty$ we have the formal solution $\Psii_f=P_\infty(I+O(\tfrac1\ze))e^{x^2\ze d_4}$,  where $P_\infty=e^w\Om^{-1}$. 
We take 
$
\Omi_1=\{ \ze\in\C^\ast \st -\tfrac\pi2 <\argu\ze < \tfrac{3\pi}4 \}
$
as initial Stokes sector, and then $\Omi_{k+\frac14}=e^{\frac\pi4\i}\Omi_k$, $k\in\tfrac14\Z$.  On $\Omi_k$ we have the fundamental solution 
$\Psii_k$, and we define Stokes matrices similarly by
$\Psii_{k+1}=\Psii_k\Si_k$, $\Psii_{k+\frac14}=\Psii_k\Qi_k$.

The data at $0$ is related to the data at $\infty$ by
$\Qz_k=d_4 \, \Qi_k\, d_4^{-1}$.

The connection matrices $E_k$ are defined by
\[
\Psii_k=\Psiz_k E_k.
\]
We shall focus on the connection matrix $E_1$, as the $E_k$ are related to each other by
$
d_4^{-1} E_k = \Qi_{k-\frac14} {}^{-1} \, d_4^{-1} E_{k-\frac14} \, \Qi_{k-\frac14}
$.

All this data refers to the connection form $\hat\al$.  It should be emphasized that $\hat\al$ depends on a particular (local) solution $(w_0,w_1,-w_1,-w_0)$ of (\ref{ost}). 
In our previous article \cite{GuItLi15}, by solving a Riemann-Hilbert problem,  we proved that $E_1$ takes the special value
\begin{equation}\label{globalE1}
E_1^{\text{global}} = \tfrac14C \Qi_{\frac34},
\quad 
{\scriptsize
C=
\left(
\begin{array}{c|ccc}
\!\!1 & & & \\
\hline
 & & & 1\!\\
 & & 1 & \\
 & \!\!1 & &
\end{array}
\right)
}
\end{equation}
for the solutions of (\ref{ost}) which are globally defined on $\C^\ast$. 
 
In this section, by using the monodromy data of $\hat\om$ from the previous section, and by using the Iwasawa factorization to relate $\hat\om$ and $\hat\al$, we shall compute $E_1$ for the larger class consisting of {\em radial solutions of (\ref{ost}) which are smooth near $0$.}  This means that the $w_i$ depend only on $\vert t\vert$ and are smooth on a punctured disk of the form $0<\vert t\vert < \eps$, for some $\eps>0$.  

We begin with the \ll generic case\rrr. 
This means that $\al_i> 0$, 
which allows us to assume that 
$L\vert_{z=0}=I$ and hence that
the Iwasawa factorization $L=L_\R L_+$ is valid in some neighbourhood of $z=0$.  

We also assume the \ll non-resonant condition\rr $\nn_i-\nn_j\notin\Z$.  Both of these assumptions will be relaxed in section \ref{resD1}.

\begin{theorem}\label{e}  {\em (Connection matrix for $\hat\al$ in the generic and non-resonant case)}
Let $w_0,w_1$ be radial solutions of (\ref{ost}) on some interval $0<\vert t\vert < \eps$ obtained from holomorphic data $p_i=c_i z^{k_i}$, $0\le i\le 3$. Assume that $k_i>-1$ for all $i$ (the generic condition) and
$\nn_i-\nn_j\notin\Z$ for all $i,j$ (the non-resonant condition).  
Let $\hat\al$ be the corresponding meromorphic connection form. Then
\[
E_k=\tfrac14 D_k\, \De\, \bar D_{\frac74-k}^{-1}\, d_4^3\,  C, 
\]
where $D_k$ is as in section \ref{omegahat}.  In particular this gives
\[
E_1=D_1 \De \bar D_1^{-1} d_4^{-1} E_1^{\text{\em global}},
\]
where $E_1^{\text{\em global}}$ is given in (\ref{globalE1}).
\end{theorem}

We begin by summarizing the strategy of the proof, which is motivated by a calculation in \cite{BoIt95}, with reference to the diagram below. On the left hand side we have 
$\Phii = (gL)^T$, because $(gL)^T=L^Tg$ is a solution of  (\ref{phiode}) and
has the correct type to be the canonical solution at $\la=\infty$ (in the non-resonant case). We recall that $g(\la)=\la^\nn$.
\[
\xymatrix{
 \quad \quad \quad \quad \quad \quad \quad \quad\quad\quad
  &   \Psii\ar@{-}[d]^{\text{Z}}
\\
 \Phii = (gL)^T\ar@{--}[r]^{\text{\ \ \ \  Iwasawa}}    &    (g L_\R G)^T
\\
 \quad \Phiz\ar@{-}[u]^{\text{D}}\ar@{--}[r]   \quad   &  \quad \Psiz\ar@{-}[u]_{\text{Y}}\quad
}
\]
From section \ref{four}, $(g L_\R G)^T$ is a solution of (\ref{psiode}), so it differs from the canonical solutions $\Psiz_k$, $\Psii_k$ by matrices $Y_k,Z_k$:
\[
\Psiz_k=(g L_\R G)^T Y_k,\quad \Psii_k=(g L_\R G)^T Z_k.
\]

\no Thus, $E_k=Y_k^{-1}Z_k$, so we have to calculate $Y_k,Z_k$.  It will follow from the Iwasawa factorization 
that $Y_k=D_k^{-1}$.  The reality condition gives a relation between $Y_k$ and $Z_k$, allowing us to calculate $Z_k$. This will lead to the formula of the theorem.  

We prepare for the proof by examining the relation between 
$\Phiz$ and $\Psiz$.   By definition, as $\la\to 0$ in $\Omz_k$, we have:
\begin{align*}
\Phiz_k &\sim  O_0 (I+O(\la) ) e^{\frac t\la d_4}
\\
\Psiz_k &\sim  P_0 (I+O(\la) ) e^{\frac t\la d_4}
\end{align*}
Since $O_0=h\Om$ and $P_0=e^{-w}\Om$, we have
\begin{equation}\label{oandp}
P_0 = G b^{-1} O_0
\end{equation}
where $G=\vert h\vert /h$ and $b=\vert h\vert e^w$ as in section \ref{four}.

\begin{proposition}\label{yandd}
With the assumptions of Theorem \ref{e}, we have
$Y_k=D_k^{-1}$.
\end{proposition}

\begin{proof}  
From the Iwasawa factorization $L=L_\R L_+$ and the definition of $Y_k$, we have
\[
\Psiz_k  =  (g L_\R G)^T Y_k =  (L_+^{-1} G)^T (gL)^T Y_k.
\]
Now, $(gL)^T=\Phii=\Phiz_k D_k$ and 
$\Phiz_k \sim  O_0 (I+O(\la) ) e^{\frac t\la d_4}$, so we obtain
\begin{align*}
\Psiz_k  &= G L_+^{-T} \Phiz_k D_k Y_k
\\
&\sim Gb^{-1} O_0  (I+O(\la) ) e^{\frac t\la d_4} D_k Y_k
\end{align*}
when $\la\to 0$  (using $L_+\sim b$).
By (\ref{oandp}), this is $P_0  (I+O(\la) ) e^{\frac t\la d_4} D_k Y_k$. 
But $\Psiz_k\sim P_0 (I+O(\la) ) e^{\frac t\la d_4}$, 
so we conclude that $D_k Y_k=I$.
\end{proof}

It will follow from this that (\ref{phiode}) and (\ref{psiode}) have exactly the same Stokes data at $\la=0$:

\begin{corollary}\label{sandr}
$\Sz_k=R_k=D_kD_{k+1}^{-1}$ and $\Qz_k=P_k=D_kD_{k+\frac14}^{-1}$.
\end{corollary}

\begin{proof}
We have $(gL_\R G)^T Y_{k+1} = \Psiz_{k+1} = \Psiz_k \Sz_k =
(gL_\R G)^T Y_k \Sz_k$.  Hence $\Sz_k=Y_k^{-1}Y_{k+1}$, which is $D_k D_{k+1}^{-1}$, by the proposition.  By Lemma \ref{DandRandP},
$R_k=D_k D_{k+1}^{-1}$.  Hence  $\Sz_k=R_k$.  
A similar argument applies to $\Qz_k$.
\end{proof}

Now we turn to the relation between $Y_k$ and $Z_k$.  This comes from the \ll loop group reality\rr condition for (\ref{psiode}):
\begin{equation}\label{loopgroupreality}
\overline{\Psiz_k(1/\bar\la)}=\De \Psii_{\frac74-k}(\la) \, C \, 4d_4.
\end{equation}
It is proved by observing that $\De \overline{\Psi(1/\bar\la)}$ is a solution of (\ref{psiode}) whenever $\Psi(\la)$ is, from which one sees that $\De \overline{\Psiz_f(1/\bar\la)}$ must be  $\Psii_f(\la)$ times a constant matrix
(cf.\ Step 1 in section 4 of \cite{GuItLiXX}
and the discussion before Lemma 2.4 of \cite{GuItLi15}). 

\begin{proposition}\label{yandz}
With the assumptions of Theorem \ref{e}, we have
$Z_k= \tfrac14 \De\, \bar Y_{\frac74-k}\, d_4^3\, C$.
\end{proposition} 

\begin{proof}  From $\Psiz_k=(g L_\R G)^T Y_k$ and 
$\overline{L_\R(1/\bar\la)} = \De L_\R(\la) \De$
(this is the property $c(L_\R(1/{\bar\la}))=L_\R(\la)$ of section \ref{four}; for simplicity we are omitting explicit dependence on $z,\zbar$ from the notation),
we have
\[
\overline{\Psiz_k(1/\bar\la)} =
\bar G \ \overline{L_\R(1/\bar\la)}^T  \ \overline{g(1/\bar\la)} \ \bar Y_k
= \bar G \ \De\  L_\R(\la)^T \ \De  \  \overline{g(1/\bar\la)} \  \bar Y_k.
\]
\no From $\Psii_{\frac74-k}=(g L_\R G)^T Z_{\frac74-k}$ we obtain 
\[
L_\R(\la)^T = G^{-1} \Psii_{\frac74-k}(\la) Z_{\frac74-k}^{-1} g(\la)^{-1}.
\]  
Substituting this into the previous formula gives
\begin{align*}
\overline{\Psiz_k(1/\bar\la)} 
&= 
\bar G \De G^{-1} \ \Psii_{\frac74-k}(\la) \ Z_{\frac74-k}^{-1} \ g(\la)^{-1} \De 
\overline{g(1/\bar\la)} \  \bar Y_k
\\
&=\De \ \Psii_{\frac74-k}(\la) \ Z_{\frac74-k}^{-1} \ \De \  \bar Y_k.
\end{align*}
Comparing this with (\ref{loopgroupreality}) above, 
we obtain
$C\,4d_4 =  Z_{\frac74-k}^{-1} \, \De \, \bar Y_k$, i.e.\
$Z_{\frac74-k} = \De \, \bar Y_k \, d_4^{-1} \, \tfrac14 C$.  This gives the stated formula for $Z_k$.
\end{proof}

\no{\em Proof of Theorem \ref{e}.\ \ }  By definition, $E_k=Y_k^{-1}Z_k$.
Proposition \ref{yandd} gives
$Y_k=D_k^{-1}$, and Proposition \ref{yandz} gives
$Z_k= \tfrac14 \De\, \bar Y_{\frac74-k}\, d_4^3\, C$.  Combining these, we obtain the formula for $E_k$.  Putting $k=1$ gives
$E_1=\tfrac14 D_1\, \De\, \bar D_{\frac34}^{-1}\, d_4^3\,  C$.  
By Lemma \ref{DandRandP}, $D_{\frac34}=P_{\frac34} D_1$.  
Now, $P_{\frac34} = \Qz_{\frac34} = d_4 \Qi_{\frac34} d_4^{-1}$,
and (by direct calculation) $\barQi_{\frac34} = (\Qi_{\frac34})^{-1}$
and $C \Qi_{\frac34} = \Qi_{\frac34} C$.  Hence
$E_1=\tfrac14 D_1 \De \bar D_1^{-1} \bar P_{\frac34}^{-1}    d_4^3  C
=\tfrac14 D_1 \De \bar D_1^{-1} d_4^{-1} \Qi_{\frac34} d_4    d_4^3  C
= D_1 \De \bar D_1^{-1} d_4^{-1} E_1^{\text{global}}$.
\qed

Let us make the formula in Theorem \ref{e} for $E_1$ more explicit.  This will allow us to introduce \ll connection matrix parameters\rr 
$e^\R_1,e^\R_2$ which correspond to the (normalized) parameters $c_0,c_1,c_2,c_3$ in the holomorphic data.
These parameters measure the extent to which $E_1$ differs from $E_1^{\text{global}} = \tfrac14C \Qi_{\frac34}$.

\begin{theorem}\label{explicite}
{\em (Connection matrix for $\hat\al$ in the generic and non-resonant case)}
With the assumptions of Theorem \ref{e}, 
and using the notation of Theorem \ref{connectionomegahat},
the connection matrix $E_1$ is given by the formula 
\[
E_1=
(PK)^{-T}
\bp
e^\R_1\   & & & \\
  &e^\R_2 & & \\
  & & 1/e^\R_2& \\
  & & & 1/e^\R_1
  \ep
(PK)^{T}
  \ 
 E_1^{\text{\em{global}}}
\]
where

\[
e^\R_1=
\hat c_0^{-2}
\frac
{
\Ga(  \frac{4-a_1}{4})
\Ga(  \frac{4-a_2}{4})
\Ga(  \frac{4-(a_1+a_2)}{4})
}
{
\Ga(  \frac{a_1}{4})
\Ga(  \frac{a_2}{4})
\Ga(  \frac{a_1+a_2}{4})
},
\
e^\R_2= 
\hat c_1^{-2}
\frac
{
\Ga(  \frac{a_1}{4})
\Ga(  \frac{4-a_2}{4})
\Ga(  \frac{4-(a_2-a_1)}{4})
}
{
\Ga(  \frac{4-a_1}{4})
\Ga(  \frac{a_2}{4})
\Ga(  \frac{a_2-a_1}{4})
}.
\]
Here 
$\hat c_0= c_0^{-\frac12} \, c^{ \frac{1-2\nn_0}{8} }\, \left( \tfrac N4 \right)^{\nn_0}$ and
$\hat c_1= c_2^{\frac12} \, c^{ \frac{-1-2\nn_1}{8} }\, \left( \tfrac N4 \right)^{\nn_1}$
are as in (\ref{chat}).
\end{theorem}

\begin{proof} 
Substituting 
$D_1= \ka_0\,  P^{-T} K^{-T}\, {\Asharp}^{-1} \hat c^{-1}$ 
(Theorem \ref{connectionomegahat})
into  the formula
$E_1=D_1\, \De\, \bar D_1^{-1} d_4^{-1}
E_1^{\text{global}}$
(Theorem \ref{e}), we obtain
\[
E_1 = 
P^{-T} K^{-T} {\Asharp}^{-1} \hat c^{-1}
\ 
\De
\ 
\hat c \, \Asharp  \bar K^T \bar P^T d_4^{-1} 
\ 
E_1^{\text{global}}.
\]
By direct calculation we have $\bar K = -  d_4 K \De$,  so
$\bar K^T = -\De K^T d_4$.  Moreover, $d_4 \bar P^T d_4^{-1}=P^T$.  Thus 
\[
E_1 = 
-P^{-T} K^{-T} {\Asharp}^{-1} \hat c^{-1}
\,
\De
\,
\hat c \, \Asharp \De  K^T P^T
\ 
E_1^{\text{global}}.
\]
Now, we have 
\[
{\Asharp}^{-1} \hat c^{-1} \De \hat c \, \Asharp \De
=
\diag
\left(
\frac{\hat c_3 \Asharp_3}{\hat c_0 \Asharp_0},
\frac{\hat c_2 \Asharp_2}{\hat c_1 \Asharp_1},
\frac{\hat c_1 \Asharp_1}{\hat c_2 \Asharp_2},
\frac{\hat c_0 \Asharp_0}{\hat c_3 \Asharp_3}
\right).
\]
From  Corollary \ref{taylor} and Theorem \ref{connectionomegahat}, we obtain:

$\Asharp_0= 2\pi\ii \ 
\Ga(\tfrac{a_1}{4})\Ga(\tfrac{a_2}{4})\Ga(\tfrac{a_1+a_2}{4})$

$\Asharp_1= -2\pi\ii \ 2^{2-2a_1} \ 
\Ga(\tfrac{4-a_1}{4})\Ga(\tfrac{a_2}{4})\Ga(\tfrac{a_2-a_1}{4})$

$\Asharp_2= 2\pi\ii \ 2^{4-2a_2} \ 
\Ga(\tfrac{a_1}{4})\Ga(\tfrac{4-a_2}{4})\Ga(\tfrac{4-(a_2-a_1)}{4})$

$\Asharp_3= -2\pi\ii \ 2^{6-2a_1-2a_2} \ 
\Ga(\tfrac{4-a_1}{4})\Ga(\tfrac{4-a_2}{4})\Ga(\tfrac{4-(a_1+a_2)}{4})$

\no These give the stated formulae.
\end{proof}

Thus $E_1$ is expressed entirely in terms of $s^\R_1,s^\R_2,e^\R_1,e^\R_2$ (and the positive normalization constants $c,N$ which may be chosen freely).  

We can now express the $e_i^\R$ purely in terms of holomorphic data. For completeness we give the expressions for the Stokes data $s_i^\R$ (from \cite{GuItLi15}) as well:

\begin{corollary}\label{explicitehol} 
{\em(Monodromy data in terms of holomorphic data)}
With the assumptions of Theorem \ref{e}, we have
\begin{align*}
s_1^\R &
=
-2\cos \pi \tfrac{ \pal_0}N +  2\cos \pi \tfrac{ \pal_2}N
\\
s_2^\R &
=
-2+4\cos \pi \tfrac{ \pal_0}N \, \cos \pi \tfrac{\pal_2}N
\end{align*}
and
\begin{align*}
e^\R_1&=
c_0 \, c^{ \frac{-1+2\nn_0}{4} }\, N^{-2\nn_0}
\frac
{
\Ga(  \frac{\pal_0}{N}  )
\Ga(  \frac{\pal_0+\pal_1}{N})
\Ga(  \frac{\pal_0+\pal_1+\pal_2}{N})
}
{
\Ga(  \frac{\pal_1}{N})
\Ga(  \frac{\pal_1+\pal_2}{N})
\Ga(  \frac{\pal_1+\pal_2+\pal_3}{N})
}
\\
e^\R_2&= c_2^{-1} \, c^{ \frac{1+2\nn_1}{4} }\, N^{-2\nn_1}
\frac
{
\Ga(  \frac{\pal_3}{N}  )
\Ga(  \frac{\pal_3+\pal_0}{N})
\Ga(  \frac{\pal_3+\pal_0+\pal_1}{N})
}
{
\Ga(  \frac{\pal_2}{N})
\Ga(  \frac{\pal_2+\pal_3}{N})
\Ga(  \frac{\pal_2+\pal_3+\pal_0}{N})
}.
\end{align*}
\end{corollary}

\begin{proof}  It suffices to convert the formulae in Theorem \ref{explicite}
from $a_i$ to $\al_i$, using (\ref{theai}), i.e.\ 
$a_i=\tfrac4N(\al_1+\cdots+\al_i)$.
\end{proof} 

To complete the picture, we shall express the $e_i^\R$ in terms of the asymptotic data, but first we have to explain what this data is.
In the generic case, radial solutions of (\ref{ost}) which are smooth on some interval of the form $(0,\eps)$ satisfy
\[
2w_i(t)\sim \ga_i \log\vert t\vert+\rho_i + o(1)
\]
as $t\to 0$, for some $\rho_i\in\R$.  
This can be predicted from the p.d.e.\ point of view (and we shall also prove it in the next proposition).   It is also clear from the o.d.e.\ point of view that local solutions need two parameters for each $w_i$.
However, it is difficult to find formulae for $\rho_0,\rho_1$
by differential equations methods. Theorem \ref{explicite} gives $\rho_0,\rho_1$ as follows. For completeness we repeat the formulae for $\ga_0,\ga_1$ here as well.  

\begin{proposition}\label{asymptoticdata} With the assumptions of Theorem \ref{e}, we have 
$2w_i(t)\sim \ga_i \log\vert t\vert+\rho_i + o(1)$
as $t\to 0$, where
\begin{align*}
\ga_0 &= -2\nn_0 = \tfrac1{N}(3\al_0-2\al_1-\al_2)
\\
\ga_1&=  -2\nn_1 = \tfrac1{N}(\al_0+2\al_1-3\al_2)
\end{align*}
and
\begin{align*}
\rho_0&= - 2\log \hat c_0 =
-2\log 
c_0^{-\frac12} \, c^{ \frac{1-2\nn_0}{8} }\, \left( \tfrac N4 \right)^{\nn_0}
\\
\rho_1&= - 2\log \hat c_1 =
-2\log
c_2^{\frac12} \, c^{ \frac{-1-2\nn_1}{8} }\, \left( \tfrac N4 \right)^{\nn_1}.
\end{align*}
\end{proposition}

\begin{proof} From section \ref{four} we have $w_i=\log\,  b_i/\vert h_i\vert$, and from the Iwasawa factorization (in the generic case) $b_i=1+o(1)$ as $t\to 0$, so
$2w_i=-2\log \vert h_i\vert + o(1)$.  From Definition \ref{weightofhandchat} we have 
$\vert h_i\vert =\vert\hat c_i t^{\nn_i}\vert$, hence 
$2w_i=-2\nn_i \log\vert t\vert - 2\log \hat c_i +o(1)$.
The formula for $\ga_0,\ga_1$ (which was already known from \cite{GuItLiXX}, \cite{GuItLi15}) is given by (\ref{weights}). 
The formula for $\rho_0,\rho_1$ is given by (\ref{chat}).  
\end{proof}

Using this, we can give the $e_i^\R$ (and the $s_i^\R$) purely in terms of asymptotic data:

\begin{corollary}\label{expliciteasymp} 
{\em(Monodromy data in terms of asymptotic data)}
With the assumptions of Theorem \ref{e}, we have 
\begin{align*}
s_1^\R &= -2\cos \tfrac\pi4 { (\ga_0+1)} -  2\cos \tfrac\pi4 {(\ga_1+3)}
\\
s_2^\R &= -2-4\cos \tfrac\pi4 {(\ga_0+1)} \, \cos \tfrac\pi4 {(\ga_1+3)}
\end{align*}
and
\begin{align*}
e^\R_1&=
e^{\rho_0}\,
2^{2\ga_0}
\frac
{
\Ga(  \frac{\pga_0}{4}+\frac14  )
\Ga(  \frac{\pga_0+\pga_1}{8}+\frac12 )
\Ga(  \frac{\pga_0-\pga_1}{8}+\frac34  )
}
{
\Ga(  \frac{\pga_1-\pga_0}{8}+\frac14  )
\Ga(  -\frac{\pga_0+\pga_1}{8}+\frac12  )
\Ga(  -\frac{\pga_0}{4}+\frac34  )
}
\\
e^\R_2&= e^{\rho_1}\,
2^{\ga_1}
\frac
{
\Ga(  \frac{\pga_1-\pga_0}{8}+\frac14  )
\Ga(  \frac{\pga_0+\pga_1}{8}+\frac12 )
\Ga(  \frac{\pga_1}{4}+\frac34  )
}
{
\Ga(  -\frac{\pga_1}{4}+\frac14  )
\Ga(  -\frac{\pga_0+\pga_1}{8}+\frac12  )
\Ga(  \frac{\pga_0-\pga_1}{8}+\frac34  )
}\, .
\end{align*}
\end{corollary}

\begin{proof}
The $e^\R_1,e^\R_2$ are given by combining Theorem \ref{explicite} with $\hat c_i=e^{-\frac12\rho_i}$ and the formula
$a_i=-\tfrac12 \ga_0 + \tfrac12 \ga_i + i$ from (\ref{theai}).
\end{proof}

\section{The global solutions}\label{conclusions}

We recall that points of the triangular region 
\[
\{(\ga_0,\ga_1)\in\R^2 \st 0\le \ga_0+1\le  \ga_1+3\le 4 \}
\]
parametrize solutions\footnote{This is for $n=3$, $l=0$
in equation (\ref{ost}), i.e.\ case 4a of \cite{GuItLiXX}.} 
of the tt*-Toda equations (\ref{ost}), and that in sections \ref{omegahat},\ref{dkandek} we have been considering points in the interior of this region.
For these global solutions on $\C^\ast$, the constants $\rho_0,\rho_1$ must be determined by $\ga_0,\ga_1$, and we shall now give formulae for them.  We shall deduce this from the corresponding formulae for the holomorphic data.

It follows from Theorem \ref{explicite} and  \cite{GuItLiXX}, \cite{GuItLi15}  that the global solutions are characterized, amongst solutions which are smooth near $0$, as those for which the connection matrix $E_1$ takes the special value $E_1^{\text{global}} = \tfrac14C \Qi_{\frac34}$.  From Theorem \ref{explicite} and its proof, the following equivalent conditions are immediate:
  
\begin{corollary}\label{speciale} Let $w_0,w_1$ be radial solutions of (\ref{ost}) on some interval $0<\vert t\vert < \eps$ obtained from holomorphic data $p_i=c_i z^{k_i}$, $0\le i\le 3$, with
$k_i>-1$.  Assume that
$\nn_i-\nn_j\notin\Z$ for all $i,j$. 
Then the following conditions (i)-(iii) are equivalent:

\no (i) $E_1 = \tfrac14C \Qi_{\frac34}$

\no (ii) $e^\R_1=e^\R_2=1$

\no (iii) $D_1=d_4 \bar D_1 \De$
\end{corollary} 

The holomorphic data and asymptotic data of the global solutions follow from this and Corollaries \ref{explicitehol} and {\ref{expliciteasymp}:

\begin{corollary}\label{holomorphicdata} 
{\em(Holomorphic data for global solutions)}
In the situation of Corollary \ref{speciale}, when (i)-(iii) hold, we have
\begin{align*}
c_0&=N^{2\nn_0} \ c^{\frac{1-2\nn_0}{4}}\ 
\frac
{
\Ga(  \frac{\pal_1}{N})
\Ga(  \frac{\pal_1+\pal_2}{N})
\Ga(  \frac{\pal_1+\pal_2+\pal_3}{N})
}
{
\Ga(  \frac{\pal_0}{N}  )
\Ga(  \frac{\pal_0+\pal_1}{N})
\Ga(  \frac{\pal_0+\pal_1+\pal_2}{N})
}
\\
c_2&=N^{-2\nn_1} \ c^{\frac{1+2\nn_1}{4}}\ 
\frac
{
\Ga(  \frac{\pal_3}{N}  )
\Ga(  \frac{\pal_3+\pal_0}{N})
\Ga(  \frac{\pal_3+\pal_0+\pal_1}{N})
}
{
\Ga(  \frac{\pal_2}{N})
\Ga(  \frac{\pal_2+\pal_3}{N})
\Ga(  \frac{\pal_2+\pal_3+\pal_0}{N})
}
\end{align*}
where $c=c_0c_1^2c_2$ and  $N=\al_0\!+\!2\al_1\!+\!\al_2$.
\end{corollary}

\begin{proof}
Take $e_1^\R=e_2^\R=1$ in Corollary \ref{explicitehol}.
\end{proof}

\begin{corollary}\label{tracywidom} 
{\em(Asymptotic data for global solutions)}
In the situation of Corollary \ref{speciale}, when (i)-(iii) hold, we have
\begin{align*}
2w_0(t)  &=  \ga_0\log\vert t\vert - \log\ 2^{2\ga_0}
\frac
{
\Ga(  \frac{\pga_0}{4}+\frac14  )
\Ga(  \frac{\pga_0+\pga_1}{8}+\frac12 )
\Ga(  \frac{\pga_0-\pga_1}{8}+\frac34  )
}
{
\Ga(  \frac{\pga_1-\pga_0}{8}+\frac14  )
\Ga(  -\frac{\pga_0+\pga_1}{8}+\frac12  )
\Ga(  -\frac{\pga_0}{4}+\frac34  )
}
+o(1)
\\
2w_1(t)  &=  \ga_1\log\vert t\vert - \log\ 2^{2\ga_1}
\frac
{
\Ga(  \frac{\pga_1-\pga_0}{8}+\frac14  )
\Ga(  \frac{\pga_0+\pga_1}{8}+\frac12 )
\Ga(  \frac{\pga_1}{4}+\frac34  )
}
{
\Ga(  -\frac{\pga_1}{4}+\frac14  )
\Ga(  -\frac{\pga_0+\pga_1}{8}+\frac12  )
\Ga(  \frac{\pga_0-\pga_1}{8}+\frac34  )
}
+o(1)
\end{align*}
as $t\to 0$.
\end{corollary}

\begin{proof}
Take $e_1^\R=e_2^\R=1$ in Corollary \ref{expliciteasymp}.
\end{proof}

This agrees with the formulae of  Tracy and Widom on page 699 of \cite{TrWi98}. Their conjecture concerning the region of validity of the formula was established in section 6 of our previous article \cite{GuItLi15}.

We recall that the asymptotics as $t\to\infty$ of all global solutions were computed in
Theorem 4.1 of \cite{GuItLi15}:
\begin{align*}
w_0(t)+w_1(t)\ &=\  -\ s_1^\R\  2^{-\frac34}\  
\pi^{-\frac12} \vert t\vert^{-\frac12} 
\ e^{-2\sqrt2 \vert t\vert }
+ O(\vert t\vert^{-1}e^{-2\sqrt2 \vert t\vert })
\\
w_0(t)-w_1(t)\ &=\  \ \  s_2^\R\  2^{-\frac32}\  
\pi^{-\frac12} \vert t\vert^{-\frac12}
\ e^{-4 \vert t\vert }
+ O(\vert t\vert ^{-1}e^{-4 \vert t\vert })
\end{align*}
where the Stokes parameters $s_1^\R, s_2^\R$ are given in terms of $\ga_0,\ga_1$ by Corollary \ref{expliciteasymp}.
Thus we have solved the connection problem (to find the explicit relation between asymptotic data at zero and infinity) in the generic and non-resonant case. To complete the picture we shall consider the remaining cases in sections
\ref{resD1} to \ref{appR}.   In the literature only partial results for these cases are available (in \cite{Wi08}).

\section{Resonance: the connection matrix for $\hat\om$}\label{resD1}

Theorems \ref{e} and \ref{explicite}, our main results so far, were proved under certain assumptions, namely

(i) $\al_i> 0$ for all $i$ (generic case)

(ii) $\nn_i-\nn_j\notin\Z$ for all $i\ne j$ (non-resonant case).

\no Condition (i) specifies the points $(\ga_0,\ga_1)$ which are in the interior of the triangular region parametrizing solutions.  It was needed in the Iwasawa factorization argument in 
Theorem \ref{e}. 

Condition (ii) is related to the canonical solution of the $\la$-system (\ref{phiode}) at infinity, which is, in general, of the form 
\[
\Phii(\la)=\left(I+O(\tfrac1\la)\right)\la^\nn\la^\NN,
\]
for some nilpotent matrix $\NN$. Condition (ii) ensures that $\NN=0$.  

In fact, condition (i) alone is sufficient to ensure that $\NN=0$.  This is because the  explicit solution given in Proposition \ref{gzero} and its expansion in Corollary \ref{taylor} 
are valid for {\em all} interior points (we have
$0<\tfrac{a_1}4<\tfrac{a_2}4<\tfrac{a_3}4<1$ for such points,
so the integrand has only simple poles). 
On the other hand it turns out that all boundary points are resonant in the sense that $\NN$ is nonzero. 

To investigate these boundary points, it will be convenient to divide the boundary into six disjoint components,  three (open) edges and three vertices, as shown in Figure \ref{6components}.  

\begin{figure}[h]
\begin{center}
\includegraphics[scale=0.5, trim= 40  170  0  160]{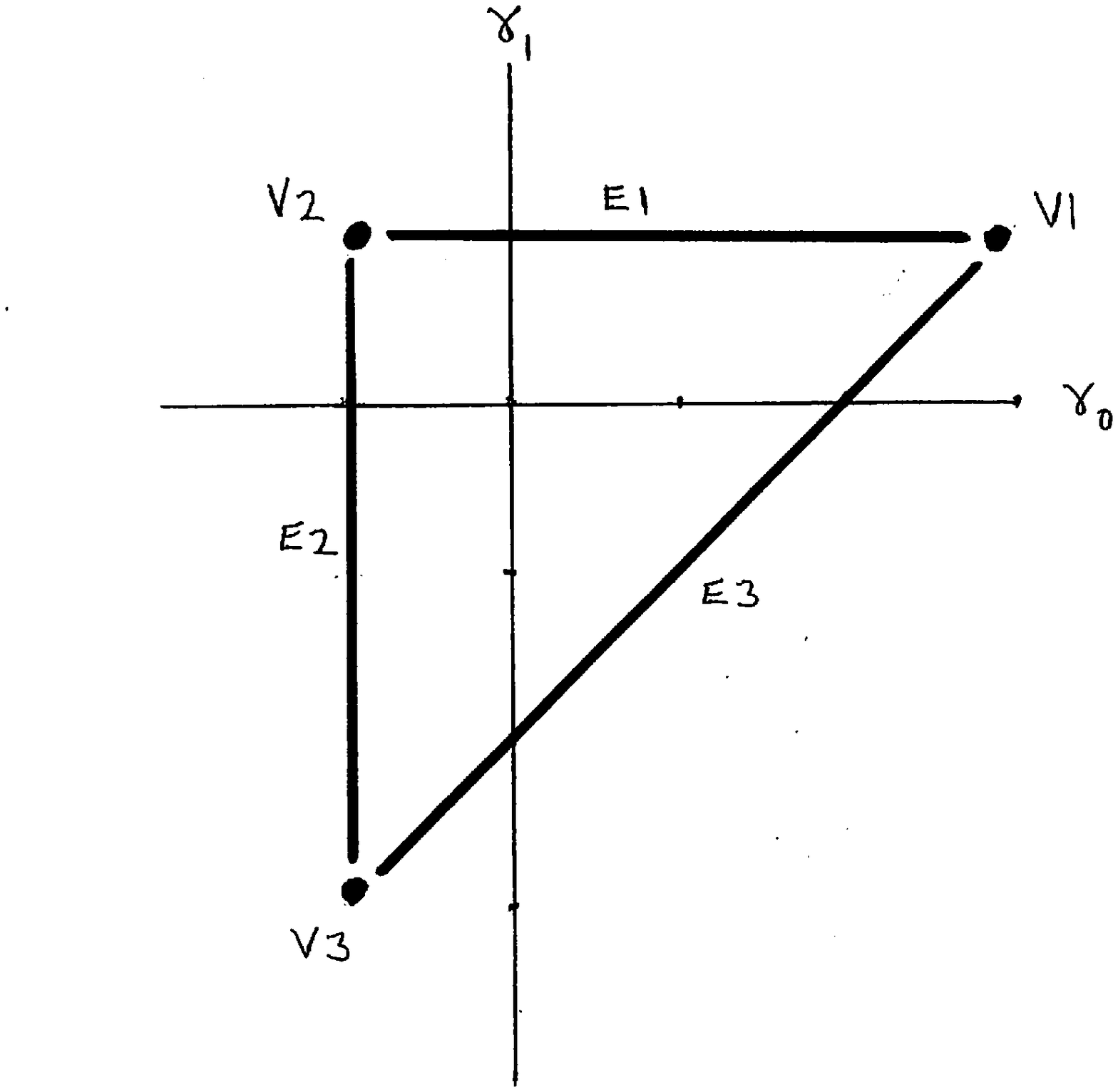}
\end{center}
\caption{The six boundary components.}\label{6components}
\end{figure}

Our strategy will be the same as in sections \ref{omegahat}-\ref{conclusions},
but there are significant new aspects, which we summarize briefly now.  We begin by giving the value of the matrix $\NN$ for each boundary component.  We shall find the same formulae for the Stokes matrices as in the non-resonant case.
However, the connection matrix $D_1$ for $\hat\om$ depends very much on the boundary component, and the current section is devoted to computing this.
In section \ref{resE1} we shall use $D_1$ to compute the corresponding connection matrix $E_1$ for $\hat\al$.  
The Iwasawa factorization argument of section \ref{dkandek} does not apply to the boundary points,  but we shall use a modified argument based on \cite{DoGuRo10}.  This will allow us to calculate $E_1$ in the resonant cases. 
After that (section \ref{resGLOBAL}), we shall give the holomorphic and asymptotic data for all local solutions near $t=0$, and characterize the global solutions in terms of this data, as in the non-resonant case. 

Let us begin with the canonical solution $\Phii(\la)=\left(I+O(\tfrac1\la)\right)\la^\nn\la^\NN$ of (\ref{phiode}) at infinity.
In Table \ref{t1} we give $\nn=\diag(\nn_0,\nn_1,\nn_2,\nn_3)$  for each of the six components, together with the parameters $a_1,a_2,a_3$.  The latter were introduced in (\ref{theai}), and will be convenient here also.

\begin{table}[h]
\renewcommand{\arraystretch}{1.5}
\begin{tabular}{c||c|c|c}
& $(\ga_0,\ga_1)$ & $(\nn_0,\nn_1,\nn_2,\nn_3)$& $(a_1,a_2,a_3)$
\\
\hline
E1 & $(\ga_0,1)$ & $(\nn_0,-\tfrac12, \tfrac12, -\nn_0)$ &
$(\tfrac a2,\tfrac a2,a), 0\!<\!a\!=\!2\nn_0\!+\!3\!<\!4$
\\
E2 & $(-1,\ga_1)$ & $(\tfrac12, \nn_1, -\nn_1,-\tfrac12 )$ &
$(a,4\!\!-\!\!a,4), 0\!<\!a\!=\!\tfrac32\!-\!\nn_1\!<\!2$
\\
E3 & $(\ga_0,\ga_0\!\!-\!\!2)$ & $\!(\nn_0,\nn_0\!+\!1, -\nn_0\!\!-\!\!1, -\nn_0)\!$ &
$(0,a,a), 0\!<\!a\!=\!2\nn_0\!+\!3\!<\!4$
\\
\hline
V1 & $(3,1)$ & $(-\tfrac32,-\tfrac12, \tfrac12, \tfrac32)$ &
$(0,0,0)$
\\
V2 & $(-1,1)$ & 
$(\tfrac12,-\tfrac12, \tfrac12, -\tfrac12)$ & 
$(2,2,4)$
\\
V3 & $(-1,-3)$ & 
$(\tfrac12,\tfrac32, -\tfrac32, -\tfrac12)$ & 
$(0,4,4)$
\end{tabular}
\bigskip
\caption{The six boundary components.}
\label{t1}
\end{table}

We shall compute $\NN$ by comparing $\Phii$ with the solution $\Phi^{\text {Frob}}$ of (\ref{phiode}) which is obtained by applying 
the Frobenius method to the scalar o.d.e.\ $\hat T_0 \, \hat y_0=0$.  This is
of the form
\[
\Phi^{\text {Frob}}
=
\bp
- & \hat y & -
\\
- & \hat P_1 \hat y & -
\\
- & \hat P_2 \hat y & -
\\
- & \hat P_3 \hat y & -
\ep,
\quad
\hat y=(\hat y^{(0)},\hat y^{(1)},\hat y^{(2)},\hat y^{(3)})
\]
where the $\hat P_i$ are the differential operators 
defined in section \ref{four},
and the $\hat y^{(i)}$ are of the form
\begin{equation}\label{matrixfrob}
\bp
\hat y^{(0)} \\ \hat y^{(1)} \\\hat y^{(2)} \\\hat y^{(3)} 
\ep
=
\la^{E^T}
\bp
\la^{\nn_0}& & &
\\
& \la^{\nn_1-1} & & 
\\
& & \la^{\nn_2-2} & 
\\
& & &\la^{\nn_3-3}
\ep
\bp
f_0 \\ f_1 \\  f_2 \\ f_3
\ep
\end{equation}
where $E$ is nilpotent and $f_0,f_1,f_2,f_3$ are holomorphic near $\la=\infty$.

\begin{proposition}\label{thecases} 
We have $\Phi^{\text{\em{Frob}}} = \Phii FU$ where $F$ is diagonal and $U$ is unipotent. The nilpotent matrix $M$ in the canonical solution $\Phii(\la)=\left(I+O(\tfrac1\la)\right)\la^\nn\la^\NN$ is 
$\NN=FU\,E\,(FU)^{-1}=FEF^{-1}= -\hat c E \hat c^{-1}$. The
matrices $E$ and $F$ are given in Table \ref{t2}, and the matrix $U$ is given in the Appendix.
{\em
\begin{table}[h]
\renewcommand{\arraystretch}{1.5}
\begin{tabular}{c||c|c}
& $E$  & $\hat c_0 \hat c_i^{-1}F_i \ (0\le i\le 3)$
\\
\hline
 E1  & $E_{1,2}$ 
 & $1,\tfrac12a t^{-\frac a2},-\tfrac12a t^{-\frac a2},\tfrac14a^3 t^{-a}$
\\
 E2  & $E_{3,0}$ 
 & $-4a_1a_2t^{-4},a_1 t^{-a},a_2(a_2  \!-\!  a_1) t^{a-4},4a_1a_2 t^{-4}$
\\
 E3  & $E_{0,1}  +  E_{2,3}$ 
 & $1,-1,a^2 t^{-a},-a^2 t^{-a}$
\\
\hline
 V1  &  $E_{0,1}  +  E_{1,2}  +  E_{2,3}$ 
& $1,-1,1,-1$
\\
 V2  & $E_{3,0}  +  E_{1,2}$ 
 & 
$-16t^{-4},2 t^{-2},-2  t^{-2},16  t^{-4}$
\\
 V3  & $E_{3,0}  +  E_{0,1}  +  E_{2,3}$ 
 & $16t^{-4}, -16t^{-4},16  t^{-4},-16  t^{-4}$
\end{tabular}
\bigskip
\caption{The matrices $E,F$ in Proposition \ref{thecases}. $E_{i,j}$ ($0\le i,j\le 3$) denotes the matrix with $1$ in the $(i,j)$ position and $0$ elsewhere.}
\label{t2}
\end{table}
}
\end{proposition}

\begin{proof} 
We shall just give the details
for the case (E1), as the other cases are very similar.  
Here we have 
\[
\hat T_0=
(\bla-\nn_0)(\bla+\tfrac32)^2(\bla+\nn_0+3)
- \tfrac{4^4}{N^4} \tfrac{z^4}{\la^4} p_0p_1p_2p_3,
\]
with indicial roots $\nn_0,-\tfrac32,-\tfrac32,-\nn_0-3$.
The Frobenius method gives 
a unique basis of solutions $\hat y$ of the form 
\begin{align*}
\hat y^{(0)}&=\la^{\nn_0}f_0
\\
\hat y^{(1)}&=
\la^{-\frac32}f_1
\\
\hat y^{(2)}&=
\la^{-\frac32}f_1\log\la+\la^{-\frac32}f_2
\\
\hat y^{(3)}&=
\la^{-\nn_0-3}f_3
\end{align*}
where
$f_i(\la)=1+\sum_{j=1}^\infty f_{i,j}\la^{-4j}$ for $i=0,1,3$ and 
$f_2(\la)=\sum_{j=1}^\infty f_{2,j}\la^{-4j}$.  This assumes that $\nn_0\ne -1, -\tfrac12, 0$. However, as in the interior case, the explicit formula for $g_0$ shows that no further logarithms arise for these points.  In the notation of (\ref{matrixfrob}), this shows that
 $E=E_{1,2}$ (column 1 of Table \ref{t2}).

Explicit calculation using the differential operators $\hat P_i$
shows that 
$\Phi^{\text {Frob}}$ has the form
\[
F 
\left(I+O(\tfrac1\la)\right)
\la^\nn\la^{E_{1,2}}
U,
\quad
U=
\bsp
\vphantom{-\frac2a}
1\, & & &
\\
 & 1 & -\frac2a &
\\
\vphantom{-\frac2a}
 & & 1 &
\\
\vphantom{-\frac2a}
 & & & 1
\esp
\]
where
\begin{align*}
F_0&=1
\\
F_1&=\!-\tfrac{N}{4}(zp_1)^{-1}(-\tfrac32-\nn_0)=t^{-\frac a2} \, \hat c_0^{-1}\hat c_1\, \tfrac12 a
\\
F_2&=\left(\!-\tfrac{N}{4}\right)^2(z^2p_1p_2)^{-1}(-\tfrac32-\nn_0)
=t^{-\frac a2} \, \hat c_0^{-1}\hat c_2\, (-\tfrac12 a)
\\
F_3&=
\left(\!-\tfrac{N}{4}\right)^3(z^3p_1p_2p_3)^{-1}
(-3-2\nn_0)(-\tfrac32-\nn_0)^2
=
t^{-a} \, \hat c_0^{-1}\hat c_3\, \tfrac14 a^3.
\end{align*}
(cf.\ the proof of Theorem \ref{connectionomegahat}).
This gives the matrix $F$ (column 2 of Table \ref{t2}).

It follows that
$\Phi^{\text {Frob}}=
\left(I+O(\tfrac1\la)\right)\la^\nn\la^\NN FU=
\Phii FU$, with
$\NN=FUE_{1,2}(FU)^{-1}
=FE_{1,2}F^{-1}
 = -\hat c E_{1,2} \hat c^{-1}$.  Thus we have computed $\NN$ and it has the properties stated in the proposition.
\end{proof}

Now we have to modify Lemma \ref{regsolutionsymmetries}, Lemma \ref{regsymmetries}, and Corollary \ref{corregsymmetries} to take account of $\NN$. 

\begin{lemma}\label{E1solutionsymmetries}   \ 

\no{Cyclic symmetry: }  
$d_4^{-1} \,\Phii(\om\la) \om^{-\NN} \om^{-\nn} d_4=\Phii(\la)$

\no{Anti-symmetry: }  
$\De\, \Phii(\om^2\la)^{-T} \om^{2\NN^T} \om^{2\nn} \De=\Phii(\la)$
\end{lemma}

\begin{proof} 
If $\Phii(\la)=(I+O(\tfrac1\la))\la^\nn \la^\NN$ is the canonical solution, then (cf.\ Lemma 3.5) $d_4^{-1}\Phii(\om\la)$ is also a solution of (\ref{phiode}), hence
$d_4^{-1}\Phii(\om\la)=\Phii(\la)C$ for some constant $C$. 
Now, 
\begin{align*}
d_4^{-1}\Phii(\om\la) &= d_4^{-1} (I+O(\tfrac1\la))\om^\nn \la^\nn \om^\NN \la^\NN
\\
&= (I+O(\tfrac1\la)) \la^\nn  (d_4^{-1}\om^\nn)  \la^\NN   \om^\NN
\\
&= (I+O(\tfrac1\la))  \la^\nn  \la^\NN   (d_4^{-1}\om^\nn)  \om^\NN  
\end{align*}
because 
$d_4^{-1}\om^\nn$
commutes with $\la^\NN$.   Hence $C=d_4^{-1}\om^\nn  \om^\NN$, and this gives the cyclic symmetry.  
Similarly, using $\De \la^{-2\nn} \De = \la^{2\nn}$ and 
$\De \la^{-2\NN^T} \De = \la^{-2\NN}$, we have
\begin{align*}
\De\, \Phii(\om^2\la)^{-T} &= \De (I+O(\tfrac1\la)) \om^{-2\nn} \la^{-\nn} \om^{-2\NN^T} \la^{-\NN^T} 
\\
&=(I+O(\tfrac1\la)) \om^{2\nn} \la^{\nn} \om^{-2\NN}  \la^{-\NN} \De
\\
&=(I+O(\tfrac1\la))  \la^{\nn} \la^{\NN} \De\, \om^{-2\nn} \om^{-2\NN^T} 
\end{align*}
where we have used $\om^{2\nn} \la^{-\NN} = \la^{\NN} \om^{2\nn}$ at the last step.
This gives $\De\, \Phii(\om^2\la)^{-T} = \Phii(\la) \De\, \om^{-2\nn} \om^{-2\NN^T}$, which is the anti-symmetry condition.
\end{proof}

\begin{lemma}\label{E1symmetries}  \ 

\no{Cyclic symmetry: }  
$D_{k+\scriptstyle\frac12} = \Pi\,  D_k \, \om^{-\NN}\om^{-\nn} d_4$

\no{Anti-symmetry: }  
$D_{k+1} =  \tfrac14 d_4 \, D_k{}^{-T}\, \om^{2\NN^T}\om^{2\nn} \De$
\end{lemma}

\begin{proof} 
This follows from Lemma \ref{E1solutionsymmetries} (cf.\ the proof of Lemma \ref{regsymmetries}).
\end{proof}

\begin{corollary}\label{corE1symmetries}   \ 

\no{Cyclic symmetry: }  
$D_k^{-1}(P_kP_{k+\frac14}\Pi)D_k = d_4^{-1} \om^\nn \om^\NN$

\no{Anti-symmetry: }  
$D_k=R_k \tfrac14 d_4 D_k^{-T} \om^{2\NN^T} \om^{2\nn} \De$
\end{corollary}

\begin{proof} 
This follows from Lemma \ref{E1symmetries} 
(cf.\  the proof of Corollary \ref{corregsymmetries}).
\end{proof}

The fact that $P_kP_{k+\frac14}\Pi$ is conjugate to $d_4^{-1} \om^\nn \om^\NN$
(and hence has the same eigenvalues as $d_4^{-1} \om^\nn$)
implies that the Stokes data is given by the same formulae as before.
Therefore, no modifications to Theorem \ref{stokesomegahat} are needed.

Next we turn to the main task of this section, the calculation of $D_1$ in the boundary cases.
Proposition \ref{gzero} still gives a solution
of $Tg_0(s)=0$, but Corollary \ref{taylor} (the $s$-expansion of $g_0(s)$) must be modified because $g_0(s)$ now has higher order poles.

Recall that if $g_0(s)$ is a solution of $Tg_0(s)=0$, then $g(\la)=\la^{\nn_0}g_0(s)=\la^{\nn_0}g_0(t/\la)$ is a solution of 
$\hat T_0 g(\la)=0$.  

\begin{proposition}\label{Rtaylor} We have
$\la^{\nn_0}g_0(s)=\tilde A_0 \hat y^{(0)}+\tilde A_1 \hat y^{(1)}+\tilde A_2 \hat y^{(2)}+\tilde A_3 \hat y^{(3)}$
where $\tilde A_0,\tilde A_1,\tilde A_2,\tilde A_3$ are given in terms of $a_1,a_2,a_3$ in the Appendix.
\end{proposition}

\begin{proof}  This is a routine calculation using the Laurent expansion of the gamma function.  For reference we give the latter at the end of this section as Lemma \ref{laurent}.
We give the case (E1) in full here, referring to the Appendix for a complete list of results, which are obtained by exactly the same method.  In this case we have
$a_1=a_2=\tfrac12 a_3=\tfrac12 a$, so 
\[
g_0(s)=
\int_{c-\i\infty}^{c+\i\infty}
\Ga( - t)\Ga(\tfrac {a}8 - t)^2\Ga(\tfrac {a}4 - t)
2^{-8t} s^{4t} dt.
\]
Let us write $X=2^{-8} s^{4}$.  As in Corollary \ref{taylor}, 
$(-2\pi\ii)^{-1}g_0(s)$ is equal to the sum of the residues of 
$\Ga( - t)\Ga(\tfrac {a}8 - t)^2\Ga(\tfrac {a}4 - t)X^t$, which has simple poles at
$t=k$, $t=\tfrac a4+k$ and double poles at $t=\tfrac a8+k$ ($k=0,1,2,\dots$).  

It is well known (and follows from Lemma \ref{laurent})
that 
\[
\Ga(-t)=\frac{r_k}{t-k} + O(1)
\]
near $t=k$ where 
\[
r_k=-{(-1)^k}/{k!}
\]
so the contribution from the poles $t=k$ of $\Ga( - t)$ 
is
$
\sum_{k=0}^\infty r_k G(k) X^k
$
where 
$G(t)=\Ga(\tfrac {a}8 - t)^2\Ga(\tfrac {a}4 - t)$.

Similarly, the contribution from the poles $t=\tfrac a4 + k$ of $\Ga(\tfrac a4 - t)$ 
is
$
\sum_{k=0}^\infty r_k H(\tfrac a4 + k) X^{\frac a4+k}
$
where 
$H(t)=\Ga(-t)\Ga(\tfrac {a}8 - t)^2$. 

To calculate the residues at the double poles $t=\tfrac a8 + k$ of $\Ga(\tfrac {a}8 - t)^2$ we use the fact (again from Lemma \ref{laurent}) that
\[
\textstyle
\Ga(\tfrac a8-t)=\frac{r_k}{t - (\frac a8+k)} + l_k + O(t - (\frac a8+k))
\]
near $t=\frac a8 + k$ where 
\[
\text{$l_k=r_k(\euler-h_k)$ with $h_k=1+\tfrac 12+\cdots+\tfrac1k$
}
\]
and $\euler$ denotes the Euler-Mascheroni constant. (We put $h_0=0$.)
Hence
\[
\textstyle
\Ga(\frac {a}8 - t)^2=
\frac{r_k^2}{(t - (\frac a8+k))^{2}} + \frac{2r_k l_k}{t - (\frac a8+k)} + 
O(1).
\]
Let us write $F(t)=\Ga(-t)\Ga(\tfrac a4-t)$.  Then the Taylor expansion of $F(t)X^t$ at $\tfrac a8 + k$ is of the form
\[
F(t)X^t=x_k + y_k (t - (\tfrac a8+k)) + O( (t - (\tfrac a8+k))^2 ),
\]
where
$x_k=F(\tfrac a8+k)X^{\frac a8 + k}$, 
$y_k=F^\pr(\tfrac a8+k)X^{\frac a8 + k} + F(\tfrac a8+k)X^{\frac a8 + k} \log X$.  
The residue of 
$\Ga( - t)\Ga(\tfrac {a}8 - t)^2\Ga(\tfrac {a}4 - t)X^t$
at $\tfrac a8 + k$ is $2r_kl_kx_k + r_k^2 y_k=2r_k^2(\euler - h_k)x_k + r_k^2 y_k$, and the contribution from the poles $t=\tfrac a8+k$ is
\[
\sum_{k=0}^\infty \  2r_k^2 (\euler-h_k) F(\tfrac a8+k) X^{\frac a8 + k}
+
r_k^2  F^\pr(\tfrac a8+k)  X^{\frac a8 + k}
+
r_k^2  F(\tfrac a8+k) X^{\frac a8 + k}\log X.
\]

Summing all three contributions, and multiplying by $\la^{\nn_0}$,
we obtain a solution of 
$\hat T_0 \hat y=0$, which therefore must be a linear combination of
the Frobenius solutions (\ref{matrixfrob}). Let us write
$\la^{\nn_0}g_0(s)=\tilde A_0 \hat y^{(0)}+\tilde A_1 \hat y^{(1)}+\tilde A_2 \hat y^{(2)}+\tilde A_3 \hat y^{(3)}$.
Using $X=2^{-8} t^4 \la^{-4}$, and $\log X= \log 2^{-8} t^4 - 4 \log\la$,
by inspection we obtain the formulae for the $\tilde A_i$ which are listed in the Appendix.
\end{proof}

The analysis of equation (\ref{phiode}) at $\la=0$ remains the same as in section 
\ref{omegahat}. Proposition \ref{transform} and Corollary \ref{laplace} apply equally well to the boundary case, so we may proceed now to the analogue of 
Theorem \ref{connectionomegahat}:

\begin{theorem}\label{Rconnectionomegahat} The connection matrix $D_1$ is given by the formula
\[
D_1=
\ii \, \pi^{\frac52} \, 2^{-2\nn_0+\frac12}\,  (P\tilde K)^{-T}\, 
{\tilde \Asharp}{}^{-T} \hat c^{-1}
\]
where $P,\hat c$ are as in Theorem \ref{connectionomegahat}, and
$\tilde K, \tilde\Asharp$ are given in the Appendix.
The matrices $P, \tilde K, \tilde \Asharp$ depend on $m$, but not on $\hat c$.
\end{theorem}

\begin{proof} 
\no{\em (i) Expression for $\Phiz_1$.}
As in the proof of Theorem \ref{connectionomegahat} we have
\begin{equation*}
\phiz_1(\la)^T=P
\bp
j(\om^2\la)
\\
j(\om\la)
\\
j(\la)
\\
j(\om^{-1}\la))
\ep,
\end{equation*}
where
$j(\la)= \ka_0^{-1}\hat c_0 \la^{\nn_0} g_0(s)$ and
$\ka_0=\ii \pi^{\frac52} 2^{-2\nn_0+\frac12}$.
By Proposition \ref{Rtaylor}, 
\begin{align*}
j(\la)&=\ka_0^{-1}\hat c_0 (\tilde A_0,\tilde A_1,\tilde A_2,\tilde A_3)
\bp
\hat y^{(0)} \\ \hat y^{(1)} \\\hat y^{(2)} \\\hat y^{(3)} 
\ep
\\
&= \ka_0^{-1}\hat c_0 (\tilde A_0,\tilde A_1,\tilde A_2,\tilde A_3)
\la^{E^T}
\la^{\nn^\pr}
\bp
f_0 \\ f_1 \\  f_2 \\ f_3
\ep
\end{align*}
where $\nn^\pr = (\nn_0,\nn_1-1,\nn_2-2,\nn_3-3)$.
From this we obtain, for any $k\in\Z$,
\[
j(\om^k\la)
= \ka_0^{-1}\hat c_0 (\tilde A_0,\tilde A_1,\tilde A_2,\tilde A_3)
(\om^k)^{E^T} (\om^k)^{\nn^\pr} 
\bp
\hat y^{(0)} \\ \hat y^{(1)} \\\hat y^{(2)} \\\hat y^{(3)} 
\ep,
\]
as $f_i(\om\la)=f_i(\la)$ and
$\la^{E^T}$ commutes with $\om^{\nn^\pr} (=d_4^{-1}\om^\nn)$.

Let us apply the identity
\begin{equation*}
\bsp \!a_1, & a_2, & \dots\, , & a_r\! \esp
\bsp
b_r & & & 
\\
\vdots & \ddots & &
\\
b_2 & \ddots & \ddots & 
\\
b_1 & b_2 & \cdots & b_r
\esp
=
\bsp \! b_1, & b_2, & \dots\, , & b_r \! \esp
\bsp
a_r & & & 
\\
\vdots & \ddots & &
\\
a_2 & \ddots & \ddots & 
\\
a_1 & a_2 & \cdots & a_r
\esp
\end{equation*}
to the term $(\tilde A_0,\tilde A_1,\tilde A_2,\tilde A_3)
(\om^k)^{E^T}$ (when $E^T$ is lower triangular; otherwise an analogous identity holds). 
We obtain 
\[
(\tilde A_0,\tilde A_1,\tilde A_2,\tilde A_3)(\om^k)^{E^T}=
o_k \,\tilde A
\]
for certain $1\times 4$, $4\times 4$ matrices $o_k, \tilde A$
where $o_k$ depends only on $\om$ (and $k$)  and 
$\tilde A$ depends only on $\tilde A_0,\tilde A_1,\tilde A_2,\tilde A_3$.
It follows that
\begin{equation*}
\phiz_1(\la)^T= 
\ka_0^{-1} \hat c_0 \, P \tilde K
\tilde A
\bp
\hat y^{(0)}
\\
\hat y^{(1)}
\\
\hat y^{(2)}
\\
\hat y^{(3)}
\ep
\end{equation*}
where 
\[
\tilde K
=
\bp
- & o_2 (\om^2)^{\nn^\prime} & -
\\
- & o_1 (\om)^{\nn^\prime} & -
\\
- & o_0 & -
\\
- & o_{\!-1} (\om^{-1})^{\nn^\prime} & -
\ep.
\]
In this way, we obtain the matrices $\tilde K,\tilde A$ listed in the Appendix.

For example, in the case (E1), where
$\nn^\pr=(\nn_0,-\tfrac32,-\tfrac32,-\nn_0-3)$,
we apply the above identity (with $r=2$) to the sub-matrices
\[
(\tilde A_1, \tilde A_2)
\bp
1 & 0
\\
\log\om^k & 1
\ep
=
(\log\om^k,1)
\bp
\tilde A_2 & 0
\\
\tilde A_1 & \tilde A_2
\ep,
\]
and this gives
$o_k=(1,\log\om^k,1,1)$ and
$\tilde A=
\bp
\tilde A_0 & & & 
\\
 & \tilde A_2 & & 
\\
 & \tilde A_1  & \tilde A_2   & 
 \\
 & & & \tilde A_3 
\ep$.  

\no{\em (ii) Expression for $\Phii$.}
From Proposition \ref{thecases}  we have 
\[
\phii(\la)^T  = F^{-1}U^{-T} \phi^{\text {Frob}} (\la)^T=
F^{-1}U^{-T}
\bp
\hat y^{(0)}
\\
\hat y^{(1)}
\\
\hat y^{(2)}
\\
\hat y^{(3)}
\ep.
\]

\no{\em (iii) The computation of $D_1$.}  
Inserting the results of (i) and (ii) into 
$D_1^{-T} {\phii}(\la)^T = {\phiz_1}(\la)^T$, we obtain
\[
D_1^{-T} = \ka_0^{-1}\hat c_0 \, P\tilde K \tilde A\,  U^T \!F = 
 \ka_0^{-1} \, P\tilde K \tilde \Asharp \hat c
\]
where $\tilde\Asharp=\hat c_0  \tilde A\, U^T \!F \hat c^{-1}$.
This gives the stated formula for $D_1$.
\end{proof}

We end this section by recalling the Laurent expansion of the gamma function, which was used in the proof 
of Proposition \ref{Rtaylor}.  (The first two terms suffice for all cases except the vertices (V1) and (V3); in those cases the first four terms are needed.)

\begin{lemma}\label{laurent}   
The Laurent expansion of $\Ga(-t)$ at its simple pole $t=0$ is
\begin{align*}
\Ga(-t)&=-\Ga(1)t^{-1}
+\Ga^\pr(1)
-\tfrac12\Ga^{\pr\pr}(1)t
+\tfrac1{3!} \Ga^{\pr\pr\pr}(1)t^2
+ O(t^3)\\
&=
\ \ r_0 t^{-1} + l_0 + m_0 t + n_0 t^2 + O(t^3)
\end{align*}
where
$r_0 =-1$,
$l_0 =-\euler$,
$m_0 =-\tfrac12(\euler^2+\tfrac16 \pi^2)$,
$n_0 =\tfrac16(-\euler^3-\tfrac12\pi^2\euler-2\zeta(3))$.
Here
$\euler=\lim_{n\to\infty} \ 
1+\tfrac 12+\cdots+\tfrac1n - \log n$ (the Euler-Mascheroni constant), and
$\zeta(3)=\sum_{k=1}^\infty k^{-3}$.
\end{lemma}

\begin{proof} This follows from the Taylor expansion
$-\Ga(-t+1)=t\Ga(-t)=r_0 + l_0 t + m_0 t^2 + n_0 t^3 + O(t^4)$
together with the well known values of the derivatives of the gamma function at $t=1$.
\end{proof}

\section{Resonance: the connection matrix for $\hat\al$}\label{resE1}

We have to modify the Iwasawa factorization argument in Theorem \ref{e}.
There we assumed $k_i>-1$, 
which allowed us to choose a fundamental solution $L$ such that $L\vert_{z=0}=I$.  But on the boundary we have (some) $k_i=-1$, and we need a new argument when this happens.  

\begin{proposition}\label{Rzfrob} There is a unique fundamental solution $L$ of the o.d.e.\ $L^{-1}L_z = \tfrac1\la \eta$ 
of the form $L=eS$, where
\[
e= z^{\frac1\la \MM^T},
\quad
\MM=
-\tfrac N4\NN = -\tfrac N4 FEF^{-1},
\quad
S=I+O(z^N/\la^4)
\]
with the properties (i) $S\vert_{z=0}=I$, (ii) $S$ is homogeneous in the sense that the $(i,j)$ entry of $S$ has weight $2(\nn_j-\nn_i)$.
\end{proposition} 

\begin{proof}  Substituting $L=eS$ into $L^{-1}L_z = \tfrac1\la \eta$, we obtain the equation 
\[
S^{-1}S_z = \tfrac1\la \eta -  \tfrac1\la  \tfrac1z S^{-1} \MM^T S
\]
for $S$.  
Observe
that $\eta-\tfrac1z\MM^T$ contains only $z^p$ with $p>-1$. 
(For example, in the case (E1),
we have $\MM=c_2E=c_2 E_{1,2}$ with $k_2=-1$ and $k_i>-1$ if $i\ne 2$, so
$\eta -  \tfrac1z  \MM^T  = c_0 E_{0,3}z^{k_0} + c_1E_{1,0} z^{k_1} + c_2E_{2,1} z^{-1} +c_3E_{3,2} z^{k_3} - c_2E_{2,1} z^{-1} = c_0 E_{0,3}z^{k_0} + c_1E_{1,0} z^{k_1} +c_3E_{3,2} z^{k_3}$.)
Hence there exists a unique solution near $z=0$
with $S\vert_{z=0}=I$.  The coefficients of this equation have the stated homogeneity property, hence $S$ does as well.  
\end{proof}

\begin{lemma}\label{gtildeL}
We have $\Phii = (\tilde g L)^T z^{-\MM}$ where $\tilde g=(\la^\nn\la^\NN)^T$.\end{lemma}

\begin{proof} 
Let $\Phi=(\tilde g L)^T$ where
$\tilde g=(\la^\nn\la^\NN)^T$. 
As $\tilde g$ is independent of $z$, $\tilde g L$ satisfies 
$(\tilde g L)^{-1}(\tilde gL)_z dz  = \tfrac1\la\eta dz = \om$. 

We claim that all entries of the $i$-th column of $\tilde g L$ have weight $2\nn_i$.  
Since $\la^\nn z^{\frac1\la \MM^T} = z^{\MM^T}\la^\nn$, we have
\[
\tilde g L = \la^{\NN^T} \la^\nn z^{\frac1\la \MM^T}  S = \la^{\NN^T}  z^{\MM^T}  \la^\nn S.
\]
Since the weights of $\la,z$ are $2,\tfrac 8N$ respectively, the factor
\[
 \la^{\NN^T}  z^{\MM^T}
 = z^{\MM^T}  \la^{\NN^T} 
\]
has all (nonzero) entries of weight zero.  Hence the entries of $\tilde g L$ have the same weights as those of $\la^\nn S$.  Now, the  $(i,j)$ entry of $S$ has weight $2(\nn_j-\nn_i)$, and the diagonal elements of $\la^\nn$ have weights $2\nn_0,2\nn_1,2\nn_2,2\nn_3$.  This justifies the claim. 

As in section \ref{four3}, this implies that 
$(\tilde g L)^{-1}(\tilde gL)_\la d\la = \hat\om$, i.e.\ $\Phi=(\tilde g L)^T$ is a solution of (\ref{phiode}).  On the other hand, we have
$(\tilde g L)^T=S^T  z^{\frac1\la \MM} \la^\nn \la^\NN = (I+O(z^N/\la^4)) \la^\nn \la^\NN z^\MM$,
so this must be $\Phii(\la) z^\MM$.
\end{proof}

Now we are ready for the Iwasawa factorization.
Recall (section \ref{four}) that this means the Iwasawa factorization for the complex loop group
$\La \SL_{n+1}\C$
with respect to the real form 
$\La_\R \SL_{n+1}\C = \{ \ga\in \La \SL_{n+1}\C \st  c(\ga(1/{\bar\la})) = \ga(\la)\}$. 

\begin{proposition}\label{Riwasawa} For each of the boundary cases, there exists\footnote{It turns out that $\gazi$ is essentially unique --- the ambiguity in $\gazi$ leads only to an insignificant reparametrization of the solutions of the tt*-Toda equations.} a loop $\gazi\in\La \SL_{n+1}\C$ with the following properties. 

\no(i) The Iwasawa factorization
\[
\gazi L=(\gazi L)_\R 
(\gazi L)_+
\]
exists near $z=0$.  

\no(ii) The $(i,j)$ entry of $(\gazi L)_+$  has weight $2(\nn_j-\nn_i)$.

\no(iii) $(\gazi L)_+=\diag(b_0,b_1,b_2,b_3)+O(\la)$ with
$b_0b_3=1=b_1b_2$.

\no The loop $\gazi$ and the functions $b_i$ are given 
in Table \ref{t3}.  In general
$(\gazi L)_\R, (\gazi L)_+$
are multi-valued, but each $b_i$ is a
single valued function of $\vert z\vert$ near $z=0$.  
\end{proposition}

\begin{table}[h]
\renewcommand{\arraystretch}{1.5}
\begin{tabular}{c||c|c}
& $\gazi$ & 
$b_i= X_i(I+o(1))$ as $z\to 0$; $\emodz=\log\vert z\vert$
\\
\hline
\vphantom{$\bsp X \\ X \\ X \\ X \\ X \\ X \esp$}
E1 & 
$\bsp 
1 & & & \\
 & \ 1\  & -\la & \\
 & & 1 & \\
 & & &1
\esp$
 &  
$X_0=1$,  $X_1=(1 \!-\! 2c_2\emodz)^{\frac12}$   
\\
\vphantom{$\bsp X \\ X \\ X \\ X \\ X \\ X \esp$}
E3 & 
$\bsp 
1\  & -\la & & \\
 & 1 & & & \\
  & & \ 1\  & -\la \\
 & & & 1
\esp$
&
$X_0=(1 \!-\! 2c_1\emodz)^{\frac12}$, $X_1=(1 \!-\! 2c_1\emodz)^{-\frac12}$ 
\\
\hline
\vphantom{$\bsp X \\ X \\ X \\ X \\ X \\ X \esp$}
V1 &  
\text{see Proof}
&  $\text{see Note $1$ for $X_0,X_1$}$
\\
\vphantom{$\bsp X \\ X \\ X \\ X \\ X \\ X \esp$}
V2 &  
$\bsp 
1  &  & & \\
 & \ 1\  & -\la & & \\
  & &  1  &  \\
-\la & & & \ 1
\esp$
&  
$X_0=(1 \!-\! 2c_0\emodz)^{-\frac12}, X_1=(1 \!-\! 2c_2\emodz)^{\frac12}$ 
\end{tabular}
\bigskip
\caption{The loop $\gazi$ and the functions
$b_i$ in Proposition \ref{Riwasawa}.}
\label{t3}
\end{table}
\no{\em Note $1$.\ } In the case (V1), we have
$X_0=(-1-c_2\ell_{\vert z\vert} - 2c_1c_2 \ell_{\vert z\vert}^2
-\tfrac43 c_1^2 c_2 \ell_{\vert z\vert}^3)^{\frac12}$,
$X_0X_1=
(-c_2 c_1^{-1} + \tfrac14 c_2^2 c_1^{-2}
-(2c_2- c_2^2 c_1^{-1}) \ell_{\vert z\vert}
+2c_2^2 \ell_{\vert z\vert}^2
+\tfrac83 c_1c_2^2 \ell_{\vert z\vert}^3
+\tfrac43 c_1^2 c_2^2 \ell_{\vert z\vert}^4
)^{\frac12}$.

\no{\em Note $2$.\ } The tt*-Toda equations have the symmetry $w\mapsto \De w\De$. In our situation ($n=3$) this means that $(-w_1,-w_0)$ is a solution whenever $(w_0,w_1)$ is a solution. In terms of  
Figure \ref{6components}, this corresponds to reflection in the line $\ga_0+\ga_1=0$. Therefore, regarding the asymptotics of the functions $w_i$, it suffices to treat the cases (E1),(E3),(V1),(V2). From now on we restrict to these cases.

\begin{proof}  We use the method of Theorem 4.1 of \cite{DoGuRo10}.  We give the details for the case (E1), then state the modifications needed for the other cases.  

In general, an Iwasawa factorization $U=U_\R U_+$
with $U_+=b+O(\la)$ is equivalent to a Birkhoff factorization $V=V_-V_+$ for $V=c(U)(1/\bar\la)^{-1} U(\la)$, as
\begin{align*}
V=c(U)(1/\bar\la)^{-1} U(\la) &= c(U_+)(1/\bar\la)^{-1} c(U_\R)(1/\bar\la)^{-1} U_\R U_+ 
\\
&= 
c(U_+)(1/\bar\la)^{-1} U_+(\la),
\end{align*}
which can be expressed in the form $V_-V_+$ with $V_-=I+O(1/\la)$ and $V_+=b^2+O(\la)$.

First we shall establish the local (near $z=0$) Iwasawa factorization of $Y=\gazi  e$,
then deduce that of $U=\gazi  eS$ $(=\gazi L)$.
In the case (E1) we write $\gazi$ (from Table \ref{t3}) as
\[
\gazi=d_\la^{-1}Ad_\la,
\quad
A=
\bsp
1 & & & \\
 & 1 & -1& \\
 & & 1 & \\
 & & &1
\esp,
\quad
d_\la=
\bsp
1 & & & \\
 & \la & & \\
 & & \la^2 & \\
 & & &\!\!\la^3
\esp.
\]
Let us introduce
\begin{equation}\label{Delta1}
\De_1=A^T \De A^{-T}.
\end{equation}
We shall make use of the properties

(i) $\gazi$ satisfies the cyclic symmetry and anti-symmetry conditions 
$\tau(\gazi(\la))=\gazi(\om\la)$ and $\si(\gazi(\la))=\gazi(-\la)$, and

(ii) $\De_1\MM=-\MM\De_1$

\no(both of which are easily verified). The first is needed so that 
$\gazi L$ lies in the (complex, twisted) loop group that we are using,  hence also its Iwasawa and Birkhoff factors. The second will be used in the following paragraph and later in the proof of Proposition \ref{Ryandz}.

Let us now compute $Z=c( \gazi  e )(1/\bar\la)^{-1}  (\gazi e)(\la)$. 
First we observe that $(\gazi e)(\la)= d_\la^{-1}A d_\la\ d_\la^{-1} z^{\MM^T} d_\la = d_\la^{-1} A z^{\MM^T} d_\la$, 
$c( \gazi  e )(1/\bar\la)=
\De d_\la A \bar z^{\MM^T} d_\la^{-1} \De$.   
Hence
\begin{align*}
Z&=
\De d_\la \bar z^{(-\MM^T)} A^{-1}  d_\la^{-1} \De  
\ 
d_\la^{-1} A z^{\MM^T} d_\la
\\
&=\la^{-3} \De d_\la \bar z^{(-\MM^T)} A^{-1} \De A\   z^{\MM^T} d_\la\\
&=\la^{-3} d_\la^{-1} d_\la \De d_\la\   \De_1^T\  \bar z^{\MM^T} z^{\MM^T} d_\la \quad\text{by (ii) above}
\\
&=d_\la^{-1} \De\,  \De_1^T e^{2\log\vert z\vert \MM^T} d_\la.
\end{align*}
 
Inserting the value of $\De_1$ from (\ref{Delta1}) we obtain
\begin{equation}\label{Z}
Z=
d_\la^{-1} 
{\scriptsize
\left(
\begin{array}{c|cc|c}
1 & & & \\
\hline
 & 1\!-\!2c_2\log\vert z\vert & -1 & \\
 & 1 & 0 & \\
 \hline
 & & & 1
\end{array}
\right)
}
d_\la.
\end{equation}
If it exists, the Birkhoff factorization $Z=Z_-Z_+$ must be of the form
\begin{equation}\label{ZBirk}
Z_-Z_+=
d_\la^{-1} 
\bsp
 \vphantom{ X_0^2} 1 & \hphantom{ X_0^2} & \hphantom{ X_0^2} & \hphantom{ X_0^2} \\
* &  \vphantom{ X_0^2} 1 & & \\
* & * &  \vphantom{ X_0^2} 1 & \\
* & * & * &  \vphantom{ X_0^2} 1
\esp
d_\la
\ \ 
d_\la^{-1} 
\bsp
 X_0^2 & * &* & *\\
  &  X_1^2 & *& *\\
  &   &  X_2^2 & *\\
  &   &   &  X_3^2
\esp
d_\la.
\end{equation}
The existence of the factorization (near $z=0$) now follows from an explicit computation of the $ X_i$ (cf.\  \cite{Gu97}, Chapter 14, Step 2).  This may be done by equating the upper principal $k\times k$ minors of (\ref{Z}) and (\ref{ZBirk}). By the anti-symmetry condition in (i) above we must have $ X_0 X_3=1= X_1 X_2$, so it suffices to do this for $k=1,2$. We obtain
\[
 X_0^2=1,\quad  X_0^2 X_1^2=1\!-\!2c_2\log\vert z\vert,
\]
respectively. The fact that these are {\em positive} for $z$ close to (but not equal to) zero shows that the 
Birkhoff factorization $Z=Z_-Z_+$ exists in some punctured neighbourhood of zero, and hence also the Iwasawa factorization $Y=Y_\R Y_+$, with single-valued $ X_i$.  
 
Next, we can deduce the existence of the local (near $z=0$) Iwasawa factorization of
$U=\gazi  L$.   For this we consider
\[
V=c( \gazi  L )(1/\bar\la)^{-1}  (\gazi  L)(\la) = 
c(S) (1/\bar\la)^{-1} Z(\la) S(\la), 
\]
where $Z(\la)=d_\la^{-1} \De\,  \De_1^T e^{2\log\vert z\vert \MM^T} d_\la$ as above.

We claim that $W=Z(\la)^{-1} c(S) (1/\bar\la)^{-1} Z(\la) S(\la)$ is well-defined in a neighbourhood of $z=0$ (including at $z=0$).  This follows from the fact that $\lim_{z\to 0} z\log\vert z\vert=0$, hence $\lim_{z\to 0} Z(\la)^{-1} c(S) (1/\bar\la)^{-1} Z(\la)S(\la)=I$. Thus we can write
\[
c(S) (1/\bar\la)^{-1} Z(\la) S(\la) = Z(\la)W(\la)= 
Z_-(\la)Z_+(\la)W(\la).
\]
As the local (near $z=0$) Birkhoff factorization of $Z_+(\la)W(\la)$ exists, so does that of $Z_-(\la)Z_+(\la)W(\la)$, 
 and its \ll positive factor\rr is of the form
$b^2=X^2(I+o(1))$.  
This completes the proof for the case (E1).

The cases (E3) and (V2) are very similar. For the case (E3) we take
$\gazi=d_\la^{-1}Ad_\la$ with
\[
A=
\bsp
1 &-1 & & \\
 & 1 & & \\
 & & 1 &-1 \\
 & & &1
\esp,
\quad
d_\la=
\bsp
1 & & & \\
 & \la & & \\
 & & \la^2 & \\
 & & &\!\!\la^3
\esp
\]
and for the case (V2) 
\[
A=
\bsp
1 & & & \\
 & 1 &-1 & \\
 & & 1 & \\
 -1 & & &1
\esp,
\quad
d_\la=
\bsp
\la^3 & & & \\
 & 1 & & \\
 & & \la & \\
 & & &\la^2
\esp.
\]
For the case (V1) we take $d_\la=\diag(1,\la,\la^2,\la^3)$ and any $A$ such that
\[
\De_1=
A^T \De A^{-T} = 
\bp
1 & 1 & \frac12\frac{c_2}{c_1} & 1\\
 & 1 & \frac{c_2}{c_1} & \frac12\frac{c_2}{c_1}\\
  & & 1 & 1\\
  & & & 1
\ep
\bp
1 & & & \\
 & -1 & & \\
  & & 1 & \\
   & & & -1
\ep.
\]
This ensures that property (ii) holds, as in the case (V1) case we have
$\MM=  c_1E_{0,1}  +  c_2E_{1,2}  +  c_1E_{2,3}$
(Table \ref{t2} of section \ref{resD1}).
Direct calculation shows that $A$ exists and is essentially unique.
\end{proof}

The existence of the Iwasawa factorization allows
us to carry out the construction of 
\[
\al = [ (\gazi L)_\R G ]^{-1} d[(\gazi L)_\R  G ]=
(w_t+\tfrac1\la W^T)dt + (-w_{\tbar}+\la W)d\tbar,
\]
as in section \ref{four2}.   Here $w_0,w_1,w_2,w_3$ are defined again by $w_i=\log b_i/\vert h_i\vert$ but now $b_0,b_1,b_2,b_3$ are as in Proposition \ref{Riwasawa} (Table \ref{t3}).

Using this, we can define 
$
\hat\al=
\left[
- \tfrac{t}{\la^2}
\ W^T
- \tfrac1\la xw_x + \tbar \,W
\right]
d\la
$
as in section \ref{four4}.  

\begin{lemma} Let $\Psi=
(\tilde g \gazi^{-1} (\gazi L)_\R G)^T$.
Then
$\Psi$ satisfies equation (\ref{psiode}).
\end{lemma}

\begin{proof} By definition
$[ (\gazi L)_\R   G ]^{-1} d[(\gazi L)_\R   G ] = \al$.  
As $\tilde g \gazi^{-1}$ is independent of $z$, we have
$[ \tilde g \gazi^{-1} (\gazi L)_\R   G ]^{-1} d[\tilde g \gazi^{-1} (\gazi L)_\R   G ]= \al$ as well.  

We claim that all entries of $\tilde g \gazi^{-1} (\gazi L)_\R   G$ have weight zero.  This would imply that $[ \tilde g \gazi^{-1} (\gazi L)_\R   G ]^{-1} [\tilde g \gazi^{-1} (\gazi L)_\R   G ]_\la d\la = \hat\al$, by the argument of section \ref{four4}. 

To prove the claim, we observe that 
$\tilde g \gazi^{-1} (\gazi L)_\R   (\gazi L)_+ G =
\tilde g \gazi^{-1} \gazi  L G = \tilde g L G =
\la^{\NN^T} \la^\nn z^{\frac1\la \MM^T} S G = \la^{\NN^T}  z^\MM \la^\nn S G$.  Since 
$(\gazi L)_+ G$ and $S G$ have the same weights, so do 
$\tilde g \gazi^{-1} (\gazi L)_\R $ and $\la^{\NN^T}  z^\MM \la^\nn$, 
and so do $\tilde g \gazi^{-1} (\gazi L)_\R   G$ and 
$\la^{\NN^T}  z^\MM \la^\nn G$.  The latter has weight zero, so this completes the proof.
\end{proof} 

The diagram in section \ref{dkandek} must now be modified as follows:
\[
\xymatrix{
 \quad \quad \quad \quad \quad \quad \quad \quad\quad\quad
  &   \Psii\ar@{-}[d]^{\text{Z}}
\\
 \Phii = (\tilde gL)^Tz^{-\MM}\ar@{--}[r]^{\text{\ \ \ \  Iwasawa}}    &    (\tilde g \gazi^{-1} (\gazi L)_\R  G)^T
\\
 \quad \Phiz\ar@{-}[u]^{\text{D}}\ar@{--}[r]   \quad   &  \quad \Psiz\ar@{-}[u]_{\text{Y}}\quad
}
\]
Here $\Phii=\Phiz_k D_k$ as in section \ref{omegahat}, but now
we define $Y_k$ and $Z_k$ by
$\Psiz_k=(\tilde g \gazi^{-1} (\gazi L)_\R   G)^T Y_k,\quad 
\Psii_k=(\tilde g \gazi^{-1} (\gazi L)_\R   G)^T Z_k$.

Using this we shall obtain the following modification of Theorem \ref{e}:

\begin{theorem}\label{Re}  The connection matrix $E_k$ is given by the formula   
\[
E_k=\tfrac14 D_k z^\MM\, 
\De_1 
(\overline{D_{\frac74-k}z^\MM})^{-1}\,
d_4^3\,  C, 
\quad
{\scriptsize
C=
\left(
\begin{array}{c|ccc}
\!\!1 & & & \\
\hline
 & & & 1\!\\
 & & 1 & \\
 & \!\!1 & &
\end{array}
\right)
}
\]
where $D_k$ is as in section \ref{resD1} and
$\De_1$ is defined by (\ref{Delta1}).
In particular this gives
\[
E_1= (D_1z^\MM)\, \De_1\, (\overline{D_1z^\MM})^{-1}\,d_4^{-1}E_1^{\text{\em global}}
\]
where $E_1^{\text{\em global}} = \tfrac14C \Qi_{\frac34}$ as in 
(\ref{globalE1}). 
\end{theorem}

The proof will be given later.  It depends on the following modifications of Propositions \ref{yandd} and \ref{yandz}.

\begin{proposition}\label{Ryandd}
$Y_k= (D_kz^\MM)^{-1}$.
\end{proposition}

\begin{proof}  This is similar to the proof of Proposition \ref{yandd}.  
We have
\begin{align*}
\Psiz_k&=(\tilde g \gazi^{-1} (\gazi L)_\R  G)^T Y_k
\\
&=(\tilde g \gazi^{-1} (\gazi L) (\gazi L)_+^{-1} G)^T Y_k
\\
&=((\gazi L)_+^{-1} G)^T (\tilde g L)^T Y_k.
\end{align*}
Since $(\tilde g L)^T = \Phii z^\MM = \Phiz_k D_k z^\MM$, and 
$\Phiz_k \sim  O_0 (I+O(\la) ) e^{\frac t\la d_4}$, we obtain
\begin{align*}
\Psiz_k  &\sim Gb^{-1} O_0  (I+O(\la) ) e^{\frac t\la d_4} D_k z^\MM Y_k
\\
&= P_0  (I+O(\la) ) e^{\frac t\la d_4} D_k z^\MM Y_k
\end{align*}
when $\la\to 0$.
But $\Psiz_k\sim P_0 (I+O(\la) ) e^{\frac t\la d_4}$, 
so we conclude that $D_k  z^\MM Y_k=I$.
\end{proof}

\begin{proposition}\label{Ryandz}  
$Z_k= \tfrac14 \De_1\, \bar Y_{\frac74-k}\, d_4^3\, C$,
where $\De_1$ is as in (\ref{Delta1}).
\end{proposition} 

\begin{proof}
This is similar to the proof of Proposition \ref{yandz}
From $\Psiz_k=(\tilde g \gazi^{-1} (\gazi L)_\R G)^T Y_k$ we obtain
\begin{align*}
\overline{\Psiz_k(1/\bar\la)} 
&=
\bar G \ 
\overline{(\gazi L)_\R(1/\bar\la)}^T  \ \overline{(\tilde g\gazi^{-1})(1/\bar\la)}^T \ \bar Y_k
\\
&=
\bar G \   
\De 
(\gazi L)_\R(\la)^T \De
\ \overline{(\tilde g\gazi^{-1})(1/\bar\la)}^T \ \bar Y_k.
\end{align*}
From $\Psii_{\frac74-k}=(\tilde g \gazi^{-1} (\gazi L)_\R G)^T Z_{\frac74-k}
=G  (\gazi L)_\R^T (\tilde g \gazi^{-1})^T  Z_{\frac74-k}$,
we obtain 
\[
(\gazi L)_\R^T =  G^{-1} \Psii_{\frac74-k} Z_{\frac74-k}^{-1} (\tilde g \gazi^{-1})^{-T}.
\]
Substituting this into the previous formula gives
\[
\overline{\Psiz_k(1/\bar\la)}=
\bar G \,   \De\,   G^{-1} \ 
\Psii_{\frac74-k}(\la) \ Z_{\frac74-k}^{-1} \ 
(\tilde g \gazi^{-1})^{-T} \De \ 
\overline{(\tilde g\gazi^{-1})(1/\bar\la)}^T \ \bar Y_k.
\]
Now, we have
$\bar G \,  \De\,  G^{-1} = \De$.
Furthermore, we claim that 
\begin{equation}\label{identityforgamma}
(\tilde g \gazi^{-1})^{-T} \De \ 
\overline{(\tilde g\gazi^{-1})(1/\bar\la)}^T = \De_1.
\end{equation}
From these two facts we have
$\overline{\Psiz_k(1/\bar\la)} =\De \Psii_{\frac74-k} Z_{\frac74-k}^{-1} \De_1 
\bar Y_k$. On the other hand, from (\ref{loopgroupreality}) we know that 
$\overline{\Psiz_k(1/\bar\la)}=\De \Psii_{\frac74-k}(\la) \, C \, 4d_4$. 
Comparing these, we deduce that 
$C\,4d_4 =  Z_{\frac74-k}^{-1} \, \De_1 \, \bar Y_k$, as required.

It remains to establish (\ref{identityforgamma}).  We have $\tilde g=(\la^\nn\la^\NN)^T$ and (using the notation of the proof of Proposition \ref{Riwasawa}) $\ga=d_\la^{-1} A d_\la$,  $\De_1=A^T \De A^{-T}$.  Thus the left hand side of (\ref{identityforgamma}) is
\[
\la^{-\NN}\la^{-\nn}(d_\la A^T d_\la^{-1})
\De
(d_\la^{-1} A^{-T} d_\la) \la^{-\nn}\la^{-\NN}.
\]
Observe that $\la^{-\nn}d_\la$ commutes with $A^T$ in all cases.
Thus we obtain
$\la^{-\NN} A^T  \la^{-\nn}  d_\la \ d_\la^{-1}  
\De
d_\la^{-1} \ d_\la \la^{-\nn} A^{-T} \la^{-\NN}
=
\la^{-\NN} A^T  \la^{-\nn}   
\De
\la^{-\nn} A^{-T} \la^{-\NN} 
=
\la^{-\NN} A^T    
\De
A^{-T} \la^{-\NN} 
=
\la^{-\NN} \De_1 \la^{-\NN}.$
By property (ii) in the proof of Proposition \ref{Riwasawa}
this is equal to $\De_1$, as $M= -\frac4N\MM$.
\end{proof}

\begin{table}[h]
\renewcommand{\arraystretch}{2.0}
\begin{tabular}{c||c|c|c}
& $\mathcal E$ & $\mathcal F$ & $e^\R_i,f^\R_i$
\\
\hline
\vphantom{$\bsp X \\ X \\ X \\ X \\ X \\ X \\ X \esp$}
E1 & 
$\bsp 
e^\R_1 & & & \\
\vphantom{f_1^\R} 
  & 1 &  & \\
  \vphantom{f_1^\R} 
   &  & 1 & \\
    & & & 1/e^\R_1
\esp$
 &
$\bsp 
\vphantom{f_1^\R} 
1 & & & \\
  &\, 1\,  &f^\R_1  & \\
  \vphantom{f_1^\R} 
   &  &\, 1\,  & \\
   \vphantom{f_1^\R} 
    & & &1
\esp$
&
$
\begin{matrix}
e^\R_1=-\frac{\tilde A_3 F_3}{\tilde A_0 F_0}
\\
f^\R_1= - \frac{F_2}{F_1} - 2\frac{\tilde A_1^\flat}{\tilde A_2}
\end{matrix}
$
\\
\vphantom{$\bsp X \\ X \\ X \\ X \\ X \\ X \\ X\esp$}
E3 
& 
$\bsp
e^\R_1\,   & & & \\
  &\,  e^\R_1\,  &  & \\
  & &  1/e^\R_1 & \\
  & & & \ 1/e^\R_1
\esp$
&
$\bsp
1  & f_1^\R & & \\
\vphantom{f_1^\R} 
 & 1 &  & \\
  & &  1 & f_1^\R \\
\vphantom{f_1^\R} 
  & & & \ 1
\esp$
&
$
\begin{matrix}
e^\R_1=\frac{\tilde A_3 F_2}{\tilde A_1 F_0}
\\
f^\R_1= - \frac{F_3}{F_2} - \frac{\tilde A_2^\flat}{\tilde A_3}
 - \frac{\tilde A_0^\flat}{\tilde A_1}
\end{matrix}
$
\\
\hline
\vphantom{$\bsp X \\ X \\ X \\ X \\ X \\ X \\ X \esp$}
V1 &  I  &  
$\bsp
1  & \ f_1^\R\  & \frac12(f_1^\R)^2& f_2^\R \\
\vphantom{f_1^\R} 
 & 1 & f_1^\R  & \frac12(f_1^\R)^2 \\
  & &  1 & f_1^\R \\
\vphantom{f_1^\R} 
  & & & \ 1
\esp$
 &
$\begin{matrix}
f^\R_1=\frac{F_1}{F_0} - 2\frac{\tilde A^\flat_2}{\tilde A_3}
\\
f^\R_2=\text{see Note}
\end{matrix}
$
\\
\vphantom{$\bsp X \\ X \\ X \\ X \\ X \\ X \esp$}
V2 & I  & 
$\bsp 
\vphantom{f_1^\R} 
1 & & & \\
\vphantom{f_1^\R} 
  &\, 1\,  & f^\R_1 & \\
  \vphantom{f_1^\R} 
   &  &\, 1\,  & \\
f^\R_2    & & &1
\esp$  &
\text{
$
\begin{matrix}
f^\R_1&= - \frac{F_2}{F_1} - 2\frac{\tilde A_1^\flat}{\tilde A_2}
\\
f^\R_2&= - \frac{F_0}{F_3} - 2\frac{\tilde A_3^\flat}{\tilde A_0}
\end{matrix}
$
}
\end{tabular}
\bigskip
\caption{The matrices $\mathcal E, \mathcal F$ in Theorem \ref{Rexplicite}.}
\label{t4}
\end{table}
\no{\em Note.\ }   In the case (V1), we have
\[
f^\R_2=
\tfrac{F_3}{F_0}
-\tfrac{(F_1)^2}{(F_0)^2}\tfrac{\tilde A_2^\flat}{\tilde A_3}
+2\tfrac{F_1}{F_0}\tfrac{(\tilde A_2^\flat)^2}{(\tilde A_3)^2}
-2\tfrac{(\tilde A_2^\flat)^3}{(\tilde A_3)^3}
+2 \tfrac{\tilde A_1^\flat \tilde A_2^\flat}{(\tilde A_3)^2}
-2 \tfrac{\tilde A_0^\flat}{\tilde A_3}.
\]

\no{\em Proof of Theorem \ref{Re}.\ \ }  From Propositions \ref{Ryandd} and \ref{Ryandz} we have
$E_k=Y_k^{-1}Z_k
=
(z^{-\MM} D_k^{-1})^{-1} \tfrac14 \De_1 \overline{z^{-\MM} D_{\frac74-k}^{-1}} d_4^3\, C$.
This is the required formula for $E_k$.  We deduce the formula for $E_1$ exactly as in the proof of Theorem \ref{e} (the generic case). \qed

In view of Theorem \ref{Re}, let us introduce the notation
\[
D_1^\flat = D_1 z^\MM.
\]
Explicitly, using Theorem \ref{Rconnectionomegahat}, this gives
\begin{align*}
D_1^\flat&= \ka_0 \hat c_0^{-1} (P\tilde K)^{-T} (\tilde A)^{-T} (FU)^{-1} z^\MM
\\
&=  \ka_0 \hat c_0^{-1} (P\tilde K)^{-T} (\tilde A)^{-T} z^{-\frac N4 E} (FU)^{-1}
\end{align*}
because $\MM=-\frac N4 \NN = -\frac N4 (FU) E (FU)^{-1}$ (Proposition \ref{thecases}).

Introducing the analogous notation
\begin{equation}\label{Aflat}
\tilde A^\flat = \tilde A z^{\frac N4 E^T} U^T,
\end{equation}
we can write
$
D_1^\flat = \ka_0 \hat c_0^{-1} (P\tilde K)^{-T} (\tilde A^\flat)^{-T} F^{-1}.
$
The matrices $\tilde A_i^\flat$ are listed in the Appendix. They are independent of $z$. It follows that the connection matrices $E_k$ are (as expected) independent of $z$.

\begin{theorem}\label{Rexplicite}  
The connection matrix $E_1$ 
is given by the formula 
\[
E_1=
(P\tilde K)^{-T}
\mathcal E \mathcal F
(P\tilde K)^{T}
  \ 
 E_1^{\text{\em{global}}}
\]
where the diagonal matrix $\mathcal E$ and the unipotent matrix $\mathcal F$ are
are given in Table \ref{t4}.
\end{theorem}

\begin{proof} 
Let us substitute
($D_1^\flat=$)
$D_1z^\MM=\ka_0\hat c_0^{-1} \, (P\tilde K)^{-T} (\tilde A^\flat)^{-T}F^{-1}$ (see above) into 
$E_1= (D_1z^\MM)\, \De_1\, (\overline{D_1z^\MM})^{-1}\,d_4^{-1}E_1^{\text{global}}$
(Theorem \ref{Re}).  Noting that $\ka_0$ and $\tilde A^\flat$ are pure imaginary,
we obtain
\[
E_1= (P\tilde K)^{-T} (\tilde A^\flat)^{-T} F^{-1} 
\De_1
F  (\tilde A^\flat)^{T}
(\overline{P\tilde K})^{T}
\,d_4^{-1}E_1^{\text{global}}.
\]
Now, by direct calculation we have
\begin{equation}\label{Delta0}
\bar{\tilde K}  = -d_4 \tilde K  \De_0
\end{equation}
where the (symmetric) matrix $\De_0$ is given in the Appendix. Hence
\[
E_1= -(P\tilde K)^{-T} (\tilde A^\flat)^{-T} F^{-1} 
\De_1
F  (\tilde A^\flat)^T
\De_0 \tilde K^T d_4 \bar P^T d_4^{-1} E_1^{\text{global}}.
\]
As in the proof of Theorem \ref{explicite}
we have $d_4 \bar P^T d_4^{-1}=P^T$, so
\begin{align*}
E_1 &= 
-(P\tilde K)^{-T} (\tilde A^\flat)^{-T} F^{-1} 
\De_1
F  (\tilde A^\flat)^{T} \De_0
(P\tilde K)^{T}
\,
E_1^{\text{global}}
\\
&=
-(P \tilde K)^{-T} X^T (P \tilde K)^T \, E_1^{\text{global}}
\end{align*}
where $X=\De_0 (\tilde A^\flat) F \De_1^T F^{-1} (\tilde A^\flat)^{-1}$.

It is now straightforward to write $-X^T$ as $\mathcal E \mathcal F$, where $\mathcal E, \mathcal F$ are as stated in Table \ref{t4}.
For example, in the case (E1), we have
\[
\De_0
\!=\!
\bsp
\vphantom{\tilde A_0 }
 & & & 1
\\
\vphantom{\tilde A_0 }
  & -1 &  & 
\\
\vphantom{\tilde A_0 }
   &  \ & 1 & 
\\
\vphantom{\tilde A_0 }
1 & & &
\esp\!,
\De_1
\!=\!
\bsp
\vphantom{\tilde A_0 }
 & & & 1
\\
\vphantom{\tilde A_0 }
  & -1 & 1 & 
\\
\vphantom{\tilde A_0 }
   &  \ & 1 & 
\\
\vphantom{\tilde A_0 }
1 & & &
\esp\!,
\tilde A^\flat
\!=\!
\bsp
\tilde A_0 & & & 
\\
 & \tilde A_2 & & 
\\
 & \tilde A_1^\flat  & \tilde A_2   & 
 \\
 & & & \tilde A_3 
\esp\!,
\]
hence $X$ is of the form
$
\bsp
* & & & 
\\
 & * & & 
\\
 & *  & *   & 
 \\
 & & & *
\esp,
$
from which the expressions in Table \ref{t4} for $e_1^\R,f_1^\R$ are obtained. 
\end{proof}

\section{Resonance: the global solutions and their asymptotics}\label{resGLOBAL}

As in the non-resonant case, the Iwasawa factorization (Proposition \ref{Riwasawa}) leads to asymptotic expressions for the solutions $w_i$ which are smooth near $t=0$. We still have $w_i=\log b_i/\vert h_i\vert$ with 
$\vert h_i\vert =\vert\hat c_i t^{\nn_i}\vert=\vert\hat c_i t^{-\ga_i/2}\vert$, but now there is an extra term as we have
$b_i=X_i(1+o(1))$
instead of
$b_i=1+o(1)$. We obtain:
\begin{equation}\label{asymptotics}
2w_i(t)=\ga_i \log\vert t\vert + \log X_i^2\hat c_i^{-2} + o(1)
\end{equation}
as $t\to 0$, where $X_i$ is given in Table \ref{t3} (section \ref{resE1}).  If $X_i$ is constant (as in the non-resonant case), the error term $o(1)$ is sufficient to ensure that we have obtained a parametrization of such solutions.  If $X_i$ is not constant it is necessary to improve the error term. This can be done by examining the proof of Proposition \ref{Riwasawa} more carefully --- the error comes from neglecting the factor
$Z^{-1} c(S) (1/\bar\la)^{-1} ZS$, and is always of the form
$\eps=O(\vert t\vert^k \log^\ell\vert t\vert)$ for some $k,\ell>0$. The precise values of $k,\ell>0$ will not play any role so we omit them below.

The global solutions are our main interest, and, at this point, we have at our disposal all necessary information about them.
First, Theorems \ref{Re} and \ref{Rexplicite} provide the following analogue of Corollary \ref{speciale}:

\begin{corollary}\label{Rspeciale}  The following conditions (i)-(iii) are equivalent:

\no (i) $E_1 = \tfrac14C \Qi_{\frac34} (= E_1^{\text{\em global}})$

\no (ii) $\mathcal E=I=\mathcal F$ (i.e.\ all $e_i^\R=1$ and all $f_i^\R=0$)

\no (iii) $D_1z^\MM= d_4 \overline{D_1z^\MM} \De_1$
(i.e.\ $D_1^\flat = d_4 \bar D_1^\flat \De_1$)
\end{corollary}

As in the non-resonant case, we may use this criterion to give explicitly the holomorphic data and asymptotic data of the global solutions. 

We shall discuss the case (E1) in detail, then state the results for the other cases, which are obtained in exactly the same way.  

As explained above, from (\ref{asymptotics}) and Table \ref{t3} we obtain
\begin{equation}\label{caseE1asymdata}
\begin{cases}
2w_0(t)&= \ga_0 \log\vert t\vert + \log  \hat c_0^{-2}
+o(1)
\\
2w_1(t)&= \log\vert t\vert + \log \left(
1-2c_2\log \vert z\vert
\right)\hat c_1^{-2}
+\eps
\end{cases}
\end{equation}
as $t\to 0$.  Note that $\ga_1=1$ for the second formula as we are in the case (E1).  We can rewrite
\begin{align*}
(1-2c_2\log \vert z\vert)\hat c_1^{-2}
&= \hat c_1^{-2} -\tfrac N2 \log \vert z\vert
\quad
\text{as $\hat c_1^{-2}=\tfrac N{4c_2}$ by (\ref{chat})}
\\
&= \hat c_1^{-2} - 2\log\tfrac{\vert t\vert}{4} +\tfrac12 \log \tfrac{c}{N^4}
\quad
\text{as $t=\tfrac{4}N \, c^{\frac1{4}} \, z^{\frac N{4}}$}
\end{align*}
and then the second formula of (\ref{caseE1asymdata}) becomes
\[
2w_1(t)= \log\vert t\vert + \log \left(
\hat c_1^{-2} - 2\log\tfrac{\vert t\vert}{4} +\tfrac12 \log \tfrac{c}{N^4}
\right)
+\eps.
\]

In order to apply Corollary \ref{Rspeciale}, we shall convert 
$\hat c_0^{-2}, \hat c_1^{-2}$ into monodromy data.  The 
monodromy data is given in Table \ref{t4} of section \ref{resE1}. 
Using the formulae there, with the explicit values of the matrices $F$ 
(Table \ref{t2} of section \ref{resD1})
and $\tilde A$ (Appendix), we obtain:
\begin{equation}\label{caseE1mondata} 
\begin{cases}
e^\R_1& = -\frac{\tilde A_3 F_3}{\tilde A_0 F_0}
= -2^{-2a} t^a  \frac{ H(\frac a4) }{G(0)} \
\tfrac14 a^3 t^{-a} \hat c_0^{-2}
= - \hat c_0^{-2} 2^{-2-2a} a^3 \frac{ H(\frac a4) }{G(0)}
\\
f^\R_1& = - \frac{F_2}{F_1} - 2\frac{\tilde A_1^\flat}{\tilde A_2}
= \hat c_1^{-2} + \euler + \tfrac4a + \tfrac12 (\log F)^\pr(\tfrac a8) + \tfrac12 \log \tfrac{c}{N^4}
\end{cases}
\end{equation}
These allow us to convert the asymptotic  formulae (\ref{caseE1asymdata}) into
\begin{equation}\label{caseE1asymdata2}
\begin{cases}
2w_0(t) & \!\!\!=\! \ga_0 \log\vert t\vert + \log\left(
- e_1^\R\, 2^{2+2a} a^{-3} G(0)/H(\tfrac a4)
\right)
+o(1)
\\
2w_1(t) & \!\!\!=\! \log\vert t\vert \!+\! \log \left(
f_1^\R \!-\! \euler \!-\! \tfrac4a \!-\! \tfrac12 (\log F)^\pr(\tfrac a8)
 \!-\! 2\log\tfrac{\vert t\vert}{4}
\right)
\!
+
\!
\eps
\end{cases}
\end{equation}
For the global solutions we have (by (ii) of Corollary \ref{Rspeciale})  $e_1^\R=1$ and $f_1^\R=0$.  Substituting these values into (\ref{caseE1asymdata2}), we obtain the asymptotics of the global solutions.  The holomorphic data of 
the global solutions (represented by $\hat c_0,\hat c_1$ in this case) can also be found by putting $e_1^\R=1$ and $f_1^\R=0$ in (\ref{caseE1mondata}).  This completes our treatment of the case (E1).

For the other cases, exactly the same procedure --- taking the explicit data from Tables \ref{t2}, \ref{t3}, \ref{t4}, and the Appendix ---
gives the holomorphic data and asymptotic data listed in the two corollaries below.  

We recall that the original form of the holomorphic data (Definition \ref{omega}) was $p_i(z)=c_i z^{k_i}=c_i z^{\al_i}/z$ ($0\le i\le 3$), and that this was transformed into diagonal matrices $\hat c,\nn$ in Definition \ref{weightofhandchat}. Explicit formulae are given in
(\ref{chat}), (\ref{alphaiandni}), and Proposition \ref{asymptoticdata}. We use the data $\hat c,\nn$ below.

\begin{corollary}\label{Rholomorphicdata} 
{\em(Holomorphic data for global solutions in the resonant cases)}
The holomorphic data $\hat c$ (for the corresponding range of $\nn$)
is listed below. Edges and vertices refer to Figure \ref{6components} in section \ref{resD1}.

\no 
\underline
{\em
(E1) $(m_0,m_1)=(\tfrac a2 - \tfrac32,-\tfrac12)$, $0<a<4$ (top edge)
}

\no
$
\hat c_0^2= - 2^{-2-2a} a^3 H(\tfrac a4) / G(0)
$

\no
$
\hat c_1^{-2}= - \euler - \tfrac4a - \tfrac12 (\log F)^\pr(\tfrac a8) - \tfrac12 \log \tfrac{c}{N^4}
$

\no 
\underline
{\em
(E3) $(m_0,m_1)=(\tfrac a2 - \tfrac32,\tfrac a2 - \tfrac12)$, $0<a<4$
(diagonal edge)
}

\no
$
\hat c_0 \hat c_1 = 
2^{-2a} a^2 Q(\tfrac a4) / P(0)
$

\no
$
\hat c_1 \hat c_0^{-1} =
- \euler - \tfrac2a - \tfrac14 (\log P)^\pr(0) 
- \tfrac14 (\log Q)^\pr(\tfrac a4) - \tfrac12 \log \tfrac{c}{N^4}
$

\no 
\underline
{\em(V1) 
$(m_0,m_1)=(-\tfrac32,-\tfrac12)$ (top right vertex)
}

\no
$
\hat c_0^{-1} \hat c_1=  2\euler + \tfrac12 \log\tfrac{c}{N^4}
$

\no
$
\hat c_0^{-2}= \tfrac43 \euler^3 + \tfrac1{24} \zeta(3) +
\euler^2 \log\tfrac{c}{N^4}
+\tfrac14\euler \log^2\tfrac{c}{N^4}
+\tfrac1{48} \log^3\tfrac{c}{N^4}
$

\no 
\underline
{\em
(V2) $(m_0,m_1)=(\tfrac12,-\tfrac12)$ (top left vertex)
}

\no
$
\hat c_0^2= 
- \euler - 2 - \tfrac12 (\log S)^\pr(1) - \tfrac12 \log \tfrac{c}{N^4}
$

\no
$
\hat c_1^{-2}= 
- \euler - 1 - \tfrac12 (\log T)^\pr(\tfrac12) - \tfrac12 \log \tfrac{c}{N^4}
$
\end{corollary}

The (products of) gamma functions $F,G,H,\dots$ 
appearing here are defined in the Appendix.

Finally we reach our goal, the asymptotic data.
Recall that all local solutions $w$ near $t=0$ have the leading term asymptotics $2w(t)\sim \ga \log\vert t\vert$. In the non-resonant case we have $2w(t)\sim \ga \log\vert t\vert+\rho$, and the diagonal matrices $\ga,\rho$ parametrize all such solutions. Therefore, for global solutions, the parameter $\rho$ is determined by $\ga$, and we gave an explicit formulae for it in Corollary \ref{tracywidom}.  In the resonant cases, we have:

\begin{corollary}\label{Rasymptoticdata} 
{\em(Asymptotic data for global solutions in the resonant cases)}
The asymptotics at zero of the global solutions (for the corresponding range of $\ga$)
are listed below.  Edges and vertices refer to Figure \ref{6components} in section \ref{resD1}.  We use the notation 
$\psi=(\log\Ga)^\pr = \Ga^\pr/\Ga$ and $\eps=O(\vert t\vert^k \log^\ell \vert t\vert)$ for some $k,\ell>0$.

\no
\underline
{\em
(E1) $-1<\ga_0<3, \ga_1=1$ (top edge)
}
\begin{align*}
2w_0(t)& \!=\! \ga_0 \log\vert t\vert + \log\left(
- 2^{2+2a} a^{-3} G(0)/H(\tfrac a4)
\right)
+o(1)
\\
& \!=\! \ga_0 \log\vert t\vert + \log\left(
2^{-2\ga_0} 
\Ga(  \tfrac{3-\ga_0}4 )
\Ga(  \tfrac{3-\ga_0}8 )^2
\Ga(  \tfrac{\ga_0+1}4 )^{-1}
\Ga(  \tfrac{\ga_0+5}8 )^{-1}
\right)
+o(1)
\\
2w_1(t)& \!=\! \log\vert t\vert + \log \left(
 - \euler \!-\! \tfrac4a \!-\! \tfrac12 (\log F)^\pr(\tfrac a8)
 \!-\! 2\log\tfrac{\vert t\vert}{4}
\right)
+\eps
\\
& \!=\! \log\vert t\vert + \log \left(
-\euler \!-\! \tfrac 4{3-\ga_0} \!+\!
\tfrac12 \psi( \tfrac{\ga_0-3}8) \!+\!
\tfrac12 \psi( \tfrac{3-\ga_0}8) \!-\! 
2\log\tfrac{\vert t\vert}{4}
\right)
+\eps
\end{align*}

\no 
\underline
{\em
(E3)  $-1<\ga_0<3, \ga_1=\ga_0-2$ (diagonal edge)
}
\begin{align*}
2(w_0+w_1)(t)& \!=\! (\ga_0 \!+\! \ga_1)\log\vert t\vert 
- 2\log\left( 2^{-2a} a^2 Q(\tfrac a4) / P(0) 
\right)
+ o(1)
\\
&=  (\ga_0 \!+\!\ga_1)\log\vert t\vert 
- 2\log\left(
2^{2\ga_0-2}
\Ga(  \tfrac{\ga_0+1}4 )^{2}
\Ga(  \tfrac{3-\ga_0}4 )^{2}
\right)
+ o(1)
\end{align*}
\begin{align*}
2( & w_0  - w_1)(t) 
\\
&\!=\!  2 \log\vert t\vert 
\!+\! 2\log \left(   
\! - \euler \!-\! \tfrac2a \!-\! \tfrac14 (\log P)^\pr(0) 
\!-\! \tfrac14 (\log Q)^\pr(\tfrac a4) \!-\! 2\log\tfrac{\vert t\vert}{4}
  \right)
+ \eps
\\
& \!=\!  2 \log\vert t\vert 
\!+\! 2\log \left(   
\! - \euler \!-\! \tfrac2{3-\ga_0} 
\!+\!
\tfrac12 \psi( \tfrac{3-\ga_0}4) 
\!+\!
\tfrac12 \psi( \tfrac{\ga_0-3}4)
\!-\! 2\log\tfrac{\vert t\vert}{4}
\right)
+ \eps
\end{align*}

\no 
\underline
{\em
(V1) $\ga_0=3, \ga_1=1$ (top right vertex)
}

\no
$
2w_0(t)=  3\log\vert t\vert + \log  
\left(
-\tfrac1{24}\zeta(3)-\tfrac43\euler^3
-4\euler^2\log\tfrac{\vert t\vert}{4}
-4\euler\log^2\tfrac{\vert t\vert}{4}
\right.
$
\newline
$
\left.
\quad -\tfrac43
\log^3\tfrac{\vert t\vert}{4}
\right)
+\eps
$

\no
$
2(w_0+w_1)(t)= 4\log\vert t\vert + 
\log 
\left(
\vphantom{\log\tfrac{\vert t\vert}{4}}
\!\!\!
-\!\!
\tfrac1{12}
\euler\zeta(3) +\tfrac43\euler^4
+(
\tfrac{16}3\euler^3
-\tfrac1{12}\zeta(3) 
)\log\tfrac{\vert t\vert}{4}
\right.
$
\newline
$
\left.
\quad
+8\euler^2\log^2\tfrac{\vert t\vert}{4}
+\tfrac{16}3\euler\log^3\tfrac{\vert t\vert}{4}
+\tfrac43\log^4\tfrac{\vert t\vert}{4}
\right)
+\eps
$

\no 
\underline
{\em
(V2) $\ga_0=-1, \ga_1=1$ (top left vertex)
}
\begin{align*}
2w_0(t)&  = - \log\vert t\vert - \log\left(
- \euler - 2 - \tfrac12 (\log S)^\pr(1) - 2\log\tfrac{\vert t\vert}{4}
\right)
+\eps
\ \ \ \ \ \ \ \
\\
& = - \log\vert t\vert - \log\left(
-2\euler +2\log2 - 2\log\vert t\vert
\right)
+\eps
\\
2w_1(t)& =  \log\vert t\vert + \log\left(
- \euler - 1 - \tfrac12 (\log T)^\pr(\tfrac12) - 2\log\tfrac{\vert t\vert}{4}
\right)
+\eps
\\
& =  \log\vert t\vert + \log\left(
-2\euler +2\log2 - 2\log\vert t\vert
\right)
+\eps
\end{align*}
\end{corollary}

This list covers all resonant cases for which $\ga_0+\ga_1\ge0$. The remaining cases (E2),(V3) have $\ga_0+\ga_1<0$. The asymptotics for these cases may be obtained by applying the transformation $(w_0,w_1)\mapsto(-w_1,-w_0)$ (see Note $2$ following Table \ref{t3}). In the case (V2) we have 
$\ga_0+\ga_1=0$ and in fact
$w_0(t)+w_1(t)=0$ for all $t$, so our p.d.e.\ reduces to the sinh-Gordon equation $(w_0)_{\ttb}= \sinh(4w_0)$, for which the asymptotics given above are well known.

We remark that the normalization constants $c,N$ in the formulae 
(\ref{caseE1mondata}) and  (\ref{caseE1asymdata2})
for local solutions cancel out in the asymptotics of the global solutions (as they should).  The appearance of $c,N$ in (\ref{caseE1mondata}) and (\ref{caseE1asymdata2})  reflects the fact that our parametrization of the local solutions depends on our conventions for the Iwasawa factorization. 

The asymptotics as $t\to\infty$ of all solutions 
(including the resonant cases) were  stated at the end of section \ref{conclusions}.  Thus we have solved the connection problem (to find the explicit relation between asymptotic data at zero and infinity) also in the resonant cases.

\section{Appendix}\label{appR}

We list the matrices $\tilde K, U, \tilde A$ appearing in the formula
\[
D_1=
\ka_0\,  (P\tilde K)^{-T}\, \tilde\Asharp^{-T} \hat c^{-1}
=
\ka_0 \,  \hat c_0^{-1} \, (P\tilde K)^{-T} (\tilde A U^T F)^{-T}
\]
of Theorem \ref{Rconnectionomegahat}.  Here 
$\ka_0=\ii \, \pi^{\frac52} \, 2^{-2\nn_0+\frac12}$,
$\hat c$ is given in (\ref{chat}),
and $F$ is given in Table \ref{t2}.  We also list the matrices
$\De_0$ and $\tilde A^\flat=\tilde A z^{\frac N4 E^T} U^T$  appearing in the formula 
\[
E_1=
(P\tilde K)^{-T}
\mathcal E \mathcal F
(P\tilde K)^{T}
E_1^{\text{{global}}},
\quad
\mathcal E \mathcal F=
-(\tilde A^\flat)^{-T} F^{-1} \De_1 F (\tilde A^\flat)^T \De_0
\]
of Theorem \ref{Rexplicite}.    As explained in Note $2$ following Theorem \ref{Riwasawa}, we restrict to the cases (E1),(E3),(V1),(V2).
To save space we use the abbreviation 
\[
\eo=\log \om=\tfrac12{\i \pi}
\]
here.

\begin{center}
{\em Formulae for the case (E1)}
\end{center}

\newcommand{\sph}{\tfrac{1}{2}}

$
\tilde K=
\bsp
\vphantom\sph
\om^{2\nn_0} & \om^{-3}2\eo & \om^{-3} & \om^{-2\nn_0-6}
\\
\vphantom\sph
\om^{\nn_0} & \om^{-\frac32}\eo & \om^{-\frac32} & \om^{-\nn_0-3}
\\
\vphantom\sph
1 & 0 & 1 & 1
\\
\vphantom\sph
\om^{-\nn_0} & \om^{\frac32}(-\eo) & \om^{\frac32} & \om^{\nn_0+3}
\esp
$
\quad
$
U=\bsp
\vphantom\sph
1 & & &
\\
\vphantom\sph
 & \ 1\  & -\frac2a &
\\
\vphantom\sph
 & & 1 &
\\
\vphantom\sph
 & & & 1
\esp
$

$
\tilde A=
\bsp
\vphantom\sph
\tilde A_0 & & & 
\\
\vphantom\sph
 & \tilde A_2 & & 
\\
\vphantom\sph
 & \tilde A_1  & \tilde A_2   & 
 \\
 \vphantom\sph
 & & & \tilde A_3 
\esp
$
\quad
$
\tilde A^\flat=
\bsp
\vphantom\sph
\tilde A_0 & & & 
\\
\vphantom\sph
 & \tilde A_2 & & 
\\
\vphantom\sph
 & \tilde A^\flat_1  & \tilde A_2   & 
 \\
 \vphantom\sph
 & & & \tilde A_3 
\esp
$
\quad
$
\De_0=
\bsp 
\vphantom\sph
 & & & \ 1\\
 \vphantom\sph
  & -1 &  & \\
  \vphantom\sph
   &  & \ 1\  & \\
   \vphantom\sph
   1 & & &
\esp
$

$\tilde A_0 =2\pi\ii \ G(0)$

$\tilde A_1 =-2\pi\ii \ 2^{-a} t^{\frac a2} F(\tfrac a8) \left(
2\euler  +  (\log F)^\pr(\tfrac a8) +    \log 2^{-8}t^4  
\right)$

$\tilde A_2 =2\pi\ii \ 2^{2-a} t^{\frac a2} F(\tfrac a8)$

$\tilde A_3 = 2\pi\ii \ 2^{-2a} t^a \ H(\tfrac a4)$

\smallskip

$\tilde A^\flat_1 =-2\pi\ii \ 2^{-a} t^{\frac a2} F(\tfrac a8) \left(
2\euler  + \tfrac 8a + (\log F)^\pr(\tfrac a8) +    \log \tfrac{c}{N^4} 
\right)$

\no where 
$F(t)=\Ga(-t)\Ga(\tfrac a4-t)$, 
$G(t)=\Ga(\tfrac {a}8 - t)^2\Ga(\tfrac {a}4 - t)$,
$H(t)=\Ga(-t)\Ga(\tfrac {a}8 - t)^2$.

\begin{center}
{\em Formulae for the case (E3)}
\end{center}

$
\tilde K=
\bsp
\vphantom\sph
\om^{2\nn_0}2\eo & \om^{2\nn_0} & \om^{-2\nn_0-6}2\eo & \om^{-2\nn_0-6}
\\
\vphantom\sph
\om^{\nn_0}\eo & \om^{\nn_0} & \om^{-\nn_0-3}\eo & \om^{-\nn_0-3}
\\
\vphantom\sph
0 & 1 & 0 & 1
\\
\vphantom\sph
\om^{-\nn_0}(-\eo) & \om^{-\nn_0} & \om^{\nn_0+3}(-\eo) & \om^{\nn_0+3}
\esp
$
\quad
$
U=\bsp
\vphantom\sph
1\  & & &
\\
\vphantom\sph
 & 1 &  &
\\
\vphantom\sph
 & & \ 1\  & -\frac2a
\\
\vphantom\sph
 & & & 1
\esp
$

$
\tilde A=
\bsp
\vphantom\sph
\tilde A_1 & & & 
\\
\vphantom\sph
\tilde A_0 & \tilde A_1 & & 
\\
\vphantom\sph
 & & \tilde A_3  & 
 \\
 \vphantom\sph
 & & \tilde A_2  & \tilde A_3
\esp
$
\quad
$
\tilde A^\flat=
\bsp
\vphantom\sph
\tilde A_1 & & & 
\\
\vphantom\sph
\tilde A^\flat_0 & \tilde A_1 & & 
\\
\vphantom\sph
 & & \tilde A_3  & 
 \\
 \vphantom\sph
 & & \tilde A^\flat_2  & \tilde A_3
\esp
$
\quad
$
\De_0=
\bsp 
\vphantom\sph
 & & -1& \\
 \vphantom\sph
  & &  &\ 1 \\
  \vphantom\sph
 -1  &  & & \\
 \vphantom\sph
    & \ 1 & &
\esp
$

$\tilde A_0 =-2\pi\ii  
P(0)  \left(
2\euler  + (\log P)^\pr(0) +   \log 2^{-8}t^4  
\right)$

$\tilde A_1 = 2\pi\ii \, 2^2 P(0)$

$\tilde A_2 = -2\pi\ii \, 2^{-2a} t^{a} 
Q(\tfrac a4)\left(
2\euler  + (\log Q)^\pr(\tfrac a4) +   \log 2^{-8}t^4  
\right)$

$\tilde A_3 = 2\pi\ii \, 2^{2-2a} t^a Q(\tfrac a4)$

\smallskip

$\tilde A^\flat_0 =-2\pi\ii  
P(0)  \left(
2\euler  + (\log P)^\pr(0) +    \log \tfrac{c}{N^4}
\right)$

$\tilde A^\flat_2 = -2\pi\ii \, 2^{-2a} t^{a} 
Q(\tfrac a4)\left(
2\euler  + \tfrac 8a + (\log Q)^\pr(\tfrac a4) +    \log \tfrac{c}{N^4} 
\right)$

\no where $P(t)=\Ga(\tfrac a4-t)^2$, $Q(t)=\Ga(-t)^2$.

\begin{center}
{\em Formulae for the case (V1)} 
\end{center}

$
\tilde K= 
\bsp
\vphantom\sph
\om^{-3} \frac{1}{3!}  {(2\eo)^3}
& \om^{-3} \frac{1}{2!}   {(2\eo)^2}
& \om^{-3} \, 2\eo  &\om^{-3}
\\
\vphantom\sph
\om^{-\frac32} \frac{1}{3!}  {\eo^3}
& \om^{-\frac32} \frac{1}{2!}  {\eo^2} & \om^{-\frac32} \,\eo  &\om^{-\frac32}
\\
\vphantom\sph
0 & 0 & 0 & 1
\\
\vphantom\sph
\om^{\frac32} \frac{1}{3!}  {(-\eo)^3} & \om^{\frac32} \frac{1}{2!}  {(-\eo)^2} & \om^{\frac32} (-\eo)  &\om^{\frac32}
\esp
$
\quad
$U=I$

$
\tilde A=
\bsp
\vphantom\sph
\tilde A_3 & & & 
\\
\vphantom\sph
\tilde A_2 & \tilde A_3 & & 
\\
\vphantom\sph
\tilde A_1 & \tilde A_2 & \tilde A_3 &
 \\
 \vphantom\sph
\tilde A_0 & \tilde A_1 & \tilde A_2 & \tilde A_3 
\esp
$
\quad
$
\tilde A^\flat=
\bsp
\vphantom\sph
\tilde A_3 & & & 
\\
\vphantom\sph
\tilde A^\flat_2 & \tilde A_3 & & 
\\
\vphantom\sph
\tilde A^\flat_1 & \tilde A^\flat_2 & \tilde A_3 &
 \\
 \vphantom\sph
\tilde A^\flat_0 & \tilde A^\flat_1 & \tilde A^\flat_2 & \tilde A_3 
\esp
$
\quad
$\De_0=
\bsp 
\vphantom\sph
1 & & & \\
\vphantom\sph  
& -1 & & \\
\vphantom\sph 
& & \ 1 & \\
\vphantom\sph 
& & & -1\esp
$

$\tilde A_0 =
-2\pi\ii
(N_0 + M_0 \ell + \tfrac12 L_0  \ell^2 + \tfrac16 R_0  \ell^3)$

$\tilde A_1 = 
2\pi\ii (4M_0 + 4L_0 \ell  + 2R_0 \ell^2)$

$\tilde A_2 = 
-2\pi\ii
(16L_0 + 16R_0 \ell)$

$\tilde A_3 = 
2\pi\ii (64 R_0)$

\smallskip

$\tilde A^\flat_0 =-2\pi\ii
(N_0 + M_0 \ell_0 + \tfrac12 L_0  \ell_0^2 + \tfrac16 R_0  \ell_0^3)$

$\tilde A^\flat_1 =2\pi\ii (4M_0 + 4L_0 \ell_0  + 2R_0 \ell_0^2)$

$\tilde A^\flat_2 =-2\pi\ii
(16L_0 + 16R_0 \ell_0)$

\no where $\ell=\log 2^{-8}t^4$ and $\ell_0=\log\frac{c}{N^4}$.
The constants $N_0,M_0,L_0,R_0$ are defined by
\[
\Ga(-t)^4=R_0t^{-4}+L_0t^{-3}+M_0t^{-2}+N_0t^{-1}+ O(1).
\]
From Lemma \ref{laurent}, their values are:
$R_0=1$,
$L_0=4\euler$,
$M_0=8\euler^2 +\tfrac13\pi^2$,
$N_0=\tfrac{32}3 \euler^3 +\tfrac43\euler\pi^2 +\tfrac43\zeta(3)$.

\begin{center}
{\em Formulae for the case (V2)}
\end{center}

$
\tilde K=
\bsp
\vphantom\sph
\om^{1} & \om^{-3}2\eo & \om^{-3} & \om^{-7}2\eo
\\
\vphantom\sph
\om^{\frac12} & \om^{-\frac32}\eo & \om^{-\frac32} & \om^{-\frac72}\eo
\\
\vphantom\sph
1 & 0 & 1 & 0 
\\
\vphantom\sph
\om^{-\frac12} & \om^{\frac32}(-\eo) & \om^{\frac32} & \om^{\frac72}(-\eo)
\esp
$
\quad
$
U=\bsp
\vphantom\sph
1 & &  &
\\
\vphantom\sph
 & \ 1\  & -\frac12 &
\\
\vphantom\sph
 & & 1 &
\\
\vphantom\sph
-\frac54 & & & 1
\esp
$

$
\tilde A=
\bsp
\vphantom\sph
\tilde A_0 & & & \tilde A_3
\\
\vphantom\sph
 & \tilde A_2 & & 
\\
\vphantom\sph
 &   \tilde A_1 & \tilde A_2   & 
 \\
 \vphantom\sph
 & & & \tilde A_0
\esp
$
\quad
$
\tilde A^\flat=
\bsp
\vphantom\sph
\tilde A_0 & & & \tilde A^\flat_3
\\
\vphantom\sph
 & \tilde A_2 & & 
\\
\vphantom\sph
 &   \tilde A^\flat_1 & \tilde A_2   & 
 \\
 \vphantom\sph
 & & & \tilde A_0
\esp
$
\quad
$\De_0=
\bsp 
\vphantom\sph
1 & & & \\
\vphantom\sph
  & -1 & & \\
  \vphantom\sph
   & & \ 1 & \\
   \vphantom\sph
    & & & -1\esp
$

$\tilde A_0 = -2\pi\ii 2^{-6} t^4 S(1)$

$\tilde A_1 = -2\pi\ii 2^{-4} t^2
T(\tfrac12)\left(
2\euler  + (\log T)^\pr(\tfrac12) +   \log 2^{-8}t^4  
\right)$

$\tilde A_2 = 2\pi\ii 2^{-2}t^2T(\tfrac12)$

$\tilde A_3 = 2\pi\ii 2^{-8} t^4
S(1)\left(
2\euler-1  + (\log S)^\pr(1) +   \log 2^{-8}t^4  
\right)$

\smallskip

$\tilde A^\flat_1 = -2\pi\ii 2^{-4} t^2
T(\tfrac12)\left(
2\euler  + 2 + (\log T)^\pr(\tfrac12) +   \log \tfrac{c}{N^4} 
\right)$

$\tilde A^\flat_3 = 2\pi\ii 2^{-8} t^4
S(1)\left(
2\euler + 4 + (\log S)^\pr(1) +   \log \tfrac{c}{N^4}  
\right)$

\no where $T(t)=-t\Ga(-t)^2$, $S(t)=\Ga(\tfrac12-t)^2$.

{\em

\noindent
Department of Mathematics\newline
Faculty of Science and Engineering\newline
Waseda University\newline
3-4-1 Okubo, Shinjuku, Tokyo 169-8555\newline
JAPAN

\noindent
Department of Mathematical Sciences\newline
Indiana University-Purdue University, Indianapolis\newline
402 N. Blackford St.\newline
Indianapolis, IN 46202-3267\newline
USA
   
\noindent
Taida Institute for Mathematical Sciences\newline
Center for Advanced Study in Theoretical Sciences  \newline
National Taiwan University \newline
Taipei 10617\newline
TAIWAN
}

\end{document}